\documentclass{amsart}
\usepackage{euscript}
\usepackage{amssymb,amsmath}
\usepackage{tikz}
\usetikzlibrary{calc}
\usepackage{amscd}
\usepackage[all]{xy}
\allowdisplaybreaks
\def\Euc{{\rm Euc}}
\def\semidirect{\rtimes}

\def\height{\mathop{\hbox{\rm ht}}\nolimits}
\def\Suz{\hbox{\it Suz}}
\def\PSL{{\rm PSL}}
\def\SL{{\rm SL}}

\let\iso\cong

\def\Re{\mathop{\rm Re}\nolimits}
\def\Aut{\mathop{\hbox{\rm Aut}}\nolimits}

\def\homotopic{\simeq}

\def\O{{\rm O}}
\def\SO{{\rm SO}}

\def\reverse{\mathop{\hbox{\rm reverse}}\nolimits}

\makeatletter
\def\dodgelinewidth{.45pt}
\def\dodgeradius{2pt}
\def\spaceabovedodgeline{1pt}
\def\spacebelowdodgeline{1.2pt} 
\def\dodge#1{\mathop{\vbox{\m@th\ialign{##\crcr\noalign{\kern\spaceabovedodgeline}\dodgefill\crcr\noalign{\kern\spacebelowdodgeline\nointerlineskip}$\hfil\displaystyle{#1}\hfil$\crcr}}}\limits}
\def\dodgefill{$\m@th
\setbox0=\vbox to\dodgelinewidth{}
\leaders\vrule height\ht0 depth0pt\hfill\dodgebump\leaders\vrule height\ht0 depth0pt\hfill$}
\def\dodgebump{\begin{tikzpicture}%
    \useasboundingbox(-\dodgeradius,0)rectangle(\dodgeradius,\dodgeradius);%
    \draw[line width=\dodgelinewidth](\dodgeradius,0) arc (0:180:\dodgeradius);%
\end{tikzpicture}}
\makeatother

\makeatletter
\def\Looplinewidth{\dodgelinewidth}
\def\Loopradius{\dodgeradius}
\def\spaceabovemeridianline{\spaceabovedodgeline}
\def\spacebelowmeridianline{3pt}
\def\Loop#1{\mathop{\vbox{\m@th\ialign{##\crcr\noalign{\kern\spaceabovemeridianline}\Loopfill\crcr\noalign{\kern\spacebelowmeridianline\nointerlineskip}$\hfil\displaystyle{#1}\hfil$\crcr}}}\limits}
\def\Loopfill{$\m@th
\setbox0=\vbox to\Looplinewidth{}
\leaders\vrule height\ht0 depth0pt\hfill\Looploop$}
\def\Looploop{\begin{tikzpicture}%
    \useasboundingbox(-\Loopradius,0)--(\Loopradius,0);%
    \draw[line width=\dodgelinewidth](0,0)circle(\Loopradius);%
\end{tikzpicture}}
\makeatother

\makeatletter
\def\geodesiclinewidth{\dodgelinewidth}

\def\spaceabovegeodesicline{\spaceabovedodgeline}
\def\spacebelowgeodesicline{1pt}
\def\geodesic#1{\mathop{\vbox{\m@th\ialign{##\crcr\noalign{\kern\spaceabovegeodesicline}\geodesicfill\crcr\noalign{\kern\spacebelowgeodesicline\nointerlineskip}$\hfil\displaystyle{#1}\hfil$\crcr}}}\limits}
\def\geodesicfill{$\m@th
\setbox0=\vbox to\geodesiclinewidth{}
\leaders\vrule height\ht0 depth0pt\hfill\geodesicarrow$}
\def\geodesicarrow{\vrule height\geodesiclinewidth width1pt\kern-1pt
\smash[tb]{\lower2.29pt\hbox{$\rightharpoonup$}}}
\makeatother

\makeatletter
\def\hyper@x#1,#2\relax{#1}
\def\hyper@y#1,#2\relax{#2}
\def\hyper@coords#1{#1}
\newif\ifhyper@vertical
\def\hyper@computer#1#2{%
  \edef\hyper@toscan{(#1)}
  \tikz@scan@one@point\hyper@coords\hyper@toscan
  \edef\hyper@sx{\the\pgf@x}
  \edef\hyper@sy{\the\pgf@y}
  \edef\hyper@toscan{(#2)}
  \tikz@scan@one@point\hyper@coords\hyper@toscan
  \edef\hyper@ex{\the\pgf@x}
  \edef\hyper@ey{\the\pgf@y}
  \pgfmathsetmacro{\hyper@mx}{(\hyper@ex + \hyper@sx)/2}
  \pgfmathsetmacro{\hyper@my}{(\hyper@ey + \hyper@sy)/2}
  \pgfmathsetmacro{\hyper@dx}{\hyper@ex - \hyper@sx}
  \pgfmathparse{\hyper@dx == 0 ? "\noexpand\hyper@verticaltrue" : "\noexpand\hyper@verticalfalse"}
  \pgfmathresult
  \ifhyper@vertical
  \edef\hyper@cmd{-- (\tikztotarget)}
  \else
  \pgfmathsetmacro{\hyper@dy}{\hyper@ey - \hyper@sy}
  \pgfmathsetmacro{\hyper@t}{\hyper@my/\hyper@dx}
  \pgfmathsetmacro{\hyper@cx}{\hyper@mx + \hyper@t * \hyper@dy}
  \pgfmathsetmacro{\hyper@radius}{veclen(\hyper@cx - \hyper@sx, \hyper@sy)}
  \pgfmathsetmacro{\hyper@sangle}{180 - atan2(\hyper@sy,\hyper@cx-\hyper@sx)}
  \pgfmathsetmacro{\hyper@eangle}{180 - atan2(\hyper@ey,\hyper@cx-\hyper@ex)}
  \edef\hyper@cmd{arc[radius=\hyper@radius pt, start angle=\hyper@sangle, end angle=\hyper@eangle]}
  \fi
}
\tikzset{%
  hyperbolic plane/.style={
    to path={
      \pgfextra{\hyper@computer\tikztostart\tikztotarget}
      \hyper@cmd
    }
  },
}
\makeatother

\let\geodesic\overline

\def\complexgeodesic#1{\geodesic{#1}{}^{\C}}

\def\cong{\equiv}
\def\orb{{\rm\scriptstyle orb}}
\def\piorb{\pi_1^\orb}
\def\Im{\mathop{\rm Im}\nolimits}
\def\Isom{\mathop{\rm Isom}\nolimits}

\def\PGL{\mathop{\rm PGL}\nolimits}
\def\Leech{\Lambda}
\def\w{\omega}

\def\ybar{\bar{y}}
\def\zbar{\bar{z}}

\def\wbar{\bar{\w}}
\def\vbar{\bar{v}}

\def\E{\EuScript{E}}
\def\Z{\mathbb{Z}}
\def\Q{\mathbb{Q}}
\def\F{\mathbb{F}}
\def\R{\mathbb{R}}
\def\C{\mathbb{C}}

\def\ubar{\bar{u}}
\def\e{\varepsilon}
\def\H{\EuScript{H}}
\def\tensor{\otimes}
\def\sset{\subseteq}
\def\G{\Gamma}
\def\PG{P\G}
\def\thetabar{\bar{\theta}}
\def\psibar{\bar{\psi}}
\def\mbar{\bar{m}}

\def\ip#1#2{\langle#1\,{|}\,#2\rangle}
\def\set#1#2{\{#1\,{|}\,#2\}}
\def\Bigset#1#2{\Bigl\{#1\Bigm|#2\Bigr\}}
\def\spanof#1{\langle#1\rangle}
\def\gend#1{\langle#1\rangle}

\def\bigip#1#2{\bigl\langle#1\bigm|#2\bigr\rangle}
\def\Bigip#1#2{\Bigl\langle#1\Bigm|#2\Bigr\rangle}

\newcommand{\abs}[1]{\lvert #1 \rvert}
\newcommand{\BB}{\mathbb{B}}

\newcommand{\op}[1]{\operatorname{#1}}

\newcommand{\smallmat}[4]{\bigl( \begin{smallmatrix} #1 & #2  \\ #3 & #4 \end{smallmatrix} \bigr)}
\newtheorem{theorem}{Theorem}[section]
\newtheorem{lemma}[theorem]{Lemma}

\newtheorem{conjecture}[theorem]{Conjecture}

\theoremstyle{remark}
\newtheorem{numberedremark}[theorem]{Remark}
\newtheorem*{remark}{Remark}
\newtheorem*{remarks}{Remarks}
\newtheorem*{correction}{Correction to \cite{Allcock-Y555}}
\numberwithin{figure}{section}
\numberwithin{equation}{section}
\begin{document}
\title{Generators for a complex hyperbolic braid group}
\author{Daniel Allcock}
\thanks{First author supported by NSF grant DMS-1101566 and a Simons
  Foundation Collaboration Grant}
\address{Department of Mathematics\\University of Texas, Austin}
\email{allcock@math.utexas.edu}
\urladdr{http://www.math.utexas.edu/\textasciitilde allcock}
\author{Tathagata Basak}
\address{Department of Mathematics, Iowa State University, Ames IA, 50011.}
\email{tathastu@gmail.com}
\urladdr{https://orion.math.iastate.edu/tathagat/}

\date{September 19, 2016}

\begin{abstract}
  We give generators for a certain complex hyperbolic braid group.
  That is, we remove a hyperplane arrangement from complex hyperbolic
  $13$-space, take the quotient of the remaining space by a discrete
  group, and find generators for the orbifold fundamental group of the
  quotient.  These generators have the most natural form: loops
  corresponding to the hyperplanes which come nearest the basepoint.
  Our results support the conjecture that motivated this study,  the
  ``monstrous proposal'', which posits a relationship between this
  braid group and the monster finite simple group.
\end{abstract}

\maketitle
\setcounter{tocdepth}{1}
\tableofcontents


\section{Introduction}
\label{sec-introduction}

\noindent
We are interested in finding generators and relations for complex hyperbolic braid
groups.  Complex hyperbolic space
$\BB^n$ is a complex manifold, diffeomorphic to the unit ball in
$\C^n$, and closely analogous to real hyperbolic space.  Our braid
groups, which we called braid-like in \cite{AB-braidlike}, arise by removing a
locally finite arrangement of complex hyperplanes, 
quotienting the remaining space by the action of a discrete group, and
then taking the fundamental group.
This is closely analogous to a well-known construction of the
classical braid group \cite{Fox-Neuwirth}: remove the hyperplanes $x_i=x_j$ from
$\C^n$, and then quotient by the symmetric group~$S_n$.  A similar
construction, with $S_n$ replaced by any other Coxeter group, yields
the Artin groups \cite{Brieskorn}\cite{van-der-Lek}.  In a different direction, replacing
$S_n$ by a finite complex reflection group leads to the complex braid
groups, whose presentations and key properties were worked out only
recently \cite{Bessis}.  In a sense we are generalizing the classical braid
group in both these ways simultaneously.

Our main result, theorem~\ref{t-26-meridians-based-at-tau-generate}, gives a
natural set of generators for
a particular complex hyperbolic braid group coming from~$\BB^{13}$.
Finding enough relations that  give a presentation of the fundamental group appears to be rather harder.  We will explain why
this example is interesting, but to set the scene we begin with our
broader motivations.  In a sense these are detours, since we do not
develop them further in this paper.

The first motivation involves singularity theory.  The Artin groups of
types $A_n$, $D_n$ and $E_n$ appear in nature as the fundamental
groups of the discriminant complements of the simple singularities,
which are also called $A_n$, $D_n$ and $E_n$.  Under a technical
assumption that often holds, this has the following consequence.
Suppose given a family of complex varieties over a base variety~$B$,
write $B_0\sset B$ for the set of smooth fibers, and suppose the fiber
over $b\in B$ has some simple singularities but no worse
singularities.  
Then these singularities have types $A_n$, $D_n$ and $E_n$ for some 
choices of subscripts, and there is a neighborhood $U$ of $b$ such that
$\pi_1(B_0 \cap U)$ is the direct product of the corresponding
Artin groups.
Informally:\ $A_n$, $D_n$ and $E_n$
Artin groups appear ``locally'' in the fundamental groups of ``most''
families of algebraic varieties.  For example, by considering families
of Riemann surfaces, one automatically expects suitable elements of
mapping class groups to satisfy the braid relation---an expectation
fulfilled by Dehn twists around curves that meet just once
(and transversely).  See
\cite{Libgober}, \cite{Looijenga-Artin-groups},
(3.5)--(3.7) of \cite{ACT-surfaces}, and lemma~1.5 and theorem~7.1 of
\cite{ACT-threefolds}
for examples and applications of these ideas; the needed technical
assumption is that the family $B$ should provide a simultaneous versal
deformation of all the singularities.

Now, singularities well beyond the simple ones have been classified,
and the next-least complicated ones are the ``affine'' singularities
$\tilde A_n$, $\tilde D_n$, $\tilde E_n$.  From the notation one
naturally expects that the fundamental groups of their discriminant
complements are the corresponding affine Artin groups.  But this is
not so \cite{van-der-Lek}.  And more-singular singularities have
discriminant complements whose fundamental groups are even further
from the Artin groups.  For example, \cite[Thm.~4.3]{Lonne} gives
presentations in the case of the Brieskorn-Pham singularities
$x_1^{d_1}+\cdots+x_n^{d_n}=0$, where each $d_i$ is at least~$2$.
These presentations include Artin relations but also additional
relations.  But in many cases, for example
\cite{Looijenga-triangle-singularities}\cite[\S10]{Looijenga-compactifications-II}\cite{Laza},
these fundamental groups are braid groups
in our sense, with $\BB^n$ replaced by the symmetric space
$\SO(2,n)/S\bigl(\O(2)\times\O(n)\big)$.  The reasoning we used for
simple singularities generalizes to any singularity, so these new
braid groups appear automatically when considering families of complex
varieties.  In particular, they appear in nature in a way that the
infinite-type Artin groups do not.  Brieskorn
\cite{Brieskorn-question} has asked for presentations for more
discriminant complements, and we hope our methods will contribute to
this.  We regard braid groups coming from hyperplane arrangements in
$\BB^n$ as an easier analogue of the
$\SO(2,n)/S\bigl(\O(2)\times\O(n)\big)$ case, hence a test bed for our
ideas.

Our second motivation is the braid groups of the finite complex
reflection groups.  The known presentations \cite{Bessis} are obtained
from Lefschetz pencils, rather than directly from the arrangement of
the hyperplanes.  In the case of Artin groups, the standard basepoint
lies in the interior of the Weyl chamber, and the standard generators
are the following paths, called ``meridians''.  Each starts at the
basepoint, moves directly toward a facet $F$ of the chamber until
close to it, then travels along a semicircle around the
complexification of $F$, and then moves directly to the image of the
basepoint under the reflection across~$F$.  Although this is not a
loop, it  becomes one after quotienting by the Weyl group.
In this way the standard generators correspond to the mirrors nearest
the basepoint.  We hope that there are analogous good generating sets
for the complex braid groups.  The Weyl chamber is not available in
this context, but there may still be natural basepoints (possibly
the ``Weyl vectors'' of \cite{Basak-Coxeter-diagrams}), and generators like those for Artin groups,
coming from the mirrors nearest the basepoint.

\smallskip
Now we discuss the braid group that is the subject of this paper.  It
comes from a group $P\Gamma$ acting on $\BB^{13}$ with finite
covolume, generated by triflections (complex reflections of
order~$3$).  The hyperplane arrangement $\H$ is the union of the
mirrors of the triflections.  In a sense this is the hardest example
available, because $n=13$ is the highest dimension for which there is a
known finite-covolume complex hyperbolic reflection group on~$\BB^n$.
Our main result, theorem~\ref{t-26-meridians-based-at-tau-generate},  concerns a particular
basepoint $\tau\in\BB^{13}$ and the meridians based there.  Meridians
are defined below, and more generally in section~\ref{sec-paths-and-homotopies}, but they are
similar to the Artin group case.  The
only difference is
that the circular-arc portion of the path is only one third of a
circle, not half, because the complex reflections have
order~$3$, not~$2$.  Our main result is that the braid group
$G_\tau=\piorb\bigl((\BB^{13}-\H)/P\Gamma,\tau\bigr)$ is generated by
the meridians associated to the $26$ mirrors nearest~$\tau$.  The
notation $\piorb$ indicates the orbifold fundamental group, which is
needed because $P\Gamma$ does not act freely on $\BB^{13}-\H$.
See section~\ref{subsec-orbifold-fundamental-group} for our conventions about orbifold fundamental
groups.

The motivation to study this particular example is the first author's
``monstrous proposal'', namely the following conjecture concerning the
sporadic finite simple group $M$ known as the monster.  See \cite{Allcock-monstrous}
for background.

\begin{conjecture}[Monstrous Proposal {\cite{Allcock-monstrous}}]
  \label{conj-monstrous-proposal}
  The quotient of $G_\tau$, by the subgroup $N$ normally generated by
  the squares of the meridians, is isomorphic the ``bimonster''
  $B=(M\times M)\semidirect\Z/2$, where $M$ is the monster finite
  simple group and $\Z/2$ acts by exchanging the
  factors in the obvious way.
\end{conjecture}

A known presentation of $B$ has $26$ generators of order~$2$,
corresponding to the points and lines of $P^2\F_3$.  Two of these
generators braid or commute according to whether the corresponding
points/lines are incident in the usual sense of projective geometry.
There is one additional relation; see \cite{Conway-Simons}.  The
amazing coincidence is that $G_\tau$'s $26$ generators may be indexed
by the points and lines of $P^2\F_3$ in the same way, and they satisfy
exactly the same commutation and braid relations \cite[Thm.~4.7]{Basak-bimonster-2}.  And
although it has not yet been verified, there are good geometric
grounds to expect that $G_\tau$'s generators also satisfy the
additional bimonster relation, or one equivalent to it modulo the
$(\hbox{meridian})^2=1$ relations.  
Assuming this, $G_\tau/N$ is a quotient of $B$, hence is isomorphic
to $B$ or $\Z/2$ or the trivial group. 
To prove conjecture \ref{conj-monstrous-proposal} one would 
need to rule out the cases $G_{\tau} = \Z/2$ or the trivial group.
So we regard
theorem~\ref{t-26-meridians-based-at-tau-generate} as significant
progress toward conjecture~\ref{conj-monstrous-proposal}.

\medskip
Now we develop just enough background to make precise the objects we
have discussed; for additional background see
section~\ref{sec-background-conventions-notation}.  We write $\E$ for
the ring $\mathbb{Z}[e^{2 \pi i/3}]$ of Eisenstein integers.  The
central character in the paper is a particular hermitian $\E$-lattice
$L$, namely the unique one of signature $(13,1)$ which equals
$\sqrt{-3}$ times its dual lattice.  A concrete model for $L$ is the
$\E$-span of the $13$ ``point-roots'' like $(0;\sqrt{-3},0,\dots,0)$
and the $13$ ``line-roots'' like $(1;1,1,1,1,0,\dots,0)$.  Here we are
using the standard hermitian form of signature $-+\cdots+$ on
$\C^{14}$, with the last $13$ coordinates indexed by the points of
$P^2\F_3$, and each line-root having coordinates $1$ at the four
points of a line of $P^2\F_3$.  The complex ball $\BB^{13}$ is the set
of complex lines of negative norm in $L \otimes_{\E} \C$.  A root
means a lattice vector of norm~$3$, for example a point- or line-root.
If $r$ is a root, then we write $R_r$ for the triflection in $r$,
meaning the isometry of $L$ which multiplies $r$ by $e^{2\pi i/3}$ and
fixes $r^\perp$ pointwise.  The mirror of this complex reflection
means the fixed point set in $\BB^{13}$.  The hyperplane arrangement
$\H$ is the union of the mirrors, and $P\Gamma$ is the subgroup of
$\Aut\BB^{13}$ generated by the triflections (which is the full
projective isometry group of~$L$).

The mirrors of the $13$ point-roots meet orthogonally at a point of
$\BB^{13}$, and similarly for the $13$ line-roots.  The basepoint
$\tau$ is the midpoint of the segment joining these two points, and
the mirrors nearest $\tau$ are exactly these $26$ mirrors.  For any
point- or line-root $r$, the corresponding meridian in
$G_\tau=\piorb\bigl((\BB^{13}-\H)/P\Gamma,\tau\bigr)$ is represented
by the following three-part path.  Let $p$ be the projection of $\tau$
to the mirror $r^\perp$, and let $U$ be an open ball centered at $p$,
small enough so that the only mirror it meets is $r^\perp$.  Let $d$
be a point of the geodesic $\geodesic{\tau p}$, lying in $U$ and
different from~$p$.  The first part of the meridian is the geodesic
$\geodesic{\tau d}$.  The second part is the circular arc from $d$ to
$R_r(d)$, centered at $p$ and positively oriented in the complex
geodesic containing $\geodesic{\tau p}$.   The third
part of the meridian is the geodesic $\geodesic{R_r(d)R_r(\tau)}$.  We
call these elements of~$G_\tau$ the point- and line-meridians.  Now
our main theorem has precise meaning:

\begin{theorem}[Main theorem]
\label{t-26-meridians-based-at-tau-generate}
The $13$ point-meridians and $13$ line-meridians generate the orbifold
fundamental group of $(\BB^{13} - \H)/P \Gamma$, based at $\tau$.
\end{theorem}

We announced this in \cite{AB-braidlike}, and our starting point for the proof is
theorem~1.5 of that paper.  That result gives a specific infinite
generating set for $G_\rho$, which is this same orbifold fundamental
group, but based at a cusp $\rho\in\partial\BB^{13}$ of $P\Gamma$.
This generating set consists of the meridians based at $\rho$ and
corresponding to the (infinitely many) mirrors which come closest to
$\rho$.  (See section~\ref{sec-paths-and-homotopies} for what we mean by taking the basepoint
at a cusp, by the meridians based there, and what it means for a
mirror to come closest to $\rho$.  There are no surprises, but some
care is needed.)

Given this starting point, our first step is to exhibit a finite
subset of these meridians, which is a generating set.  This occupies
section~\ref{sec-finitely-many-generators-based-at-a-cusp}, and the key argument concerns generators for the
$P\Gamma$-stabilizer of~$\rho$.  Our second step is show that the
meridians corresponding to the point- and line-roots, but based at
$\rho$ rather than $\tau$, are also a generating set.  This is
section~\ref{sec-26-generators-based-at-a-cusp}, and the method is to show that the subgroup of
$G_\rho$ they generate contains all the generators from
section~\ref{sec-finitely-many-generators-based-at-a-cusp}.  Finally, in section~\ref{sec-change-of-basepoint} we show that moving the
basepoint from $\rho$ to $\tau$, along the geodesic
$\geodesic{\rho\tau}$, identifies the point- and line-meridians based
at $\rho$ with those based at $\tau$ in the obvious way.
This implies theorem~\ref{t-26-meridians-based-at-tau-generate}.
At heart,
all of our arguments involve concrete homotopies between various paths
in $\BB^{13}-\H$.  Besides setting up our general definition of
meridians, section~\ref{sec-paths-and-homotopies} contains several theorems saying that such
homotopies exist, provided that certain totally geodesic triangles in
$\BB^{13}$ miss~$\H$.  The paper rests on the verification of this
property for a total of $10$ triangles, in the appendices.

Part of
this verification relies on computer calculation; we also used the
computer to verify the paper's many hand-calculations involving
vectors in~$\C^{14}$.  These calculations are involved enough that
a reader skimming over them might 
imagine that the main theorem is a numerical
accident.  In fact, behind most of these calculations lurk special properties of the Leech lattice,
such as inequalities that are exactly what is needed to complete a proof.
So we will emphasize these properties when they arise; they lend the 
calculations a certain sense of inevitability.

The authors are grateful to the RIMS (Kyoto U.) for its hospitality during
part of this work.

\section{Background, conventions, notation}
\label{sec-background-conventions-notation}

\subsection{Eisenstein lattices}
\label{t-def-eisenstein-lattice}
Let $\omega = e^{2 \pi i/3}$ and $\theta = \omega - \wbar =
\sqrt{-3}$. Let $\E$ be the ring $\Z[\omega]$ of Eisenstein integers.
An Eisenstein lattice $K$ means an hermitian $\E$-lattice, i.e., a
free $\E$-module with an hermitian form $\ip{\;}{\;} : K \times K \to
\Q(\w)$, linear in the first variable and antilinear in the second.
We abbreviate $K\tensor_\E\C$ to $K\tensor\C$.  If $K$ is
nondegenerate then its dual lattice is defined as $K^* = \lbrace
x \in K \otimes \C \colon \ip{x}{k} \in \E \text{\; for all \;} k \in
K \rbrace$.  The norm $v^2$ of $v \in K$ means $\ip{v}{v}$.  If $X$ is
a subset of a lattice then we write $X^\perp$ for the
set of lattice vectors orthogonal to it.  If $x,y,\dots$ lie in an
$\E$-lattice, then $\spanof{x,y,\dots}$ means their $\E$-span.

In the appendices, starting with lemma~\ref{lem-the-center}, the Eisenstein integer
$\psi=1-3\wbar$ plays an important role.

\subsection{Complex hyperbolic space}
\label{subsec-complex-hyperbolic-space}
We call an Eisenstein lattice $K$ Lorentzian if it has signature
$(n,1)$.  In that case we let $\BB(K)\sset P(K\tensor\C)$ denote the set of complex
lines of negative norm in $K \otimes\C$.  Topologically it is a
complex ball of dimension $n$.  It has a
natural metric called the Bergman metric, and is sometimes called
complex hyperbolic space. If it is clear what lattice we mean then we
sometimes write $\BB^n$ in place of $\BB(K)$.   In particular, $\BB^{13}$ will
always mean $\BB(L)$ for the lattice $L$ defined in section~\ref{t-P2-F3-model-of-L} below.  An inclusion
of Lorentzian lattices induces an inclusion of their complex
balls.  If this inclusion has codimension~$1$ then we
call the smaller ball a hyperplane.

Any vector $v$ of negative norm in $K\tensor\C$ determines a point $\C
v$ in $\BB(K)$. Often we use the same symbol for the point and the
vector.  The distance between two points of $\BB(K)$ is given by
\begin{equation}
\label{eq-complex-hyperbolic-metric}
  d(v, w) = \cosh^{-1} \sqrt{  \frac{\abs{   \ip{v}{w}}^2}{v^2\, w^2} } 
\end{equation}
where $v$ and $w$ are two negative norm vectors of $K\tensor\C$.  Similarly, if
$v,s\in K\tensor\C$ have negative and positive norm respectively, then  
\begin{equation}
\label{eq-distance-point-to-hyperplane}
  d\bigl(v, \BB(s^\perp)\bigr) = \sinh^{-1} \sqrt{ -\frac{\abs{   \ip{v}{s}}^2}{v^2\, s^2} } 
\end{equation}
When $s$ is a root (see section~\ref{subsec-roots-in-general} below) we usually write $s^\perp$ in place of
$\BB(s^\perp)$ when it is clear that we mean this rather than the
orthogonal complement in the lattice.  Formulas \eqref{eq-complex-hyperbolic-metric} and \eqref{eq-distance-point-to-hyperplane}
differ 
from those in
\cite{Goldman} by an unimportant factor of~$2$.

The boundary $\partial \BB(K)$ of $\BB(K)$ in $P(K\tensor\C)$ is a
real $(2n-1)$-sphere.  Its points are the projectivizations of null
vectors, meaning non-zero vectors in $K\tensor\C$ of norm~$0$.  Given
such a vector $\rho$, we define a sort of distance-to-$\rho$ function
on $\BB(K)$, called the height:
\begin{equation}
  \label{eq-definition-of-height}
  \height_\rho(v)=-\frac{|\ip{v}{\rho}|^2}{v^2}
\end{equation}
This is invariant under scaling $v$, so it descends to a function on
$\BB(K)$.  We will say that one point of $\BB(K)$ is closer to $\rho$
than another point is, if its value of $\height_\rho$ is smaller.
The horosphere centered at $\rho$, of height $h$, means the
set of $v\in\BB(K)$ with $\height_\rho(v)=h$.  We define open and
closed horoballs the same way, replacing $=$ by~$<$ and~$\leq$.
Scaling $\rho$ by $\lambda\in\C$  scales the function $\height_\rho$
by $|\lambda|^2$.  But for us
there will be a canonical normalization, because we will always take
$\rho$ to be a primitive lattice vector.  Under this condition, the
only allowed scaling of $\rho$ is by sixth roots of unity, which does
not affect the height function.  We call the point of $\partial\BB(K)$
represented by a primitive lattice vector a cusp, and for
convenience we also call that vector itself a cusp.
We write $\overline{\BB(K)}$ for the topological closure
$\BB(K)\cup\partial\BB(K)$ of $\BB(K)$ in $P(K\tensor\C)$.

\subsection{Geodesics and totally geodesic triangles}
\label{subsec-geodesics-and-geodesic-triangles}
Continuing to take $K$ as above, suppose $v,w\in K\tensor\C$ represent
points of $\overline{\BB(K)}$.  Then the geodesic joining them may be
described as follows.  First we rescale one or both of $v,w$ so that
$\ip{v}{w}\in(-\infty,0]$; this inner product will be negative
unless $v$ and $w$ both represent the same point of $\partial\BB(K)$.  After this rescaling, the
geodesic is the image in projective space of the real line segment in
$K\tensor\C$ joining the vectors.  If we have in mind a third point of
$\overline{\BB(K)}$, represented by $x\in K\tensor\C$, then there may
or may not be a totally geodesic surface in $\overline{\BB(K)}$ whose closure
contains all three points.  If there is one then we write $\triangle v
w x$ for their
convex hull.
There are two cases in which we use this notation.  One is when
$v,w,x$ all lie in some complex $2$-space in $K\tensor\C$.  Then they
lie in (the closure of) some $\BB^1$.  In this case we call
$\triangle v w x$ a complex triangle.

The other case occurs when it is possible to scale $v,w,x$ so that
their pairwise inner products lie in $(-\infty,0]$.  Suppose this has
  been accomplished.  The image in $\BB(K)$ of the set of
  negative-norm vectors in the real span of $v,w,x$ is a totally
  geodesic copy of the real hyperbolic plane.  In this case the
  projectivization of the convex hull of $v$, $w$ and $x$ in
  $K\tensor\C$ is the convex hull of their images in
  $\overline{\BB(K)}$.  In this case we call $\triangle v w x$ a
  totally real triangle, and don't usually distinguish between the
  triangle in $K\tensor\C$ and its image in $\overline{\BB(K)}$.

    For us this situation occurs as follows.  Let $b$ be a point of
    $\overline{\BB(K)}$, $H$ be a hyperplane in $\BB(K)$, $p$ be the
    point of $H$ closest to $b$, and $q$ be another point of $H$.
    Using the same letters for vectors representing these points, and
    $s$ for a positive-norm vector orthogonal to $H$, we first scale
    $b$ so that its inner product with $s$ is real, and then scale $q$
    so that that its inner product with $b$ lies in $(-\infty,0]$.
      Then $p=b-\ip{b}{s}s/s^2$ has inner product
      $b^2-\ip{b}{s}^2/s^2\leq0$ with $b$, and by $s\perp q$ we have
      $\ip{p}{q}=\ip{b}{q}\leq0$.  So $\triangle b p q$ is a totally
      real triangle.

\subsection{The $P^2\F_3$ model of the lattice $L$}
\label{t-P2-F3-model-of-L}
Now we set up an explicit model for the $\E$-lattice $L$ that governs
everything in this paper, called the $P^2\F_3$ model.
It was implicit in \cite{Basak-bimonster-1} (eq. 25 in proof of
prop. 6.1) and was defined explicitly in \cite{Allcock-Y555}.
Most of our
computations will be done in this coordinate system.  See section~\ref{subsec-Leech-model-of-L} for
an alternate model of $L$, the Leech model.

We write elements of $\C^{14}$ as vectors $x =
(x_0; x_1, \dotsb, x_{13} )$ and use the standard hermitian form
of signature $(13,1)$ on $\C^{14}$, namely
\begin{equation*}
\ip{x}{y}=-x_0 \bar{y}_0+x_1\bar{y}_1+\dots+x_{13}\bar{y}_{13}.
\end{equation*}
We index the last $13$ coordinates by the points of $P^2\F_3$.  $L$
consists of all vectors $x\in\E^{14}$ such that $x_0 \cong
x_1+\dots+x_{13}$ mod $\theta$ and that $(x_1,\dots,x_{13})$, modulo
$\theta$, is an element of the ``line code'', meaning the
$7$-dimensional subspace of $\F_3^{13}$ spanned by the characteristic
functions of the lines of $P^2\F_3$.  The elements of the line code
are tabulated in tables 2 and~3 of \cite{Allcock-Y555};  the explicit list is not
needed in this paper.  We write $\Gamma$ for $\Aut L$.  It contains
the simple group $L_3(3)=\PGL_3(\F_3)$, acting by permuting the
points of $P^2\F_3$ in the obvious way.  It also contains the group
$3^{13}=(\Z/3)^{13}$, acting by multiplying the last~$13$ coordinates by cube
roots of unity.

An important property of $L$ is that it equals $\theta\cdot L^*$.  The
proof amounts to checking that all inner products are divisible
by~$\theta$ and that $\det L=-3^7$.  The first part is easy, using
the point- and line-roots from section~\ref{subsec-point-roots-etc}.
The second part follows from the fact that $L/(\theta\E)^{14}\iso\F_3^7$.

\subsection{Roots, mirrors and the hyperplane arrangement $\H$}
\label{subsec-roots-in-general}
A root of $L$ means a lattice vector of norm~$3$.  What makes
roots special is that their ``triflections''
preserve~$L$.  That is, supposing
$s$ is a root, we define $R_s$
as the automorphism of $L$ that fixes $s^{\bot}$ pointwise and
multiplies $s$ by the cube root of unity $\omega$.  A formula is
\begin{equation}
  \label{eq-formula-for-omega-reflection}
R_s:  x\mapsto x-(1-\w)\frac{\ip{x}{s}}{s^2}s
\end{equation}
One can show that this isometry of $L\tensor\C$ preserves $L$.  (The
key  is that all inner products in $L$ are divisible by
$\theta$.)  It is called the
$\w$-reflection in~$s$.  Replacing $\w$ by $\wbar$ gives the
$\wbar$-reflection, which is also $R_s^{-1}$.  Both of these
isometries are complex reflections of order~$3$, sometimes called
triflections.  It is known that $\Gamma$ is generated by the
triflections in the roots of $L$; see \cite{Basak-bimonster-1} or \cite{Allcock-Y555}.

As the fixed-point set of the reflection $R_s$, the hyperplane
$\BB(s^\perp)\sset\BB(L)$ is called the mirror of~$s$.  In this paper
we use the word mirror exclusively for hyperplanes orthogonal to roots.
The union of all
mirrors is called $\H$.  This hyperplane arrangement is central
to the paper, since our goal is to study the orbifold fundamental
group of $\bigl(\BB^{13}-\H\bigr)/P\Gamma$.

\subsection{Point-roots, line-roots, $13$-points and $26$-points}
\label{subsec-point-roots-etc}
It is possible to number the points and lines of $P^2\F_3$ by the numbers
$1,\dots,13$, such that the $j$th line is the set of points
$\{j,j+1,j+3,j+9\}$.  Here the indices should be read modulo~$13$.  We
adopt this numbering for the following important roots of $L$.  First,
the $13$ point-roots $p_i$ are the vectors
$p_1=(0;\theta,0,\dots,0),\dots,p_{13}=(0;0,\dots,0,\theta)$.  And
second, for $j=1,\dots,13$, the line-root $l_j$ is the vector of the
form $(1;1^4,0^9)$ with $1$'s along the $j$th line of $P^2\F_3$.
Explicitly,
\begin{center}
\begin{tabular}{lll}
$l_{1}=(1;1101000001000)$&$l_{8\phantom{0}}=(1;0001000110100)$\\
$l_{2}=(1;0110100000100)$&$l_{9\phantom{0}}=(1;0000100011010)$\\
$l_{3}=(1;0011010000010)$&$l_{10}=(1;0000010001101)$\\
$l_{4}=(1;0001101000001)$&$l_{11}=(1;1000001000110)$\\
$l_{5}=(1;1000110100000)$&$l_{12}=(1;0100000100011)$\\
$l_{6}=(1;0100011010000)$&$l_{13}=(1;1010000010001)$\\
$l_{7}=(1;0010001101000)$
\end{tabular}
\end{center}
We speak of a point-root $p_i$ and a line-root $l_j$ as being incident
when the corresponding point and line of $P^2\F_3$ are.  For example,,
$\ip{p_i}{l_j}=\theta$ or $0$ according to whether or not $p_i$ and
$l_j$ are incident.
Also, distinct
point-roots are orthogonal, as are distinct line-roots.

There is some obvious symmetry preserving this configuration of roots:
$L_3(3)$ sends point-roots to point-roots and line-roots to
line-roots.  But there is additional symmetry.  From a correlation of
$P^2\F_3$, i.e., an incidence-preserving exchange of points with
lines, one can construct an isometry of $L$ that sends the point-roots
to line-roots and the line-roots to negated point-roots.  (We give
such an isometry explicitly in the proof of
lemma~\ref{lem-step-4-line-case}.)  Together with scalars and
$L_3(3)$, this generates a subgroup $\bigl(6\times L_3(3)\bigr)\cdot2$
of $\Gamma$, whose image in $P\Gamma$ is $L_3(3):2$.  Here we are
using ATLAS notation \cite{ATLAS}: a group has ``structure $A.B$'' if
it has a normal subgroup isomorphic to a group $A$, with quotient
isomorphic to a group $B$.  If the extension splits then one can
indicate this by writing $A:B$ instead.  If it does not split then one
can write $A\cdot B$.  We are also using another ATLAS notation:
writing $n$ to indicate a cyclic group of order~$n$.

We define the point- and line-mirrors to be the mirrors of the point-
and line-roots.  The $13$ point-mirrors are mutually orthogonal and
intersect at a single point of $\BB^{13}$, represented by
\begin{equation*}
p_\infty=(\thetabar;0^{13}).
\end{equation*}
For lack of a better name we call it
a $13$-point to indicate the $13$ mirrors passing through it.  (It is
easy to see that there are no mirrors through it except the
point-mirrors.)  We apply the same term to its $\Gamma$-translates,
such as the intersection point of the $13$ line-mirrors, represented by
\begin{equation*}
l_\infty=(4;1^{13}).
\end{equation*}
 We summarize the inner product information about
$p_\infty,p_1,\dots,p_{13},l_\infty,l_1,\dots,l_{13}$ as follows.
First, all have norm~$3$ except for $p_\infty$ and $l_\infty$, which
have norm~$-3$.  Second, if $i\neq j$ then $p_i\perp p_j$ and $l_i\perp
l_j$.  Third, $\ip{p_\infty}{l_\infty}=4\theta$.  Finally, if $i$ and
$j$ are not both $\infty$, then $\ip{p_i}{l_j}$ is $\theta$ or $0$
according to whether or not $p_i$ and $l_j$ are incident.  Here we are
regarding $p_\infty$ as ``incident'' to every $l_j$, and $l_\infty$ as
``incident'' to every $p_i$.

We define the basepoint $\tau$ used
in theorem~\ref{t-26-meridians-based-at-tau-generate} as the midpoint of
the segment joining $p_\infty$ and $l_\infty$.  It is represented by
the vector 
\begin{equation*}
\tau = l_{\infty} + i p_{\infty}=(4+\sqrt3;1^{13})
\end{equation*}
 of norm $-6-8\sqrt3$, and is the unique fixed point of $L_3(3){:}2\sset
P\Gamma$.  It is known that the mirrors closest to $\tau$ are exactly
the $26$ point- and line-mirrors.  (See
\cite[prop.~1.2]{Basak-bimonster-1}, where $\tau$ was called
$\bar{\rho}$, or lemma~\ref{lem-mirrors-near-the-26-point} in this paper.)
Two consequences of this are that no mirrors pass through $\tau$, and
that $L_3(3){:}2$ is the full $\PG$-stabilizer of $\tau$.  For lack of
a better name, we call $\tau$ a $26$-point to indicate these $26$
nearest mirrors.
We use the same language for its
$\Gamma$-translates.

\subsection{The Leech model of $L$}
\label{subsec-Leech-model-of-L}
By the Leech lattice we mean what might better be called the complex
Leech lattice $\Leech$.  It is a $12$-dimensional positive definite
$\E$-lattice, described in detail in \cite{Wilson}.  At the smallest
scale at which it is integral as an $\E$-lattice, it has minimal
norm~$6$ and satisfies $\Leech=\theta\cdot\Leech^*$ and
$\det\Leech=3^6$.  Also see \cite{Wilson} for a thorough study of
$\Aut\Leech$, which is the universal central extension $6\cdot\Suz$ of
Suzuki's sporadic finite simple group~$\Suz$.  It will be important
for us that $6\cdot\Suz$ acts transitively on the lattice vectors of
norm~$6$ resp.\ $9$.  We will postpone an explicit description of
$\Lambda$ until the second half of
appendix~\ref{app-how-real-triangles-meet-mirrors}, because we won't
need it until then.

The Leech model of $L$ is useful when one has in mind a Leech cusp
(see below).  It is:
$L\iso\Leech\oplus\bigl(\begin{smallmatrix}0&\thetabar\\\theta&0\end{smallmatrix}\bigr)$.
  This means that we write lattice vectors as $(x;y,z)$ where
  $x\in\Leech$ and $y,z\in\E$, with the inner product given by
\begin{equation}
  \label{eq-inner-product-in-Leech-model}
  \Bigip{(x;y,z)}{(x';y',z')}
  =\ip{x}{x'}+
  \begin{pmatrix}
    y&z
  \end{pmatrix}
  \begin{pmatrix}
    0&\thetabar\\\theta&0
  \end{pmatrix}
  \begin{pmatrix}
    \ybar'\\\zbar'
  \end{pmatrix}
\end{equation}
This model of $L$ was introduced in \cite{Allcock-Inventiones}, and
proven to be isometric to the $P^2\F_3$ model in \cite[Lemma~2.6]{Basak-bimonster-1}.  The Leech
cusp $\rho=(0;0,1)$ is distinguished, and the mirrors nearest it can
be conveniently parameterized as explained in section~\ref{subsec-Leech-cusps-and-Leech-roots}.

\subsection{Leech cusps and Leech roots}
\label{subsec-Leech-cusps-and-Leech-roots}
Earlier we declared that if $v\in L$ is a primitive null vector then
we call it (or the point of $\partial\BB^{13}$ it represents) a cusp.
We refine this to ``Leech cusp'' in the special case that
$v^\perp/\gend{v}$ is isometric to the Leech lattice.  An example
which will play a key role in this paper is the primitive null vector
$\rho$ defined as $(0;0,1)$ in the Leech model.  For background 
we
remark that there are five $\Gamma$-orbits on primitive null vectors
$v\in L$, corresponding to the five possibilities for the isometry
class of $v^\perp/\gend{v}$, which are the ``Eisenstein Niemeier
lattices''.  (See
Lemma~2 and Theorem~4 of \cite{Allcock-Y555}.)
In the $P^2\F_3$ model it is harder to
find a Leech cusp, but an example is $\rho=(3\w-1;-1,\dots,-1)$.
(Proof sketch following the proof of theorem~1 in \cite{Allcock-Y555}: $\rho^\perp/\gend{\rho}$ is an
Eisenstein Niemeier lattice, and its isometry group obviously contains
$L_3(3)$. By \cite[Thm.~4]{Allcock-Y555} the Leech lattice is the only Eisenstein
Niemeier lattice whose symmetry group contains~$L_3(3)$.)  We hope
that the simultaneous use of~$\rho$ for $(0;0,1)$ in the Leech model
and $(3\w-1;-1,\dots,-1)$ in the $P^2\F_3$ model is helpful rather
than confusing: we think of the models as two ways to describe a
single lattice.  In lemma~\ref{lem-point-and-line-roots-in-Leech-model} we
complete the identification of the two models by explicitly writing
down the point- and line-roots in the Leech model.

The following description of
vectors $s\in L\tensor\C$ not orthogonal to $\rho$ is very useful: every such $s$
can be written uniquely in the form
\begin{equation}
\label{eq-definition-of-vector-s}
s
=
\biggl(
\sigma
;
m
,
\frac{\theta}{\mbar}\Bigl(\frac{\sigma^2-N}{6}+\nu\Bigr)
\biggr)
\end{equation}
where $\sigma\in\Leech\tensor\C$, $m\in\C-\{0\}$, $N$ is the norm
$s^2$, and $\nu$ is purely imaginary.  Restricting the first
component to $\Leech$ and the others to $\E$ gives the elements of
$L-\rho^\perp$.  Further restricting $N$ to~$3$ gives the roots of
$L$, and finally restricting $m$ to~$1$ gives the Leech roots (see
equation~\eqref{eq-Leech-roots-in-Leech-model} below).
We remark that scaling $s$ by $\lambda\in\C-\{0\}$ scales $\sigma$ and
$m$ by $\lambda$, and $N$ and $\nu$ by $|\lambda|^2$.

The elaborate form of the last coordinate
in \eqref{eq-definition-of-vector-s}
allows the following interpretation of $\ip{s}{s'}$, where $s$ is from
\eqref{eq-definition-of-vector-s} and similarly for $s'$:
\begin{equation}
\label{eq-general-inner-product-formula}
\begin{split}
\ip{s}{s'}
=
m\mbar'
\biggl[
\frac{1}{2}
\biggl(
\frac{N'}{|m'|^2}+&\frac{N}{|m|^2}
-
\Bigl(\frac{\sigma}{m}-\frac{\sigma'}{m'}\Bigr)^2
\biggl)
\\
&+
\Im\Bigip{\frac{\sigma}{m}}{\frac{\sigma'}{m'}}
+
3\Bigl(\frac{\nu'}{|m'|^2}-\frac{\nu}{|m|^2}\Bigr)
\biggr].
\end{split}
\end{equation}
The content of this is that
$\ip{s}{s'}$ is governed by the relative positions of $\sigma/m$ and
$\sigma'/m'$ in $\Leech\tensor\C$, with $\nu$ and $\nu'$ influencing
only the imaginary part of the bracketed term.
(\emph{Caution:} we are using the convention that the imaginary part
of a complex number is imaginary; for example $\Im\theta$ is $\theta$
rather than $\sqrt3$.)  
One proves the formula by writing out $\ip{s}{s'}$,
completing the square and patiently rearranging.

Because the Leech lattice has no vectors of norm${}<6$, a Leech cusp
$\rho\in\partial\BB^{13}$ lies in (the closures of) no mirrors.  The
mirrors which come closest to it are called first-shell mirrors, and
the second- and third-shell mirrors are defined similarly.  We use the
same language for their corresponding roots.  When the Leech cusp is
$\rho$, then $s$ is a first-, second- or third-shell root if
$|\ip{\rho}{s}|^2$ is $3$, $9$ or~$12$ respectively.  It will be useful
to fix a scaling of the first-shell roots: a Leech root (with respect
to~$\rho$) means a root $s$ with $\ip{\rho}{s}=\theta$, and 
we call its mirror a Leech mirror.

In the Leech model, the Leech roots (with respect to $\rho$) are the
vectors
\begin{equation}
  \label{eq-Leech-roots-in-Leech-model}
s
=
\biggl(
\sigma
;
1
,
\theta\Bigl(\frac{\sigma^2-3}{6}+\nu\Bigr)
\biggr)
\end{equation}
with $\sigma\in\Leech$ and
$\nu\in\frac{1}{\theta}\Z+\frac{1}{2\theta}$ if $\sigma^2$ is
divisible by~$6$, or $\nu\in\frac{1}{\theta}\Z$ otherwise.  This is a
specialization of \eqref{eq-definition-of-vector-s}. 
We note that if $(\sigma; 1, \alpha)$ is a Leech root,
then the Leech roots of the form $(\sigma; 1, *)$ are
$(\sigma; 1, \alpha + n)$ as $n$ runs over $\Z$.
In the $P^2\F_3$ model there is no simple
formula for the Leech roots, but (with respect to $\rho$) the
point-roots are examples of Leech roots, and the line-roots are
examples of second-shell roots.
For later use, we record the inner products
\begin{equation*}
\ip{\rho}{p_i} = \theta \text{\; and \;} 
\ip{\rho}{l_i} = 3 \bar{\omega} \text{\; for \;} i = 1, \dotsb, 13.
\end{equation*}

In section~\ref{sec-paths-and-homotopies} we will choose a closed horoball $A$ centered at
$\rho$ and disjoint from $\H$, and choose a point $a\in A$.  We will
speak of ``meridians'' to refer to certain elements $M_{a,A,H}$ of the
orbifold fundamental group (see below) of
$(\BB^{13}-\H)/P\Gamma$, or their underlying paths $\mu_{a,A,H}$.
The definition of a meridian involves a choice of a mirror~$H$;
we will call it a Leech meridian if $H$ is a Leech
mirror.

\subsection{Meridians and the orbifold fundamental group}
\label{subsec-orbifold-fundamental-group}
The orbifold fundamental group $G_b=\piorb\bigl((\BB^{13}-\H)/\PG,b\bigr)$
depends on the choice of a base point $b\in\BB^{13}-\H$.  
It
means the following set of equivalence classes of
pairs 
$(\gamma,g)$, where $g\in\PG$ and $\gamma$ is a path in $\BB^{13}-\H$
from $b$ to $g(b)$.  One such pair is equivalent to another
one $(\gamma',g')$ if $g=g'$ and $\gamma$ and $\gamma'$ are homotopic
in $\BB^{13}-\H$, rel endpoints.   
The group operation
is 
$(\gamma,g)\cdot(\gamma',g')=(\hbox{$\gamma$ followed by
  $g\circ\gamma'$},g g')$.
Inversion is given by $(\gamma,g)^{-1}=(g^{-1}\circ\reverse(\gamma),g^{-1})$.
Projection of $(\gamma,g)$ to $g$ defines a homomorphism $G_b\to\PG$.
It is surjective because $\BB^{13}-\H$ is path-connected.
The kernel is obviously $\pi_1(\BB^{13}-\H,b)$, yielding the exact sequence
\begin{equation}
\label{eq-exact-sequence-on-pi-1}
1\to\pi_1\bigl(\BB^{13}-\H,b\bigr)\to G_b\to\PG\to1
\end{equation}

The elements of $G_b$ we will use are the ``meridians'' $M_{b,H}$ and
$M_{b,q,H}$, where $H$ is a mirror and $q$ is a point of $H$.  These
are defined in section~\ref{sec-paths-and-homotopies}.  Their
underlying paths are written $\mu_{b,H}$ and $\mu_{b,q,H}$ and are
also called meridians.  The meridians associated to the point- and
line-mirrors are called the point- and line-meridians.  The main
theorem of this paper is that $G_\tau$ is generated by these $26$
meridians, when the basepoint $\tau$ is the $26$-point from
section~\ref{subsec-point-roots-etc}.  There are also ``fat
basepoint'' versions of meridians: $M_{a,A,H}$, $M_{a,A,q,H}$,
$\mu_{a,A,H}$ and $\mu_{a,A,q,H}$ with~$A$ a horoball centered at a
Leech cusp, small enough to miss~$\H$, and $a$ is a point of~$A$.
These are also defined in section~\ref{sec-paths-and-homotopies}.  Of
the meridians of this sort, we mostly use those with $H$ a Leech
mirror, which we call Leech meridians.  We also speak of second-shell
meridians, with the corresponding meaning.

In any $\mu_{\cdots}$ or $M_{\cdots}$, we write just $s$ instead of
$s^\perp$ in the subscript when the hyperplane is the mirror $s^\perp$
of a root.  For example, $\mu_{b,s}$ rather than $\mu_{b,s^\perp}$.

\section{Meridians and homotopies between them}
\label{sec-paths-and-homotopies}

\noindent
At its heart this paper consists of explicit manipulations of paths in
$\BB^{13}-\H$.  In this section we describe the most important paths
and manipulations.  
For concreteness we work only with our particular arrangement $\H$ in
$\BB^{13}$, but it will be obvious that everything works equally well
in the generality of \cite{AB-braidlike}.  (In particular, 
for the mirror arrangements of finite complex reflection groups.)
The ideas are simple, and we hope the
following summary will enable the reader to skip the technical
details.

Our basic path is called a ``meridian'' and written
$\mu_{b,q,H}$.  Here $b$ is a point of $\BB^{13}-\H$,
$H$ is a mirror, and $q$ is a point of $H$.  The path begins
at $b$ and travels along a geodesic towards $q$, stopping near it.
Next it travels directly toward $H$, stopping very near it.  Then it
travels $2\pi/3$ of the way around $H$, on a positively oriented
circular arc.  Finally it ``returns'' by retracing the paths in the
first two steps of the construction (or rather their images under the
$\w$-reflection in $H$).  See figure~\ref{fig-examples-of-meridians-new}.  The term ``meridian'' is
from knot theory: we think of the image of $\H$ in 
$\BB^{13}/P\Gamma$ as a sort of knot.  ``Meridian''  indicates
that (the image of) $\mu_{b,q,H}$ is a loop, freely homotopic to a
circle around (the image of) $H$.

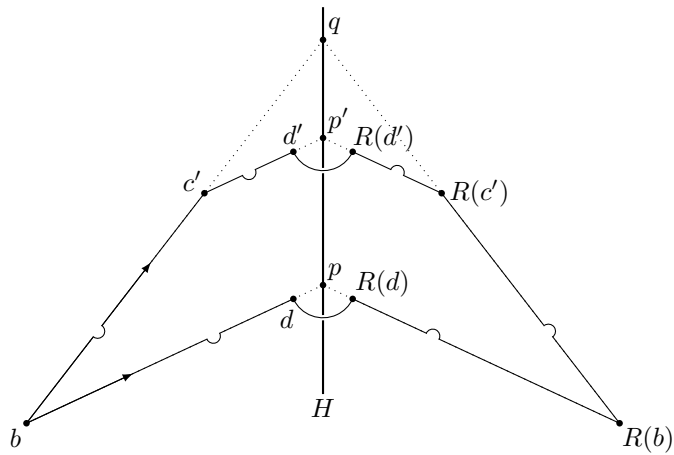
\begin{figure}
  \def\declarecenter#1#2#3#4{%
    \coordinate (middle1) at ($(#1)!.5!(#2)$);%
    \coordinate (middle2) at ($(#2)!.5!(#3)$);%
    \coordinate (aux1) at ($(middle1)!1!90:(#2)$);%
    \coordinate (aux2) at ($(middle2)!1!90:(#3)$);%
    \coordinate (#4) at ($(intersection of middle1--aux1 and middle2--aux2)$);%
  }
  \begin{center}
    \begin{tikzpicture}[x=1.45cm,y=1.45cm]
      \def\URx{3.3}
      \def\URy{.3}
      \def\LLx{-2.9}
      \def\LLy{-3.8}
      \clip(\LLx,\LLy)rectangle(\URx,\URy);
      \def\similarityfactor{2.5}
      \def\qTOpprimeDIST{.9}
      \def\pTOdANGLE{-155}
      \def\pTOdDIST{.3}
      \def\pprimeTOcprimeDIST{1.2}
      \coordinate (q) at (0,0);
      \path (q)++(0,-\qTOpprimeDIST) coordinate (pprime);
      \path (q)++(0,-\qTOpprimeDIST*\similarityfactor) coordinate (p);
      \path (p)++(\pTOdANGLE:\pTOdDIST) coordinate (d);
      \path (pprime)++(\pTOdANGLE:\pTOdDIST) coordinate (dprime);
      \path (pprime)++(\pTOdANGLE:\pprimeTOcprimeDIST) coordinate (cprime);
      \path (p)++(\pTOdANGLE:\pprimeTOcprimeDIST*\similarityfactor) coordinate (b);
      \path (p)++(180-\pTOdANGLE:\pTOdDIST) coordinate (Rd);
      \path (pprime)++(180-\pTOdANGLE:\pTOdDIST) coordinate (Rdprime);
      \path (pprime)++(180-\pTOdANGLE:\pprimeTOcprimeDIST) coordinate (Rcprime);
      \path (p)++(180-\pTOdANGLE:\pprimeTOcprimeDIST*\similarityfactor)coordinate(Rb);
      \path (p)++(0,-1) coordinate (bottom);
      \path (q)++(0,.3) coordinate (top);
      \draw[thick] (top)--(bottom);
      \draw[white,line width=2pt]
      (p)++(\pTOdANGLE:\pTOdDIST)arc(\pTOdANGLE:-180-\pTOdANGLE:\pTOdDIST);
      \draw(p)++(\pTOdANGLE:\pTOdDIST)arc(\pTOdANGLE:-180-\pTOdANGLE:\pTOdDIST);
      \draw[white,line width=2pt]
      (pprime)++(\pTOdANGLE:\pTOdDIST)arc(\pTOdANGLE:-180-\pTOdANGLE:\pTOdDIST);
      \draw(pprime)++(\pTOdANGLE:\pTOdDIST)arc(\pTOdANGLE:-180-\pTOdANGLE:\pTOdDIST);
      \draw (d)--(b)--(cprime)--(dprime);
      \draw (Rd)--(Rb)--(Rcprime)--(Rdprime);
      \def\dotradius{.03}
      \foreach\x in{p,q,pprime} \fill (\x) circle (\dotradius);
      \foreach\x in{d,cprime,dprime,b} \fill (\x) circle (\dotradius);
      \foreach\x in{Rd,Rcprime,Rdprime,Rb}\fill(\x) circle (\dotradius);
      \coordinate  (tip) at ($(b)!.4!(d)$);
      \draw[-latex] (b)--(tip);
      \coordinate (tip) at ($(b)!.7!(cprime)$);
      \draw[-latex] (b)--(tip);
      \draw [dotted] (p)--(b);
      \draw [dotted] (p)--(Rb);
      \draw [dotted] (q)--(b);
      \draw [dotted] (q)--(Rb);
      \draw [dotted] (pprime)--(dprime);
      \draw [dotted] (pprime)--(Rdprime);
      \def\divRwhite{.06}
      \def\divR{.065}
      \coordinate (div) at ($(b)!.7!(d)$);%
      \fill[white] (div) circle (\divRwhite);
      \draw(div)++(\pTOdANGLE:\divR)arc(\pTOdANGLE:180+\pTOdANGLE:\divR);
      \coordinate (Rdiv) at ($(Rb)!.7!(Rd)$);%
      \fill[white] (Rdiv) circle (\divRwhite);
      \draw(Rdiv)++(-\pTOdANGLE:\divR)arc(-\pTOdANGLE:-180-\pTOdANGLE:\divR);
      \coordinate (div) at ($(cprime)!.5!(dprime)$);%
      \fill[white] (div) circle (\divRwhite);
      \draw(div)++(\pTOdANGLE:\divR)arc(\pTOdANGLE:180+\pTOdANGLE:\divR);
      \coordinate (Rdiv) at ($(Rcprime)!.5!(Rdprime)$);%
      \fill[white] (Rdiv) circle (\divRwhite);
      \draw(Rdiv)++(-\pTOdANGLE:\divR)arc(-\pTOdANGLE:-180-\pTOdANGLE:\divR);
      \coordinate (div) at ($(b)!.4!(cprime)$);%
      \fill[white] (div) circle (\divRwhite);
      \draw(div)++(-125:\divR)arc(-125:180-125:\divR);
      \coordinate (div) at ($(Rb)!.4!(Rcprime)$);%
      \fill[white] (div) circle (\divRwhite);
      \draw(div)++(-55:\divR)arc(-55:180-55:\divR);
      \tikzstyle{lab}=[inner sep=0pt, outer sep=2pt]
      \draw(p)node[lab,anchor=south west]{$p$};
      \draw(pprime)node[lab,anchor=south west]{$p'$};
      \draw(q)node[lab,anchor=south west]{$q$};
      \draw(bottom)node[lab,anchor=north]{$H$};
      \draw(b)node[lab,anchor=north east]{$b$};
      \draw(d)node[lab,anchor=north,outer sep=4pt]{$d\,\,\,$};
      \draw(cprime)node[lab,anchor=south east,outer sep=1pt]{$c'$};
      \draw(dprime)node[lab,anchor=south,outer sep=3pt]{$d'$};
      \draw(Rb)node[lab,anchor=north west,outer sep=1pt]{$R(b)$};
      \draw(Rd)node[lab,anchor=south west,outer sep=1pt]{$R(d)$};
      \draw(Rcprime)node[lab,anchor=west,outer sep=3pt]{$R(c')$};
      \draw(Rdprime)node[lab,anchor=south west,outer sep=0pt]{$R(d')$};
    \end{tikzpicture}
  \end{center}
  \caption{The meridians $\mu_{b,q,H}$ (top) and $\mu_{b,H}$ (bottom)
    go left to right from $b$ to $R(b)$.  Here $b$ is the basepoint,
    $H$ is a mirror, $R$ is its $\w$-reflection, $q\in H$, and $p$ is
    the point of $H$ nearest $b$.  The semicircular arcs in the paths
    indicate possible diversions around points of $\H$.}
  \label{fig-examples-of-meridians-new}
\end{figure}

We abbreviate the notation to $\mu_{b,H}$ when $q$ is the point of $H$
nearest $b$.  These are the meridians considered in our earlier paper
\cite{AB-braidlike}.

Lemmas \ref{lem-homotopy-for-moving-the-basepoint} and~\ref{lem-homotopy-for-moving-the-point-on-the-hyperplane-ordinary-basepoint-version} are the main results of this section.
The conceptual content of
lemma~\ref{lem-homotopy-for-moving-the-point-on-the-hyperplane-ordinary-basepoint-version}
is the following.  Consider the first turning point of $\mu_{b,q,H}$,
which we took to be near $q$ (called $c'$ in
figure~\ref{fig-examples-of-meridians-new}).  If we used a
point slightly further from $q$ in place of $c'$, then we would obtain a homotopic
path.  If we gradually moved the turning point all the way back to
$b$, and all the intermediate paths missed $\H$, then we would have a
homotopy from $\mu_{b,q,H}$ to $\mu_{b,H}$.
Lemma~\ref{lem-homotopy-for-moving-the-point-on-the-hyperplane-ordinary-basepoint-version}
shows that this applies if $\triangle b p q$ misses $\H$ except at the
obvious intersection points.  Here $p$ is the point of $H$ closest to
$b$.

The conceptual content of
lemma~\ref{lem-homotopy-for-moving-the-basepoint} is the following.
Consider a second basepoint $b'$, and join $b$ to it by some path
$b_{t\in[0,1]}$ that misses $\H$.  One expects that under reasonable
hypotheses, the family of paths $\mu_{b_t,q,H}$ will form a homotopy
between $\mu_{b,q,H}$ and $\mu_{b',q,H}$.
Lemma~\ref{lem-homotopy-for-moving-the-basepoint} gives sufficient
conditions for this: first, the geodesics $\geodesic{b_t q}$ should
miss $\H$ (except at $q$), and second, $H$ should be orthogonal to all
other mirrors through $q$.

Informally we think of the meridians $\mu_{b,q,H}$ and $\mu_{b,H}$ as
elements of the orbifold fundamental group
$G_b=\piorb\bigl((\BB^{13}-\H)/P\Gamma,b\bigr)$.  But strictly
speaking, an element of this group is an ordered pair (see
section~\ref{subsec-orbifold-fundamental-group}).  So we also say
``meridian'' for the ordered pairs $(\mu_{b,q,H},R_H)$ and
$(\mu_{b,H},R_H)$, where $R_H$ is the $\w$-reflection with mirror~$H$.
We write $M_{b,q,H}$ and $M_{b,H}$ for them.

In our applications the hyperplane is always the mirror of a root
$s$. As mentioned in section~\ref{subsec-orbifold-fundamental-group},
when writing $\mu_{\cdots}$ or $M_{\cdots}$ we will write just $s$ in
the subscript rather than $s^\perp$, when $s$ is a root and the
hyperplane is its mirror.  For example, $\mu_{b,s}$ rather than
$\mu_{b,s^\perp}$.

Our informal statements of lemmas \ref{lem-homotopy-for-moving-the-basepoint} and~\ref{lem-homotopy-for-moving-the-point-on-the-hyperplane-ordinary-basepoint-version} are so natural,
even obvious, that the reader may wonder what else there is to say.
There are two things to fuss over.  First, the basepoint might not be in general
position with respect to the hyperplanes.    So the proper
definition of $\mu_{b,q,H}$ must include detours around any
hyperplanes met by the version of $\mu_{b,q,H}$ described above.
Then it requires work even to show that the
resulting homotopy class is well-defined (lemma~\ref{lem-well-definedness-of-homotopy-class}).  Second, we
would like to use the Leech cusp $\rho$ as a sort of basepoint, even
though $\rho$ is not a point of $\BB^{13}$.  We accomplish this by
using the closed horoball $A$ centered at $\rho$ and disjoint from
$\H$, and a point $a\in A$.
Informally we think of $\rho$ as the basepoint, but $a$ is the
official basepoint.  In the end one can work with the resulting
meridians just like the ones above.  But the notation acquires an
extra subscript: $\mu_{a,A,q,H}$, $\mu_{a,A,H}$, $M_{a,A,q,H}$
and~$M_{a,A,H}$.
Finally, to realize that lemma \ref{lem-homotopy-for-moving-the-basepoint}
is not a mere tautology, 
 notice the role played by the technical assumption:
{\it $H$ is orthogonal to all other mirrors through $q$}.

\medskip
Now we begin the technical content of the section.  For
$b,c\in\overline{\BB^{13}}$ we write $\geodesic{b c}$ for the geodesic
segment from $b$ to $c$.  Now suppose $b,c\notin\H$.  It may happen
that $\geodesic{b c}$ meets $\H$, so we define a perturbation
$\dodge{b c}$ of $\geodesic{b c}$ in the obvious way.  The notation may
be pronounced ``$b$ dodge $c$'' or ``$b$ detour $c$'', and the precise
definition is the following.  We write $\complexgeodesic{b c}$ for the
complex line containing $\geodesic{b c}$.  By the local finiteness of
the mirror arrangement, $\complexgeodesic{b c}\cap\H$ is a discrete
set.  Consider the path got from $\geodesic{b c}$ by using positively
oriented semicircular detours in $\complexgeodesic{b c}$, around the
points of $\geodesic{b c}\cap\H$, in place of the corresponding
segments of $\geodesic{b c}$.  By $\dodge{b c}$ we mean this path,
with the radius of these detours taken to be small enough.  (This
means: small enough for the construction to make sense and the
resulting homotopy class in $\BB^{13}-\H$, rel endpoints, to be
radius-independent.)

Now we define $\mu_{b,H}$ for $b\in\BB^{13}-\H$ and $H$ a hyperplane
of $\H$.  See the bottom path in figure~\ref{fig-examples-of-meridians-new}.  Write $p$ for the point of $H$ nearest $b$ and $R$ for the
$\w$-reflection with mirror $H$.   Choose an open ball $U$ around
$p$, small enough to miss all the mirrors except those through~$p$,
and choose some $d\in U\cap\bigl(\geodesic{b p}-\{p\}\bigr)$.
From $b\notin\H$ we get $d\notin\H$.
Then $\mu_{b,H}$
means $\dodge{b d}$, followed by the positively oriented circular arc
in $\complexgeodesic{b p}$ from $d$ to $R(d)$, centered at $p$,
followed by $R\bigl(\reverse(\dodge{b d})\bigr)$.  It is easy to see
that the resulting path is independent of our choices of $U$ and $d$,
up to homotopy rel endpoints in $\BB^{13}-\H$.
As mentioned above, $M_{b,H}$ means the element $(\mu_{b,H},R)$ of the
orbifold fundamental group, and is also called a meridian.
Although we don't need
it, we mention for background that $M_{b,H}^3$ is the loop in
$\BB^{13}-\H$ that we called $\Loop{b H}$ (``$b$ loop $H$'') in
\cite{AB-braidlike}.  It goes from $b$ to $d$, encircles $H$ in
$\complexgeodesic{b p}$, and then returns to~$b$.

Next we define $\mu_{b,q,H}$, where $q$ is a point of $H$. See the top
path in figure~\ref{fig-examples-of-meridians-new}.  Choose an open
ball $U$ around $q$ that is small enough to miss all the mirrors
except those through~$q$, and choose some point $c'\in
U\cap\bigl(\geodesic{b q}-\{q\}\bigr)$.  From $b\notin\H$ we get
$c'\notin\H$, so $\mu_{c',H}$ is defined. We define $\mu_{b,q,H}$ as
$\dodge{b c'}$ followed by $\mu_{c',H}$ followed by
$R(\reverse(\dodge{b c'}))$.
Figure~\ref{fig-examples-of-meridians-new} also shows primed versions
of the points $p$ and $d$ from the previous paragraph's definition of
$\mu_{c',H}$.  The next lemma shows that $\mu_{b,q,H}$ is well-defined
as a homotopy class.  The tricky part is dealing with possible detours
in the subpath $\mu_{c',H}$.  (For the meridians in this paper, it
happens that no such detours occur, in which case the lemma is
obvious.)  Just as with $M_{b,H}$, we refer to the orbifold
fundamental group element $M_{b,q,H}=(\mu_{b,q,H},R)$ as a meridian.

\begin{lemma}[Well-definedness of $\mu_{b,q,H}$]
  \label{lem-well-definedness-of-homotopy-class}
 The homotopy class of $\mu_{b,q,H}$, rel endpoints, is independent of
 the choices of $U$ and~$c'$.
\end{lemma}

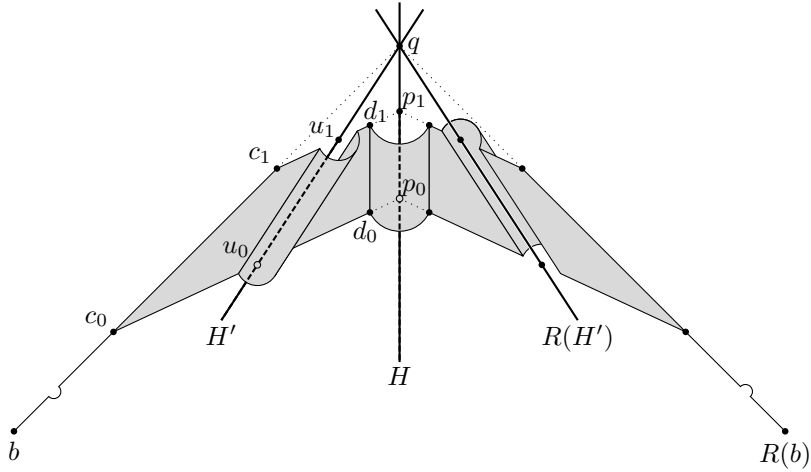
\begin{figure}
  \begin{center}
    \begin{tikzpicture}[x=1.45cm,y=1.45cm]
      \def\URx{3.8}
      \def\URy{.4}
      \def\LLx{-3.7}
      \def\LLy{-3.9}
      \clip(\LLx,\LLy)rectangle(\URx,\URy);
      \def\qTOponeDIST{.6}
      \def\qTOpzeroDIST{1.4}
      \def\Hcolor{gray!30}
      \def\pTOdANGLE{-155}
      \def\pTOdDIST{.3}
      \def\qTObANGLE{-135}
      \def\qTObDIST{5}
      \def\HprimeANGLE{57}
      \def\HprimetopDIST{.4}
      \def\HprimebottomDIST{3}
      \def\eps{.19}
      \coordinate (q) at (0,0);
      \coordinate (p1) at (0,-\qTOponeDIST);
      \coordinate (p0) at (0,-\qTOpzeroDIST); 
      \path (p0)++(\pTOdANGLE:\pTOdDIST) coordinate (d0);
      \path (p1)++(\pTOdANGLE:\pTOdDIST) coordinate (d1);
      \coordinate (b) at (\qTObANGLE:\qTObDIST);
      \coordinate (c0) at ($(intersection of p0--d0 and q--b)$);%
      \coordinate (c1) at ($(intersection of p1--d1 and q--b)$);%
      \coordinate (Hprimetop) at (\HprimeANGLE:\HprimetopDIST);
      \coordinate (Hprimebottom) at (\HprimeANGLE:-\HprimebottomDIST);
      \coordinate (u0) at ($(intersection of p0--c0 and Hprimetop--Hprimebottom)$);%
      \path (u0)++(\pTOdANGLE:\eps) coordinate (x0);
      \path (u0)++(\HprimeANGLE-90:\eps) coordinate (y0);
      \path (u0)++(180+\pTOdANGLE:\eps) coordinate (z0);
      \coordinate (u1) at ($(intersection of p1--c1 and Hprimetop--Hprimebottom)$);%
      \path (u1)++(\pTOdANGLE:\eps) coordinate (x1);
      \path (u1)++(\HprimeANGLE-90:\eps) coordinate (y1);
      \path (u1)++(180+\pTOdANGLE:\eps) coordinate (z1);
      \path (p0)++(180-\pTOdANGLE:\pTOdDIST) coordinate (Rd0);
      \path (p1)++(180-\pTOdANGLE:\pTOdDIST) coordinate (Rd1);
      \coordinate (Rb) at (180-\qTObANGLE:\qTObDIST);
      \coordinate (Rc0) at ($(intersection of p0--Rd0 and q--Rb)$);%
      \coordinate (Rc1) at ($(intersection of p1--Rd1 and q--Rb)$);%
      \coordinate (RHprimetop) at (-\HprimeANGLE:-\HprimetopDIST);
      \coordinate (RHprimebottom) at (-\HprimeANGLE:\HprimebottomDIST);
      \coordinate (Ru0) at ($(intersection of p0--Rc0 and RHprimetop--RHprimebottom)$);%
      \path (Ru0)++(180-\pTOdANGLE:\eps) coordinate (Rx0);
      \path (Ru0)++(-\HprimeANGLE+90:\eps) coordinate (Ry0);
      \path (Ru0)++(-\pTOdANGLE:\eps) coordinate (Rz0);
      \coordinate (Ru1) at ($(intersection of p1--Rc1 and RHprimetop--RHprimebottom)$);%
      \path (Ru1)++(180-\pTOdANGLE:\eps) coordinate (Rx1);
      \path (Ru1)++(-\HprimeANGLE+90:\eps) coordinate (Ry1);
      \path (Ru1)++(-\pTOdANGLE:\eps) coordinate (Rz1);
      \path (p0)++(0,-1.5) coordinate (bottom);
      \path (q)++(0,.4) coordinate (top);
      \draw[thick] (top)--(bottom);
      \draw[thick] (Hprimetop)--(Hprimebottom);
      \fill[\Hcolor](c0)--(x0)--(x1)--(c1)--cycle;
      \draw(x1)--(x0)--(c0)--(c1)--(x1);
      \fill[\Hcolor](d0)--(z0)--(z1)--(d1)--cycle;
      \draw(d0)--(z0)--(z1)--(d1)--cycle;
      \def\Hleftpathback{(y0) arc (\HprimeANGLE-90:\pTOdANGLE+180:\eps)
      --(z1) arc (\pTOdANGLE+180:\HprimeANGLE-90:\eps)--cycle}
      \fill[\Hcolor]\Hleftpathback;
      \draw\Hleftpathback;
      \def\Hleftpathfront{(y0) arc (\HprimeANGLE-90:\pTOdANGLE:\eps)
        --(x1) arc (\pTOdANGLE:\HprimeANGLE-90:\eps)--cycle}
      \fill[\Hcolor]\Hleftpathfront;
      \draw\Hleftpathfront;
      \def\Hmiddlepath{%
      (p0)++(\pTOdANGLE:\pTOdDIST)
      arc (\pTOdANGLE:-180-\pTOdANGLE:\pTOdDIST)
      -- (Rd1)
      arc (-180-\pTOdANGLE:\pTOdANGLE:\pTOdDIST)
      -- (d1)--cycle}
      \fill[\Hcolor] \Hmiddlepath;
      \draw\Hmiddlepath;
      \def\Hrightpathback{(Rz0)arc(-\pTOdANGLE:-\HprimeANGLE+90:\eps)
        --(Ry1)arc(-\HprimeANGLE+90:-\pTOdANGLE:\eps)--cycle}
      \fill[\Hcolor]\Hrightpathback;
      \draw\Hrightpathback;
      \draw (Rx1)arc(-180-\pTOdANGLE:-\pTOdANGLE:\eps);
      \fill[\Hcolor](Rz0)--(Rz1)--(Rd1)--(Rd0)--cycle;
      \draw(Rz0)--(Rz1)--(Rd1)--(Rd0)--(Rz0);
      \fill[\Hcolor](Rx0)--(Rx1)--(Rc1)--(Rc0)--cycle;
      \draw(Rx1)--(Rx0)--(Rc0)--(Rc1)--(Rx1);
      \draw[white,line width=1.4pt]($(RHprimetop)!.3!(RHprimebottom)$)--(RHprimebottom);
      \draw[thick] (RHprimetop)--(RHprimebottom);
      \pgfsetdash{{2.3pt}{1pt}}{2.4pt}
      \draw[thick](p0)--(p1);
      \pgfsetdash{{2.3pt}{1pt}}{3.8pt}
      \draw[thick](p0)--(bottom);
      \pgfsetdash{{2.3pt}{1pt}}{2.3pt}
      \draw[thick](u0)--(Hprimebottom);
      \pgfsetdash{{2.3pt}{1pt}}{2.2pt}
      \draw[thick](u0)--(u1);
      \pgfsetdash{}{0pt}
      \draw (b)--(c0);
      \draw (Rb)--(Rc0);
      \draw [dotted] (q)--(c1);
      \draw [dotted] (q)--(Rc1);
      \draw [dotted] (p0)--(d0);
      \draw [dotted] (p0)--(Rd0);
      \draw [dotted] (p1)--(d1);
      \draw [dotted] (p1)--(Rd1);
      \coordinate (div) at ($(b)!.4!(c0)$);%
      \def\divRwhite{.06}
      \def\divR{.065}
      \fill[white] (div) circle (\divRwhite);
      \draw(div)++(\qTObANGLE:\divR)arc(\qTObANGLE:180+\qTObANGLE:\divR);
      \coordinate (Rdiv) at ($(Rb)!.4!(Rc0)$);%
      \fill[white] (Rdiv) circle (\divRwhite);
      \draw(Rdiv)++(-\qTObANGLE:\divR)arc(-\qTObANGLE:-180-\qTObANGLE:\divR);
      \def\dotradius{.03}
      \foreach\x in{q,p1} \fill (\x) circle (\dotradius);
      \foreach\x in{b,c0,c1,d0,d1,u1} \fill (\x) circle (\dotradius);
      \foreach\x in{Rb,Rc0,Rc1,Rd0,Rd1,Ru0,Ru1} \fill (\x) circle (\dotradius);
      \fill[\Hcolor] (p0) circle (\dotradius); 
      \draw (p0) circle (\dotradius);
      \fill[\Hcolor] (u0) circle (\dotradius); 
      \draw (u0) circle (\dotradius);
      \tikzstyle{lab}=[inner sep=0pt, outer sep=2pt]
      \draw(q)node[lab,anchor=west,outer sep=3pt]{$q$};
      \draw(p1)node[lab,anchor=south west,outer sep=1pt]{$p_1$};
      \draw(p0)node[lab,anchor=south west,outer sep=1pt]{$p_0$};
      \draw(bottom)node[lab,anchor=north]{$H$};
      \draw(b)node[lab,anchor=north,outer sep=4pt]{$b$};
      \draw(c0)node[lab,anchor=south east]{$\,\,c_0$};
      \draw(c1)node[lab,anchor=south east]{$c_1$};
      \draw(d0)node[lab,anchor=north,outer sep=3pt]{$d_0\,\,$};
      \draw(d1)node[lab,anchor=south,outer sep=1pt]{$\,\,\,d_1$};
      \draw(u0)node[lab,anchor=south east,outer sep=3pt]{$u_0$};
      \draw(u1)node[lab,anchor=south east,outer sep=1pt]{$u_1$};
      \draw(Hprimebottom)node[lab,anchor=north]{$H'$};
      \draw(Rb)node[lab,anchor=north,outer sep=4pt]{$R(b)$};
      \draw(RHprimebottom)node[lab,anchor=north]{$R(H')$};
    \end{tikzpicture}
  \end{center}
  \caption{The well-definedness of $\mu_{b,q,H}$; see lemma~\ref{lem-well-definedness-of-homotopy-class}.
  The unlabeled points on the right side are the $R$-images of $c_0$,
  $c_1$, $u_0$, $u_1$, $d_0$ and $d_1$.}
  \label{fig-well-definedness-of-mu-b-q-H}
\end{figure}

\begin{proof}
    Suppose $U$ is a ball around $q$ as above, and
    write $c_0,c_1\in U\cap\bigl(\geodesic{bq}-\{q\}\bigr)$ for two
    candidates for $c'$.  We choose the subscripts so
    that $c_0$ is further from $q$ than $c_1$ is.  We will show that
    the version of $\mu_{b,q,H}$ defined using $c'=c_0$ is homotopic
    rel endpoints to the version of $\mu_{b,q,H}$ defined using
    $c'=c_1$.  The shaded surface in figure~\ref{fig-well-definedness-of-mu-b-q-H} is the homotopy that we
    will construct.  It is a homotopy from
    $\mu_{c_0,H}$ to $\mu_{c_1,H}$ that moves $c_0$ along
    $\geodesic{c_0 c_1}$. This is enough to
    build a homotopy between the two versions of $\mu_{b,q,H}$.
    It follows that the homotopy class of $\mu_{b,q,H}$ is independent
    of the choice of $c$.   It follows  that it is also
    independent of the choice of~$U$.
    
    We parameterize $\geodesic{c_0c_1}$ by $t\mapsto c_t$ with $t$
    varying over $[0,1]$.  We write $p_t$ for the point of $H$ nearest
    $c_t$.  The issue we must deal with is the following. If
    $\geodesic{c_t p_t}-\{p_t\}$ meets $\H$, then varying $t$ will
    move the intersection points but not eliminate them.  This is
    visible in figure~\ref{fig-well-definedness-of-mu-b-q-H}, where
    $H'$ is a mirror giving rise to such intersections.  The method of
    section~\ref{subsec-geodesics-and-geodesic-triangles} shows that
    $c_0$, $p_0$ and $q$ span a totally real triangle $\triangle
    c_0p_0q$, and also that all the $p_t$ lie on the segment
    $\geodesic{p_0 q}$.  By the nonpositive curvature of $\BB^{13}$,
    $p_0$ is closer to $q$ than $c_0$ is.  So $\triangle c_0p_0q$ lies
    in $U$.  In particular, the mirrors meeting this triangle are among those
    containing~$q$.

    It follows that $\H$ meets $\triangle c_0p_0q$ in the union of
    finitely many geodesic segments from its vertex $q$ to the
    opposite edge $\geodesic{c_0p_0}$.  Now restrict this to the
    quadrilateral with vertices $c_0$, $c_1$, $p_1$ and $p_0$.  Its
    intersection with any mirror is either empty or a segment
    $\geodesic{u_0 u_1}$, which we may parameterize by $t\mapsto u_t$
    with $u_t\in\geodesic{c_t p_t}$ as shown in
    figure~\ref{fig-well-definedness-of-mu-b-q-H}.  If there is more
    than one such segment then they are disjoint from each other.

    Now, any mirror that meets $\geodesic{p_0 p_1}$
    contains all of $\geodesic{p_0 p_1}$.  So we may choose $\e>0$
    such that every mirror meeting the closed $\e$-neighborhood of
    $\geodesic{p_0 p_1}$ contains $\geodesic{p_0 p_1}$.  Now let $D_t$
    be the closed $\e$-disk in the complex geodesic
    $\complexgeodesic{c_t p_t}$, centered at $p_t$.  The only mirrors
    it meets are the ones containing $\geodesic{p_0 p_1}$.  All of
    these mirrors contain $p_t$, and therefore contain no other points of
    $D_t$.  So $D_t\cap\H=\{p_t\}$.
    Write $d_t$ for the point where $\geodesic{c_t p_t}$ 
    enters $D_t$.  Because $D_t$ misses $\H$ except at $p_t$, we may
    take $\mu_{c_t,H}$ to be $\dodge{c_t d_t}$, followed by the
    positively-oriented arc in $\partial D_t$ from $d_t$ to $R(d_t)$,
    followed by $R\bigl(\reverse(\dodge{c_t d_t})\bigr)$.  The $\partial D_t$
    portions of these paths vary continuously with $t$, sweeping out
    the $\frac13$-tube  in the center of
    figure~\ref{fig-well-definedness-of-mu-b-q-H}.

    The shaded surface to the left of this in figure~\ref{fig-well-definedness-of-mu-b-q-H}
    is got by modifying the
    quadrilateral with edges $\geodesic{d_1 c_1}$, $\geodesic{c_1
      c_0}$, $\geodesic{c_0 d_0}$ and $\set{d_t}{t\in[0,1]}$.  The
    modification is to replace a strip around each segment
    $\geodesic{u_0 u_1}$ (notation as above and in figure~\ref{fig-well-definedness-of-mu-b-q-H}) by a
    semi-cylindrical strip that dodges $\H$.  More precisely, we may take
    this semi-cylindrical strip to be a union of positively oriented semicircles $S_t$, where
    $S_t$ lies in the complex geodesic $\complexgeodesic{c_t p_t}$ and
    has center $u_t$ and some small constant radius.
    Each $S_t$ misses $\H$ by reasoning similar to the previous paragraph.
    Because the
    $u_t$ and the $\complexgeodesic{c_t p_t}$ vary continuously with
    $t$, the $S_t$'s do too.  

    The part of the homotopy on the right side of the figure is the
    $R$-image of the part on the left.  The left, right and middle
    parts fit together to give the promised homotopy from
    $\mu_{c_0,H}$ to $\mu_{c_1,H}$.
\end{proof}

\begin{lemma}[Homotopy between meridians under change of basepoint]
  \label{lem-homotopy-for-moving-the-basepoint}
  Suppose $H$ is a mirror, $q$ is a point
  of $H$, and that the other mirrors through $q$ are
  orthogonal to $H$.  Suppose
  $b:[0,1]\to\BB^{13}-\H$
    is a path, written $t\mapsto b_t$, such that each
  $\geodesic{b_t q}$ misses $\H$ except at $q$.  Then
  $M_{b_0,q,H}\in G_{b_0}$ is
  identified with $M_{b_1,q,H}\in G_{b_1}$ under the
  isomorphism of orbifold fundamental groups $G_{b_0}\iso G_{b_1}$
  induced by the path $b$.
\end{lemma}

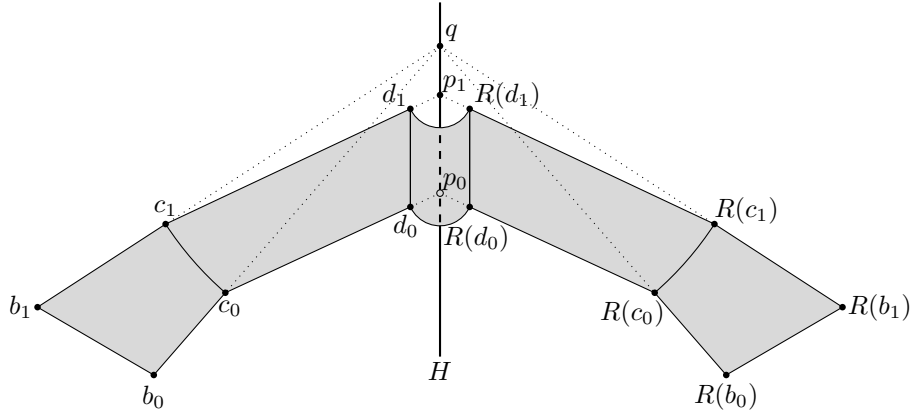
\begin{figure}
  \begin{center}
    \begin{tikzpicture}[x=1.45cm,y=1.45cm]
      \def\URx{4.4}
      \def\URy{.4}
      \def\LLx{-4}
      \def\LLy{-3.4}
      \clip(\LLx,\LLy)rectangle(\URx,\URy);
      \def\pTOdANGLE{-155}
      \def\pTOdDIST{.3}
      \def\qTObzeroANGLE{-131}
      \def\qTObzeroDIST{4}
      \def\qTOboneANGLE{-147}
      \def\qTOboneDIST{4.4}
      \def\qTOcDIST{3}
      \def\pzeroTOczeroDIST{1.4}
      \def\poneTOconeDIST{1.5}
      \def\bzeroFACTOR{-1.8}
      \def\bzeroFACTOR{-1}
      \def\boneFACTOR{-1.5}
      \coordinate (q) at (0,0);
      \coordinate (p0) at (0,-1.35);
      \coordinate (p1) at (0,-.45); 
      \path (q)++(\qTObzeroANGLE:\qTObzeroDIST) coordinate (b0);
      \path (q)++(\qTOboneANGLE:\qTOboneDIST) coordinate (b1);
      \path (q)++(\qTObzeroANGLE:\qTOcDIST) coordinate (c0);
      \path (q)++(\qTOboneANGLE:\qTOcDIST) coordinate (c1);
      \path (p0)++(\pTOdANGLE:\pTOdDIST) coordinate (d0);
      \path (p1)++(\pTOdANGLE:\pTOdDIST) coordinate (d1);
      \path (q)++(180-\qTObzeroANGLE:\qTObzeroDIST) coordinate (Rb0);
      \path (q)++(180-\qTOboneANGLE:\qTOboneDIST) coordinate (Rb1);
      \path (q)++(180-\qTObzeroANGLE:\qTOcDIST) coordinate (Rc0);
      \path (q)++(180-\qTOboneANGLE:\qTOcDIST) coordinate (Rc1);
      \path (p0)++(180-\pTOdANGLE:\pTOdDIST) coordinate (Rd0);
      \path (p1)++(180-\pTOdANGLE:\pTOdDIST) coordinate (Rd1);
      \path (p0)++(0,-1.5) coordinate (bottom);
      \path (q)++(0,.4) coordinate (top);
      \draw[thick] (top)--(bottom);
      \def\Hcolor{gray!30}
      \fill[\Hcolor] (b0)--(c0)--(d0)--(d1)--(c1)--(b1)--cycle;
      \fill[\Hcolor] (Rb0)--(Rc0)--(Rd0)--(Rd1)--(Rc1)--(Rb1)--cycle;
      \def\Hmiddlepath{%
      (p0)++(\pTOdANGLE:\pTOdDIST)
      arc (\pTOdANGLE:-180-\pTOdANGLE:\pTOdDIST)
      -- (Rd1)
      arc (-180-\pTOdANGLE:\pTOdANGLE:\pTOdDIST)
      -- (d1)--cycle}
      \fill[\Hcolor] \Hmiddlepath;
      \draw\Hmiddlepath;
      \draw (c0)--(d0)--(d1)--(c1)--(b1)--(b0)--(c0);
      \draw (Rc0)--(Rd0)--(Rd1)--(Rc1)--(Rb1)--(Rb0)--(Rc0);
      \draw (q)++(\qTObzeroANGLE:\qTOcDIST)arc(\qTObzeroANGLE:\qTOboneANGLE:\qTOcDIST);
      \draw (q)++(180-\qTObzeroANGLE:\qTOcDIST)arc(180-\qTObzeroANGLE:180-\qTOboneANGLE:\qTOcDIST);
      \begin{scope}
        \clip \Hmiddlepath;
        \draw[dashed,thick] (top)++(0,-.2)--(bottom);
      \end{scope}
      \def\dotradius{.03}
      \foreach\x in{q,p1} \fill (\x) circle (\dotradius);
      \foreach\x in{b0,b1,c0,c1,d0,d1} \fill (\x) circle (\dotradius);
      \foreach\x in{Rb0,Rb1,Rc0,Rc1,Rd0,Rd1} \fill (\x) circle (\dotradius);
      \draw [dotted] (q)--(c0);
      \draw [dotted] (q)--(c1);
      \draw [dotted] (q)--(Rc0);
      \draw [dotted] (q)--(Rc1);
      \draw [dotted] (p1)--(d1);
      \draw [dotted] (p1)--(Rd1);
      \draw [dotted] (p0)--(d0);
      \draw [dotted] (p0)--(Rd0);
      \fill[\Hcolor] (p0) circle (\dotradius); 
      \draw (p0) circle (\dotradius);
      \tikzstyle{lab}=[inner sep=0pt, outer sep=2pt]
      \draw(q)node[lab,anchor=south west]{$q$};
      \draw(p0)node[lab,anchor=south west,outer sep=1pt]{$p_0$};
      \draw(p1)node[lab,anchor=south west,outer sep=1pt]{$p_1$};
      \draw(bottom)node[lab,anchor=north]{$H$};
      \draw(b0)node[lab,anchor=north,outer sep=4pt]{$b_0$};
      \draw(b1)node[lab,anchor=east]{$b_1$};
      \draw(c0)node[lab,anchor=north,outer sep=3pt]{$\,\,c_0$};
      \draw(c1)node[lab,anchor=south,outer sep=3pt]{$c_1$};
      \draw(d0)node[lab,anchor=north,outer sep=3pt]{$d_0\,\,$};
      \draw(d1)node[lab,anchor=south east,outer sep=1pt]{$d_1$};
      \draw(Rb0)node[lab,anchor=north,outer sep=3pt]{$R(b_0)$};
      \draw(Rb1)node[lab,anchor=west]{$R(b_1)$};
      \draw(Rc0)node[lab,anchor=north east,outer sep=3pt]{$R(c_0)\!\!\!\!$};
      \draw(Rc1)node[lab,anchor=south west,outer sep=0pt]{$R(c_1)$};
      \draw(Rd0)node[lab,anchor=north,outer sep=7pt]{$\,\,\,R(d_0)$};
      \draw(Rd1)node[lab,anchor=south west,outer sep=0pt]{$\,R(d_1)$};
    \end{tikzpicture}
  \end{center}
  \caption{Illustration for lemma~\ref{lem-homotopy-for-moving-the-basepoint}.  The paths $\mu_{b_0,q,H}$
    and $\mu_{b_1,q,H}$ go left to right across the bottom and top
    respectively, and the shaded area is a homotopy between them. It 
    moves the basepoint from $b_0$ to $b_1$.
  }
  \label{fig-homotopy-for-moving-the-basepoint}
\end{figure}

\begin{proof}
The proof consists mainly of looking at
figure~\ref{fig-homotopy-for-moving-the-basepoint}.  Let $\e>0$ be
small enough that the only mirrors at distance${}\leq\e$ from $q$ are
those that contain $q$.  
For each $t\in[0,1]$, take the
point $c_t$ in the definition of $\mu_{b_t,q,H}$ to be the point of
$\geodesic{b_t q}$ that lies at distance $\e$ from $q$.
(We suppose without loss of generality that $\e< d(q,b_t)$ for all $t$,
so this construction  makes sense.)
The leftmost
region in figure~\ref{fig-homotopy-for-moving-the-basepoint}
represents the union of the paths $\geodesic{b_t c_t}$.  They all miss
$\H$ by hypothesis.  The rightmost region is its $R$-image, where $R$
still denotes the $\w$-reflection in~$H$.

Now let $p_t$ be the point of $H$ closest to $c_t$; by the nonpositive
curvature of $\BB^{13}$, $p_t$ is closer to $q$ than $c_t$ is.
So every point of every geodesic $\geodesic{c_t p_t}$ lies at
distance${}\leq\e$ from $q$.
We claim these geodesics also
miss $\H$, except that every $p_t$ lies
in~$H$.  Otherwise, some mirror $H'\neq H$ would meet some
$\geodesic{c_t p_t}$.  By the construction of $\e$, $H'$
passes through $q$.  So our orthogonality hypothesis gives $H'\perp
H$.  Because of this orthogonality, the only way $\geodesic{c_t
  p_t}$ can meet $H'$ is by lying entirely in $H'$.  This
would contradict the known fact $c_t\notin\H$.

Now let $\e'>0$ be small enough that the closed
$\e'$-neighborhood of $\set{p_t}{t\in[0,1]}$ misses every mirror
except $H$.  Call this neighborhood $U$.
By definition,
the ``middle'' part of $\mu_{b_t,q,H}$ is the path $\mu_{c_t,H}$,
whose definition involves choosing a point $d_t$ very near $p_t$.  We 
choose $d_t$ to be the point of $\geodesic{c_t p_t}$ at distance $\e'$ from
$p_t$.
(We suppose without loss of generality that $\e'< d(H,c_t)$
for every $t$, so this construction makes sense.) 
The second region in figure~\ref{fig-homotopy-for-moving-the-basepoint} is the union of the
geodesics $\geodesic{c_t d_t}$, and the second region from the
right is its $R$-image.

The $\frac13$-of-a-tube portion of the homotopy in figure~\ref{fig-homotopy-for-moving-the-basepoint} is
the surface swept out by the path $t\mapsto d_t$ under the rotations
around $H$ by angles in $[0,2\pi/3]$.  
It lies
in $U-H=U-\H$, so it misses~$\H$.  The five pieces of the homotopy fit
together to form a homotopy in $\BB^{13}-\H$ from $\mu_{b_0,q,H}$ to
$\mu_{b_1,q,H}$, in which $b_0$ moves to $b_1$ along the path~$b$.
It follows that the isomorphism $G_{b_0}\iso G_{b_1}$ induced by this
path identifies $\bigl(\mu_{b_0,q,H},R\bigr)$ with
$\bigl(\mu_{b_1,q,H},R\bigr)$.  That is, it identifies $M_{b_0,q,H}$
with $M_{b_1,q,H}$.
\end{proof}

Philosophically, we would like to take the Leech cusp $\rho$ as a basepoint for
analyzing the orbifold fundamental group of $(\BB^{13}-\H)/P\Gamma$.
But we must work around the difficulty that $\rho$ is not a point of
$\BB^{13}$.  To do this we follow the ``fat basepoint'' strategy of
\cite{AB-braidlike}, by choosing a closed horoball $A$ centered at $\rho$ and
small enough to miss~$\H$.  Lemma~4.3 of \cite{AB-braidlike} shows that we may take
$A$ to be the closed horoball of any height${}<1$.
We also choose some point $a$ of~$A$.  

We continue to suppose that $H$ is a mirror and $R$ is its
$\w$-reflection.  Now we take $p$ to be $H$'s point closest to
$\rho$. (See section~\ref{subsec-complex-hyperbolic-space} for the meaning of ``closest''.)
If the meridian were defined then it would begin with a geodesic
(possibly with small detours) from $\rho$ to a point near $p$.  This
path would pierce $\partial A$ at the point of $\partial A$ nearest to
$p$, which we will call $b$.  To define the corresponding path based
at $a$ we simply replace the segment $\geodesic{\rho b}$ by
$\geodesic{ab}$.  Formally, $\mu_{a,A,H}$ is $\geodesic{a b}$ followed
by $\mu_{b,H}$ followed by $R\bigl(\geodesic{ba}\bigr)$.  See
figure~\ref{fig-homotopy-for-moving-the-point-on-the-hyperplane} for a
picture: $\mu_{a,A,H}$ travels from left to right across the bottom.
The point marked~$d$ comes from the definition of~$\mu_{b,H}$.

Similarly, if $q$ is any point of $H$ then we define $b'$ as the point
of $\partial A$ nearest to $q$, and then define $\mu_{a,A,q,H}$ as
$\geodesic{a b'}$ followed by $\mu_{b'\!,q,H}$ followed by
$R\bigl(\geodesic{b'a}\bigr)$.  This travels from left to right across
the top of
figure~\ref{fig-homotopy-for-moving-the-point-on-the-hyperplane}.  The
point marked $c'$ comes from the definition of $\mu_{b'\!,q,H}$, and
the points marked $d'$ and $p'$ come from the definition of the
subpath $\mu_{c'\!,H}$ of $\mu_{b'\!,q,H}$.  Just as for ordinary
basepoints, we refer to the orbifold fundamental group elements
$M_{a,A,H}=(\mu_{a,A,H},R)$ and $M_{a,A,q,H}=(\mu_{a,A,q,H},R)$ as
meridians.

\begin{numberedremark}
\label{rem-change-a-without-loss}
  In several places it will clarify an argument to assume ``without
  loss of generality'' that $a$ is some particular point of $A$.  One
  should expect that the exact location of $a$ is unimportant: if $a'$
  is another point of $A$ then $G_a$ and $G_{a'}$ are canonically
  isomorphic.  One just moves the basepoint from $a$ to $a'$ along any
  path in the simply connected space~$A$.  But more is needed: this
  should identify each meridian $M_{a,A,q,H}\in G_a$ with
  $M_{a',A,q,H}\in G_{a'}$.  It is easy to see that this holds.  In
  fact, arranging for it to be true was the main reason for
  introducing $A$ into the definition of the meridians.  This was
  necessary because the canonical isomorphism $G_a\iso G_{a'}$ does
  {\it not\/} identify every $M_{a,q,H}$ with the corresponding
  $M_{a',q,H}$.
\end{numberedremark}

\begin{numberedremark}
  There is a more elegant treatment of fat basepoints in
  \cite[Appendix~A]{Bessis}.  This treats $A$ itself as the basepoint
  and avoids mention of $a$.  But this approach is not well-suited for
  moving the basepoint from $a$ to the $26$-point~$\tau$ (that lies
  outside $A$).  We do this in section~\ref{sec-change-of-basepoint},
  and it is an essential part of our proof of
  theorem~\ref{t-26-meridians-based-at-tau-generate}.
\end{numberedremark}

\begin{lemma}[Homotopies between meridians based ``at'' a cusp]
  \label{lem-homotopy-for-moving-the-point-on-the-hyperplane-cusp-version}
  Suppose $\rho\in\partial\BB^{13}$ is a Leech cusp, $A$ is a closed horoball centered at
  $\rho$ and disjoint from $\H$, and $a$ is a basepoint in $A$.  Let
  $H$ be a mirror, $p$ be the
  point of $H$ closest to $\rho$, and $q$ be another point of $H$.
  Our key assumption is that the totally real triangle $\triangle \rho
  p q$ is
  disjoint from $\H$ except that $\geodesic{p q}$ lies in $H$ and
  that $q$ might lie in additional mirrors.  Then
  $\mu_{a,A,H}$ and $\mu_{a,A,q,H}$ are homotopic rel
  endpoints in $\BB^{13}-\H$, and  $M_{a,A,H}=M_{a,A,q,H}$ in~$G_a$.
\end{lemma}
 
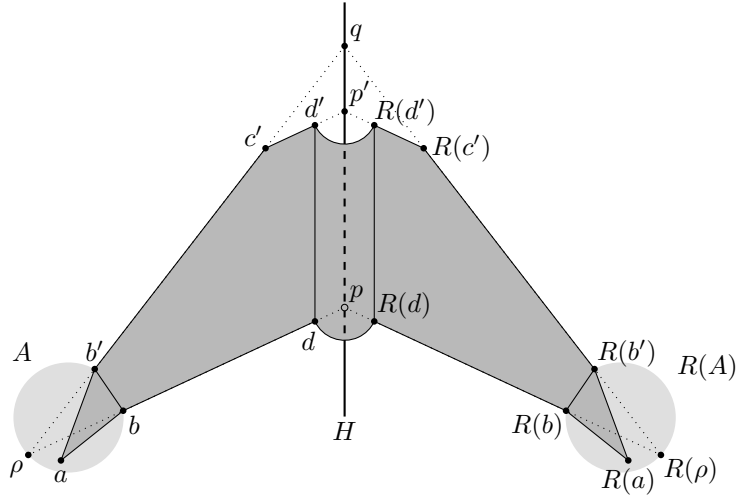
\begin{figure}
  \def\declarecenter#1#2#3#4{%
    \coordinate (middle1) at ($(#1)!.5!(#2)$);%
    \coordinate (middle2) at ($(#2)!.5!(#3)$);%
    \coordinate (aux1) at ($(middle1)!1!90:(#2)$);%
    \coordinate (aux2) at ($(middle2)!1!90:(#3)$);%
    \coordinate (#4) at ($(intersection of middle1--aux1 and middle2--aux2)$);%
  }
  \begin{center}
    \begin{tikzpicture}[x=1.45cm,y=1.45cm]
      \def\URx{3.6}
      \def\URy{.4}
      \def\LLx{-3.1}
      \def\LLy{-4.2}
      \clip(\LLx,\LLy)rectangle(\URx,\URy);
      \def\similarityfactor{4}
      \def\qTOpprimeDIST{.6}
      \def\pTOdANGLE{-155}
      \def\pTOdDIST{.3}
      \def\pprimeTOcprimeDIST{.8}
      \coordinate (q) at (0,0);
      \path (q)++(0,-\qTOpprimeDIST) coordinate (pprime);
      \path (q)++(0,-\qTOpprimeDIST*\similarityfactor) coordinate (p);
      \path (p)++(\pTOdANGLE:\pTOdDIST) coordinate (d);
      \path (pprime)++(\pTOdANGLE:\pTOdDIST) coordinate (dprime);
      \path (pprime)++(\pTOdANGLE:\pprimeTOcprimeDIST) coordinate (cprime);
      \path (p)++(\pTOdANGLE:\pprimeTOcprimeDIST*\similarityfactor) coordinate (rho);
      \coordinate (b) at ($(rho)!.3!(p)$);
      \coordinate (bprime) at ($(rho)!.21!(q)$);
      \declarecenter{rho}{b}{bprime}{A}
      \path (A)++(-100:.4) coordinate (a);
      \path (p)++(180-\pTOdANGLE:\pTOdDIST) coordinate (Rd);
      \path (pprime)++(180-\pTOdANGLE:\pTOdDIST) coordinate (Rdprime);
      \path (pprime)++(180-\pTOdANGLE:\pprimeTOcprimeDIST) coordinate (Rcprime);
      \path (p)++(180-\pTOdANGLE:\pprimeTOcprimeDIST*\similarityfactor)coordinate(Rrho);
      \coordinate (Rb) at ($(Rrho)!.3!(p)$);
      \coordinate (Rbprime) at ($(Rrho)!.21!(q)$);
      \declarecenter{Rrho}{Rb}{Rbprime}{RA}
      \path (RA)++(-80:.4) coordinate (Ra);
      \def\Acolor{gray!25}
      \fill[\Acolor] let \p1=($(b)-(A)$), \n0={veclen(\p1)} in (A) circle (\n0);
      \fill[\Acolor] let \p1=($(Rb)-(RA)$), \n0={veclen(\p1)} in (RA) circle (\n0);
      \path (p)++(0,-1) coordinate (bottom);
      \path (q)++(0,.4) coordinate (top);
      \draw[thick] (top)--(bottom);
      \def\Hcolor{gray!55}
      \fill[\Hcolor] (a)--(bprime)--(cprime)--(dprime)--(d)--(b)--cycle;
      \fill[\Hcolor] (Ra)--(Rbprime)--(Rcprime)--(Rdprime)--(Rd)--(Rb)--cycle;
      \def\Hmiddlepath{%
      (p)++(\pTOdANGLE:\pTOdDIST)
      arc (\pTOdANGLE:-180-\pTOdANGLE:\pTOdDIST)
      -- (Rdprime)
      arc (-180-\pTOdANGLE:\pTOdANGLE:\pTOdDIST)
      -- (d)}
      \fill[\Hcolor] \Hmiddlepath;
      \draw\Hmiddlepath;
      \draw (d)--(b)--(a)--(bprime)--(cprime)--(dprime);
      \draw (Rd)--(Rb)--(Ra)--(Rbprime)--(Rcprime)--(Rdprime);
      \draw (b)--(bprime);
      \draw (Rb)--(Rbprime);
      \begin{scope}
        \clip \Hmiddlepath;
        \draw[dashed,thick] (top)++(0,-.2)--(bottom);
      \end{scope}
      \def\dotradius{.03}
      \foreach\x in{q,pprime} \fill (\x) circle (\dotradius);
      \foreach\x in{d,cprime,dprime,a,b,bprime,rho} \fill (\x) circle (\dotradius);
      \foreach\x in{Rd,Rcprime,Rdprime,Ra,Rb,Rbprime,Rrho}\fill(\x) circle (\dotradius);
      \draw [dotted] (p)--(rho);
      \draw [dotted] (p)--(Rrho);
      \draw [dotted] (q)--(rho);
      \draw [dotted] (q)--(Rrho);
      \draw [dotted] (pprime)--(dprime);
      \draw [dotted] (pprime)--(Rdprime);
      \fill[\Hcolor] (p) circle (\dotradius); 
      \draw (p) circle (\dotradius);
      \tikzstyle{lab}=[inner sep=0pt, outer sep=2pt]
      \draw(p)node[lab,anchor=south west]{$p$};
      \draw(pprime)node[lab,anchor=south west]{$p'$};
      \draw(q)node[lab,anchor=south west]{$q$};
      \draw(bottom)node[lab,anchor=north]{$H$};
      \draw(A)++(120:.5)node[anchor=south east]{$A$};
      \draw(a)node[lab,anchor=north,outer sep=4pt]{$a$};
      \draw(rho)node[lab,anchor=north east]{$\rho$};
      \draw(b)node[lab,anchor=north west]{$b$};
      \draw(bprime)node[lab,anchor=south,outer sep=3pt]{$b'$};
      \draw(d)node[lab,anchor=north,outer sep=4pt]{$d\,\,\,$};
      \draw(cprime)node[lab,anchor=south east,outer sep=1pt]{$c'$};
      \draw(dprime)node[lab,anchor=south,outer sep=3pt]{$d'$};
      \draw(RA)++(30:.5)node[anchor=south west]{$R(A)$};
      \draw(Ra)node[lab,anchor=north,outer sep=3pt]{$R(a)$};
      \draw(Rrho)node[lab,anchor=north west,outer sep=1pt]{$R(\rho)$};
      \draw(Rb)node[lab,anchor=north east,outer sep=1pt]{$R(b)$};
      \draw(Rbprime)node[lab,anchor=south west,outer sep=1pt]{$R(b')$};
      \draw(Rd)node[lab,anchor=south west,outer sep=1pt]{$R(d)$};
      \draw(Rcprime)node[lab,anchor=west,outer sep=3pt]{$R(c')$};
      \draw(Rdprime)node[lab,anchor=south west,outer sep=0pt]{$R(d')$};
    \end{tikzpicture}
  \end{center}
  \caption{Illustration of the meridians $\mu_{a,A,H}$, left to right
    across the bottom, and $\mu_{a,A,q,H}$, left to right across the
    top.  The dark region indicates lemma~\ref{lem-homotopy-for-moving-the-point-on-the-hyperplane-cusp-version}'s homotopy between
    them.  The light gray balls are horoballs centered at Leech
    cusps $\rho$ and $R(\rho)$.
  }
  \label{fig-homotopy-for-moving-the-point-on-the-hyperplane}
\end{figure}

\begin{proof}
  The proof mostly consists of looking at
  figure~\ref{fig-homotopy-for-moving-the-point-on-the-hyperplane},
  which uses the notation $b,d,b',c',d',p',R$ from the definitions
  given above.  The paths $\mu_{a,A,q,H}$ and $\mu_{a,A,H}$ are the
  upper and lower paths around the dark gray region, traversed from
  left to right.  This region indicates the homotopy between them that
  we will construct.

  By hypothesis, for small enough $\e>0$, the $\e$-neighborhood of
  $\geodesic{p p'}$ meets no mirrors except~$H$.  This neighborhood,
  which we call $U$, is convex by the nonpositive curvature of
  $\BB^{13}$ (see the remark after Prop.~II.2.4 of \cite{Bridson-Haefliger}).  By moving $d$ closer to $p$ and $d'$ closer to
  $p'$ we may suppose without loss of generality that they lie in $U$.
  The portion of the homotopy that appears in the figure as $\frac13$
  of a tube is the surface swept out by $\geodesic{d d'}$ under the
  rotations around $H$ by angles in $[0,2\pi/3]$.  It lies in $U$
  since $\geodesic{d d'}$ does.  It misses $\H$ because it misses $H$,
  which is the only mirror that meets~$U$.  The pentagon with vertices
  $b'$, $c'$, $d'$, $d$ and $b$ misses $\H$ because it lies in
  $\triangle \rho p q-\geodesic{p q}$.  (Recall from section~\ref{subsec-geodesics-and-geodesic-triangles} that
  $\triangle \rho p q$ is totally geodesic.)  The triangle with vertices
  $a$, $b$ and $b'$ misses $\H$ because it lies in $A$.  By symmetry,
  the $R$-images of this pentagon and triangle also miss
  $\H$.  The $\frac13$-tube, two pentagons and two triangles
  fit together to give a homotopy between $\mu_{a,A,H}$ and
  $\mu_{a,A,q,H}$.
\end{proof}

We will also need a version of this where the basepoint is an ordinary
point of $\BB^{13}-\H$ rather than a Leech cusp.  The picture and proof differ from
lemma~\ref{lem-homotopy-for-moving-the-point-on-the-hyperplane-cusp-version}
only by replacing
all of $\rho$, $A$, $a$, $b$ and $b'$ by a single point $b$.

\begin{lemma}[Homotopies between meridians based at a point of $\BB^{13}$]
  \label{lem-homotopy-for-moving-the-point-on-the-hyperplane-ordinary-basepoint-version}
  Suppose $b\in\BB^{13}-\H$.  Let
  $H$ be a mirror, $p$ be the
  point of $H$ closest to $b$, and $q$ be another point of $H$.
  Suppose that the totally real triangle $\triangle b p q$ is
  disjoint from $\H$ except that $\geodesic{p q}$ lies in $H$ and
  that $q$ might lie in additional mirrors.  Then
  $\mu_{b,H}\homotopic\mu_{b,q,H}$  rel
  endpoints in $\BB^{13}-\H$, and
  $M_{b,H}=M_{b,q,H}$ in~$G_b$.
  \qed
\end{lemma}

\section{Finitely many generators based ``at'' a cusp}
\label{sec-finitely-many-generators-based-at-a-cusp}

\noindent
In this section we begin the proof of our main result,
theorem~\ref{t-26-meridians-based-at-tau-generate}.  Our starting
point is theorem~1.4 from \cite{AB-braidlike}, which states that the
set of all Leech meridians generates
$G_a=\piorb\bigl((\BB^{13}-\H)/P\Gamma,a\bigr)$, where $a$ is a
basepoint chosen near a Leech cusp $\rho$.  Recall from
section~\ref{subsec-Leech-cusps-and-Leech-roots} that the 
Leech roots are the roots $r$ satisfying $\ip{\rho}{r} = \theta$ 
and Leech
meridians are the meridians $M_{a,A,s}$ with $s$ varying over the
Leech roots, where $A$ is a closed horoball centered at~$\rho$,
disjoint from~$\H$ and containing~$a$.  This is an infinite
generating set, which we improve in
theorem~\ref{thm-130-meridians-generate} by exhibiting an explicit
finite set of Leech meridians that generates $G_a$.  The exact nature
of this generating set is not so important, but it must be explicit
and finite so that we will be able to prove
theorem~\ref{thm-26-meridians-at-cusp-generate} in the next section.

\begin{theorem}[$130$ generators based ``at'' a cusp]
\label{thm-130-meridians-generate}
Suppose $\rho=(3\w-1;-1,\dots,-1)$
is the Leech cusp from section~\ref{subsec-Leech-cusps-and-Leech-roots}, $A$ is a closed horoball centered at
$\rho$ and disjoint from $\H$, and $a$ is an element of $A$.  
Write $S$ for the set of $130$ Leech roots 
\begin{equation*}
  p_i,
  \qquad
  p_i-\rho,
  \qquad
  \wbar p_i-l_j,
  \qquad
  \wbar p_i-l_j-\rho
\end{equation*}
with $i,j=1,\dots,13$ and
$p_i$ and $l_j$  incident in the last two cases.
Then $G_a$ is generated by the Leech meridians
$M_{a,A,s}$, where $s$ varies over~$S$.
\end{theorem}

In this section we will write $G$ for the subgroup of $G_a$ generated
by these particular Leech meridians, and our goal is to show that $G$
is all of $G_a$.  The strategy is to show that $G$ contains some
isometries of $\BB^{13}$ that fix $\rho$ and act transitively on the
Leech roots.  It follows from this that every Leech meridian is
$G$-conjugate into $G$, hence lies in $G$.  Since the set of all Leech
meridians generates $G_a$, it follows that $G$ is all of $G_a$, as
desired.

But $G_a$ is an orbifold fundamental group, not a group of isometries
of $\BB^{13}$, so we must say what we mean by $G_a$ ``containing''
some isometries of $\BB^{13}$.  We will explain this and then develop the details of the
strategy just outlined.  The idea is that the subgroup
$P\Gamma_{\!A}$ of $P\Gamma\sset\Isom(\BB^{13})$ preserving $A$ can be
thought of as a subgroup
of $G_a$.  To see this, consider the  obvious homomorphism
\begin{equation}
  \label{eq-P-Gamma-A-as-a-subgroup-of-G-a}
  P\Gamma_{\!A}
  =
  \piorb(A/P\Gamma_{\!A},a)
  \to
  \piorb\bigl((\BB^{13}-\H)/P\Gamma,a\bigr)
  =G_a.
\end{equation}
The first equality holds because $A$ is simply connected, and the map
exists because $A$ misses $\H$.  This map $P\Gamma_{\!A}\to G_a$ is an
embedding, because following it by the natural map $G_a\to P\Gamma$
(see section~\ref{subsec-orbifold-fundamental-group}) gives the
identity on $P\Gamma_{\!A}$.  In this manner, we may regard
$P\Gamma_{\!A}$ as a subgroup of $G_a$.  Concretely, if $g\in
P\Gamma_{\!A}$, then the corresponding element of $G_a$ is
$(\gamma,g)$ where $\gamma$ is any path in $A$ from $a$ to $g(a)$.
The choice of $\gamma$ is unimportant since $A$ is simply connected.

The elements of $P\Gamma_{\!A}$ we will need are the following
``translations'' $T_{\lambda,z}$. The language comes from their
analogy with the isometries of real hyperbolic space which look like
Euclidean translations in
the upper half space model.  Given $\lambda\in\Leech$ and $z\in\Im\C$ such that
$z-\lambda^2/2\in\theta\E$, the definition is
  \begin{align}
    &(\phantom{0}\llap{$l$};0,0)\mapsto\bigl(l;0,\thetabar^{-1}\ip{l}{\lambda}\bigr)
    \notag\\
    \label{eq-definition-of-translation}
    \llap{$T_{\lambda,z}:$ }&
    (0;1,0)\mapsto\bigl(\lambda;1,\theta^{-1}(z-\lambda^2/2)\bigr)\\
    &(0;0,1)\mapsto(0;0,1)
    \notag
  \end{align}
One can see that the
translations form a Heisenberg group by checking the identities:
  \begin{align}
    \label{eq-product-of-translations}
    T_{\lambda,z}T_{\lambda',z'}&{}=T_{\lambda+\lambda',z+z'+\Im\ip{\lambda}{\lambda'}}
    \\
    \label{eq-inversion-in-group-of-translations}
    \bigl(T_{\lambda,z})^{-1}&{}=T_{-\lambda,-z}
    \\
    \label{eq-commutator-of-translations}
    T_{\lambda,z}T_{\lambda',z'}T_{\lambda,z}^{-1}T_{\lambda',z'}^{-1}
    &{}=T_{0,2\Im\ip{\lambda}{\lambda'}}
  \end{align}
  This group preserves~$A$ and~$\rho$, and acts simply transitively on
  the Leech roots. 
  The center of this group consists of $T_{0, z}$ given
  by $T_{0,z} (v; 1 , \alpha) = (v; 1, \alpha + \theta^{-1} z)$ where 
  $z \in \theta \Z$.
  The next lemma, which is the crucial trick in this
  section, obtains some elements of $P\Gamma_{\!A}$ as products of
  meridians.  Then we combine these in
  lemma~\ref{lem-Heisenberg-trick-foundation} to obtain translations.

\begin{lemma}
\label{lem-homotopy-into-horoball}
Suppose $\rho\in L$ is a Leech cusp, $A$ is a closed horoball centered
at $\rho$ and disjoint from $\H$, and $a\in A$.  Suppose $s$ is a
Leech root, and write $s'$ for the Leech root $s-\rho$.  Write $H$ and
$H'$ for their mirrors and $R$ and $R'$ for their $\w$-reflections.
Then
\begin{equation}
  \label{eq-homotopy-into-horoball}
  M_{a,A,H'}
  \cdot
  M_{a,A,H}
  =
  R'\cdot R.
\end{equation}
Here the product of meridians on the left is evaluated in the orbifold
fundamental group $G_a$, and the product of reflections on
the right is evaluated in $P\Gamma$, with result in $P\Gamma_{\!A}$.   When stating equality we are regarding $P\Gamma_{\!A}$ as a
subgroup of $G_a$ via \eqref{eq-P-Gamma-A-as-a-subgroup-of-G-a}.
\end{lemma}

\def\applyR#1#2{
  \pgfmathparse{(1-(#1))*(1-(#1))+((#2)*(#2))}
  \let\denom\pgfmathresult
  \pgfmathparse{(1-(#1))/\denom}
  \let\resultx\pgfmathresult
  \pgfmathparse{(#2)/\denom}
  \let\resulty\pgfmathresult
}

\def\URx{1.1}
\def\URy{1.8}
\def\LLx{-2.6}
\def\LLy{-.15}
\begin{figure}
  \begin{center}
    \begin{tikzpicture}[x=3cm,y=3cm,every to/.style={hyperbolic plane}]
      \clip(\LLx,\LLy)rectangle(\URx,\URy);
      \def\dotradius{.015}
      \def\Hpx{-.5}
      \def\Hpy{.866}
      \def\Hx{.5}
      \def\Hy{\Hpy}
      \def\RpHx{-1.5}
      \def\RpHy{\Hpy}
      \coordinate(Hp)at(\Hpx,\Hpy);
      \coordinate(H)at(\Hx,\Hy);
      \draw[thick] (\Hpx,\URy)--(Hp);
      \draw[thick] (\Hx,\URy)--(H);
      \draw[thick] (-1,0) arc[radius=1,start angle=180,end angle=120];
      \draw[thick] (0,0) arc[radius=1,start angle=180,end angle=120];
      \def\Acolor{gray!40}
      \def\Ay{1.3}
      \fill[color=\Acolor] (\LLx,\Ay) rectangle (\URx,\URy);
      \def\ax{0}
      \def\ay{1.6}
      \coordinate(a) at (\ax,\ay);
      \coordinate(ashifted) at (\ax-2,\ay);
      \fill(a)circle(\dotradius);
      \draw[thick](a)to(\Hpx,\Ay);
      \draw[thick](a)to(\Hx,\Ay);
      \def\fillradius{.08}
      \def\detourradius{.0825}
      \fill[color=white] (Hp) ++ (-\fillradius,0) arc[radius=\fillradius,start
        angle=180,end angle=225] --(Hp);
      \fill[color=white] (Hp) ++ (-\fillradius,0) arc[radius=\fillradius,start
        angle=180,end angle=60] --(Hp);
      \draw[thick] (Hp)++(0,\detourradius) arc[radius=\detourradius, start
        angle=90, end angle=214];
      \fill(Hp)circle(\dotradius);
      \fill[color=white] (H) ++ (-\fillradius,0) arc[radius=\fillradius,start
        angle=180,end angle=225] --(H);
      \fill[color=white] (H) ++ (-\fillradius,0) arc[radius=\fillradius,start
        angle=180,end angle=60] --(H);
      \draw[thick] (H)++(0,\detourradius) arc[radius=\detourradius, start
        angle=90, end angle=214];
      \fill(H)circle(\dotradius);
      \pgfmathparse{.5/\Ay}
      \let\smallradius\pgfmathresult
      \fill[color=\Acolor] (0,\smallradius) circle[radius=\smallradius];
      \fill[color=\Acolor] (-1,\smallradius) circle[radius=\smallradius];
      \def\dashsize{2pt}
      \def\gapsize{2pt}
      \begin{scope}
        \pgfsetdash{{\dashsize}{\gapsize}}{1pt}
        \draw (\Hpx,\URy)--(\Hpx,\Hpy);
        \draw (\Hx,\URy)--(\Hx,\Hpy);
        \draw (-2.5,\URy)--(-2.5,\Hpy);
        \draw (\RpHx,\URy)--(\RpHx,\Hpy);
        \pgfsetdash{{\dashsize}{\gapsize}}{1pt}
        \draw (0,0) arc[radius=1,start angle=0,end angle=60];
        \draw (0,0) arc[radius=1,start angle=180,end angle=120];
        \draw (-1,0) arc[radius=1,start angle=0,end angle=60];
        \draw (-1,0) arc[radius=1,start angle=180,end angle=120];
        \draw (-2,0) arc[radius=1,start angle=0,end angle=60];
        \draw (-2,0) arc[radius=1,start angle=180,end angle=120];
        \draw (0,0) arc[radius=.333,start angle=180,end angle=60];
        \draw (1,0) arc[radius=.333,start angle=0,end angle=120];
        \draw (1,0) arc[radius=1,start angle=0,end angle=60];
      \end{scope}
      \applyR{\Hx}{\Ay}
      \let\Rbx\resultx
      \let\Rby\resulty
      \coordinate(Rb) at (\Rbx,\Rby);
      \applyR{\ax}{\ay}
      \coordinate(Ra) at (\resultx,\resulty);
      \draw[thick](Ra)to(Rb);
      \fill(Ra)circle(\dotradius);
      \applyR{\ax+1}{\ay}
      \coordinate(Rpa) at (\resultx-1,\resulty);
      \fill(Rpa)circle(\dotradius);
      \draw[thick](Rpa)to(\Rbx-1,\Rby);
      \draw (\LLx,0)--(\URx,0);
      \fill (1,0) circle (\dotradius);
      \fill (0,0) circle (\dotradius);
      \fill (-1,0) circle (\dotradius);
      \fill (-2,0) circle (\dotradius);
      \tikzstyle{lab}=[inner sep=0pt, outer sep=2pt]
      \draw(a)node[lab,anchor=south]{$a$};
      \draw(Hp)node[lab,anchor=south west]{$H'$};
      \draw(H)node[lab,anchor=south west]{$H$};
      \draw(Rpa)node[lab,anchor=north]{$R'(a)$};
      \draw(Ra)node[lab,anchor=west]{$R(a)$};
      \draw(0,0)node[lab,anchor=north]{$0\rlap{${}=R(\rho)$}$};
      \draw(-1,0)node[lab,anchor=north]{$-1\rlap{${}=R'(\rho)$}$};
      \draw(-2,0)node[lab,anchor=north]{$-2$};
      \draw(1,0)node[lab,anchor=north]{$1$};
      \pgfmathparse{(\Ay+\Hpy)/2}
      \let\arrowy\pgfmathresult
      \draw(\Hpx,\arrowy)node[lab,anchor=west]{$\mu_{a,A,H'}$};
      \draw(\Hx,\arrowy)node[lab,anchor=west]{$\mu_{a,A,H}$};
      \draw[thick,-latex](\Hpx,\Ay)--(\Hpx,\arrowy);
      \draw[thick,-latex](\Hx,\Ay)--(\Hx,\arrowy);
    \end{tikzpicture}
  \end{center}
  \caption{The Leech meridians appearing in \eqref{eq-homotopy-into-horoball}.  The cusp
    $\rho$ is at infinity and the hyperplanes $H$ and $H'$
    appear as the points where they meet the pictured $\BB^1$.  The top gray region is the
    horoball $A$, and the regions below it are its
    images under $R$ and~$R'$.}
  \label{fig-homotopy-into-horoball-preparation}
\end{figure}
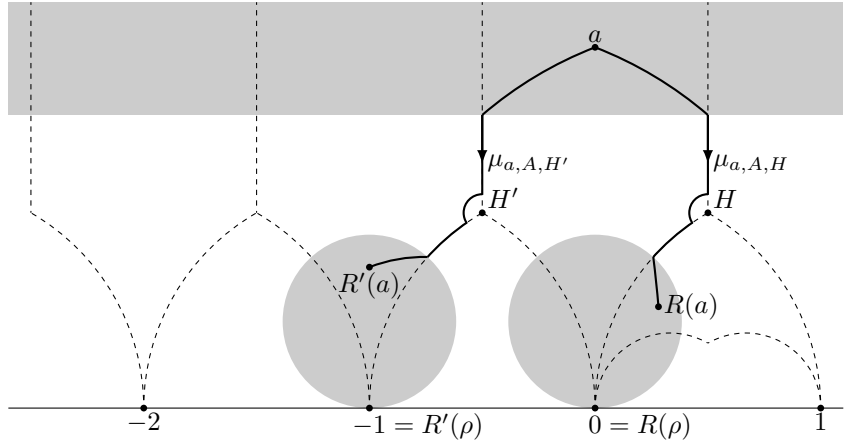

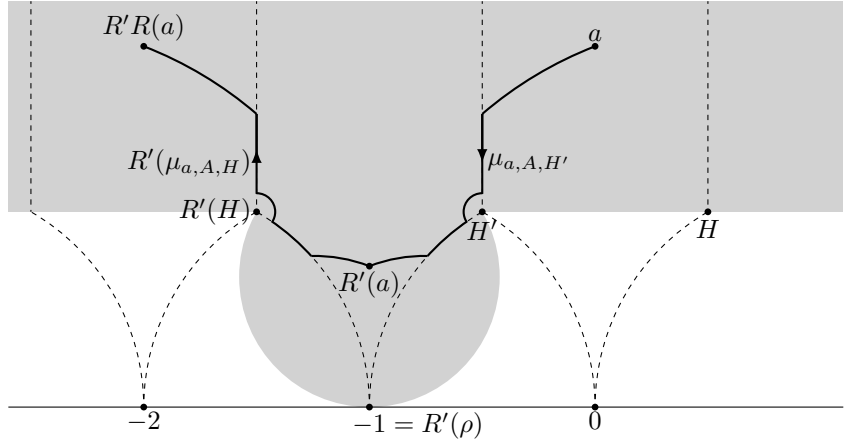
\begin{figure}
  \begin{center}
    \begin{tikzpicture}[x=3cm,y=3cm,every to/.style={hyperbolic plane}]
      \clip(\LLx,\LLy)rectangle(\URx,\URy);
      \def\dotradius{.015}
      \def\Hpx{-.5}
      \def\Hpy{.866}
      \def\Hx{.5}
      \def\Hy{\Hpy}
      \def\RpHx{-1.5}
      \def\RpHy{\Hpy}
      \coordinate(Hp)at(\Hpx,\Hpy);
      \coordinate(H)at(\Hx,\Hy);
      \coordinate(RpH)at(\RpHx,\RpHy);
      \fill[color=gray!35] (\LLx,\Hpy) rectangle (\URx,\URy);
      \pgfmathparse{\Hpy*2/3}
      \let\bigradius\pgfmathresult
      \fill[color=gray!35] (-1,\bigradius) circle[radius=\bigradius];
      \draw[thick] (\Hpx,\URy)--(Hp);
      \draw[thick] (\RpHx,\URy)--(RpH);
      \draw[thick] (-1,0) arc[radius=1,start angle=0,end angle=60];
      \draw[thick] (-1,0) arc[radius=1,start angle=180,end angle=120];
      \def\Ay{1.3}
      \fill[color=gray!35] (\LLx,\Ay) rectangle (\URx,\URy);
      \def\ax{0}
      \def\ay{1.6}
      \coordinate(a) at (\ax,\ay);
      \coordinate(ashifted) at (\ax-2,\ay);
      \fill(a)circle(\dotradius);
      \fill(ashifted)circle(\dotradius);
      \draw[thick](a)to(\Hpx,\Ay);
      \draw[thick](ashifted)to(\RpHx,\Ay);
      \def\fillradius{.08}
      \def\detourradius{.0825}
      \fill[color=gray!35] (Hp) ++ (-\fillradius,0) arc[radius=\fillradius,start
        angle=180,end angle=225] --(Hp);
      \fill[color=gray!35] (Hp) ++ (-\fillradius,0) arc[radius=\fillradius,start
        angle=180,end angle=60] --(Hp);
      \draw[thick] (Hp)++(0,\detourradius) arc[radius=\detourradius, start
        angle=90, end angle=214];
      \fill(Hp)circle(\dotradius);
      \fill(H)circle(\dotradius);
      \fill[color=gray!35] (RpH) ++ (\fillradius,0) arc[radius=\fillradius,start
        angle=0,end angle=-45] --(RpH);
      \fill[color=gray!35] (RpH) ++ (\fillradius,0) arc[radius=\fillradius,start
        angle=0,end angle=120] --(RpH);
      \draw[thick] (RpH)++(0,\detourradius) arc[radius=\detourradius, start
        angle=90, end angle=-34];
      \fill(RpH)circle(\dotradius);
      \pgfmathparse{.5/\Ay}
      \let\smallradius\pgfmathresult
      \fill[color=gray!35] (-1,\smallradius) circle[radius=\smallradius];
      \def\dashsize{2pt}
      \def\gapsize{2pt}
      \begin{scope}
        \pgfsetdash{{\dashsize}{\gapsize}}{1pt}
        \draw (\Hpx,\URy)--(\Hpx,\Hpy);
        \draw (\Hx,\URy)--(\Hx,\Hpy);
        \draw (-2.5,\URy)--(-2.5,\Hpy);
        \draw (\RpHx,\URy)--(\RpHx,\Hpy);
        \pgfsetdash{{\dashsize}{\gapsize}}{1pt}
        \draw (0,0) arc[radius=1,start angle=0,end angle=60];
        \draw (0,0) arc[radius=1,start angle=180,end angle=120];
        \draw (-1,0) arc[radius=1,start angle=0,end angle=60];
        \draw (-1,0) arc[radius=1,start angle=180,end angle=120];
        \draw (-2,0) arc[radius=1,start angle=0,end angle=60];
        \draw (-2,0) arc[radius=1,start angle=180,end angle=120];
      \end{scope}
      \applyR{\Hx}{\Ay}
      \let\Rbx\resultx
      \let\Rby\resulty
      \coordinate(Rb) at (\Rbx,\Rby);
      \applyR{\ax+1}{\ay}
      \coordinate(Rpa) at (\resultx-1,\resulty);
      \fill(Rpa)circle(\dotradius);
      \draw[thick](Rpa)to(\Rbx-1,\Rby);
      \draw[thick](Rpa)to(-\Rbx-1,\Rby);
      \draw (\LLx,0)--(\URx,0);
      \fill (0,0) circle (\dotradius);
      \fill (-1,0) circle (\dotradius);
      \fill (-2,0) circle (\dotradius);
      \tikzstyle{lab}=[inner sep=0pt, outer sep=2pt]
      \draw(a)node[lab,anchor=south]{$a$};
      \draw(\ax-2,\ay)node[lab,anchor=south]{$R'R(a)$};
      \draw(Hp)node[lab,anchor=north,outer sep=3pt]{$H'$};
      \draw(H)node[lab,anchor=north,outer sep=3pt]{$H$};
      \draw(RpH)node[lab,anchor=east]{$R'(H)$};
      \draw(Rpa)node[lab,anchor=north]{$R'(a)$};
      \draw(0,0)node[lab,anchor=north]{$0$};
      \draw(-1,0)node[lab,anchor=north]{$-1\rlap{${}=R'(\rho)$}$};
      \draw(-2,0)node[lab,anchor=north]{$-2$};
      \pgfmathparse{(\Ay+\Hpy)/2}
      \let\arrowy\pgfmathresult
      \draw(\Hpx,\arrowy)node[lab,anchor=west]{$\mu_{a,A,H'}$};
      \draw(\RpHx,\arrowy)node[lab,anchor=east]{$R'(\mu_{a,A,H})$};
      \draw[thick,-latex](\Hpx,\Ay)--(\Hpx,\arrowy);
      \draw[thick,-latex reversed](\RpHx,\Ay)--(\RpHx,\arrowy);
    \end{tikzpicture}
  \end{center}
  \caption{The proof of lemma~\ref{lem-homotopy-into-horoball}.  The gray region is the
    union of the open height~$1$ horoballs around $\rho$ and $R(\rho)$.
    Since these horoballs miss $\H$, the path from $a$ to $R'R(a)$ is
    homotopic into $A$.}
  \label{fig-homotopy-into-horoball-punchline}
\end{figure}

\begin{proof}
We use the Leech model of $L$, with $\rho=(0;0,1)$ as usual.  By the
transitivity of translations on Leech roots, we may suppose without
loss of generality that $s=(0;1,-\w)$, so $s'=(0;1,\wbar)$.  Consider
the subspace $V=\{(0;x,y)\}$ of $L\tensor\C$.  We plot the point in
$P(V)\iso P^1$ represented by $(0;x,y)$ as $y/x\in\C\cup\{\infty\}$.
Using this, we have drawn the situation in
figure~\ref{fig-homotopy-into-horoball-preparation}.  In particular,
$\rho$ corresponds to the point at infinity and $\BB^1=\BB(V)$ to the
upper half plane.  The orthogonal complements in $V$ of $s$ and $s'$
are spanned by $(0;1,-\wbar)$ and $(0;1,\w)$ respectively.  So
$H=s^\perp$ and $H'={s'}^\perp$ correspond to $-\wbar$ and $\w$ in the
upper half plane.  The triflections $R'$ and $R$ act by
counterclockwise hyperbolic rotations by $2\pi/3$ around these points.
It is easy to see that $R(\infty)=0$, $R'(0)=\infty$ and
$R'(\infty)=-1$.  In particular, $R'R\in P\Gamma_{\!A}$, which is part
of the lemma.  The closed horoball $A$ and its images under $R$ and
$R'$ also appear in
figure~\ref{fig-homotopy-into-horoball-preparation}.  To draw the
figure we have assumed that $a$, a priori any point of $A$, actually
lies in $\BB^1$ and has real part~$0$.  This is harmless, by
remark~\ref{rem-change-a-without-loss}.

By definition (section~\ref{subsec-orbifold-fundamental-group}),
the left side of \eqref{eq-homotopy-into-horoball} is equal to
$$
\bigl(\mu_{a,A,H'},R'\bigr)
\cdot
\bigl(\mu_{a,A,H},R\bigr)
=
\Bigl(\hbox{$\mu_{a,A,H'}$ followed by $R'\circ\mu_{a,A,H}$},\ R'
R\Bigr)
$$ The path part of the right side appears in
figure~\ref{fig-homotopy-into-horoball-punchline}. This figure is the
same as figure~\ref{fig-homotopy-into-horoball-preparation} except as
follows.  First, we have drawn $R'\circ\mu_{a,A,H}$ instead of
$\mu_{a,A,H}$.  Second, we have drawn the height~$1$ open horoballs
around $\rho=\infty$ and $R'(\rho)=-1$, instead of $A$ and its
translates.  It is easy to see that $\mu_{a,A,H'}$ followed by
$R'\circ\mu_{a,A,H}$ lies in the union of these two horoballs.  (One
must check that $s'$ and $R'(s)$ are first-shell roots with respect to
$R'(\rho)$, which is almost obvious.)  The union of these horoballs is
contractible and misses~$\H$.  So our path from $a$ to $R'R(a)$ may be
homotoped (rel endpoints and missing~$\H$) to a path in $A$.  We have
proven that the left side of \eqref{eq-homotopy-into-horoball} is
equal to
$$
\bigl(\hbox{some path in $A$ from $a$ to $R' R(a)$}, R' R\bigr)
$$
in the
orbifold fundamental group $G_a$.  This is exactly what the equality
\eqref{eq-homotopy-into-horoball}  asserts, so the proof is
complete.
\end{proof}

\begin{remark}
  Figures \ref{fig-homotopy-into-horoball-preparation}
  and~\ref{fig-homotopy-into-horoball-punchline} bring to mind the classical action of $\SL_2\Z$
  on the upper half plane.  This is because the projective isometry
  group of $\smallmat{0}{\thetabar}{\theta}{0}$ happens to be
  $\PSL_2\Z$.  The dashed lines indicate some fundamental domains for
  the index~$2$ subgroup generated by triflections.
\end{remark}

\begin{lemma}
  \label{lem-Heisenberg-trick-foundation}
  Suppose $\rho$, $A$ and $a$ are as in
  lemma~\ref{lem-homotopy-into-horoball}, and $S_0$ is a set of Leech
  roots whose differences span $\rho^\perp/\spanof{\rho}$.  Then the
  subgroup of $P\Gamma_{\!A}$ generated by the products $R' R$ of
  lemma~\ref{lem-homotopy-into-horoball}, with $s$ varying over $S_0$,
  contains all the translations \eqref{eq-definition-of-translation}.  In particular, it
  acts transitively on the Leech roots.
\end{lemma}

\begin{proof}
  At heart this is lemma~8 of \cite{Allcock-Y555}.  But the
  statement there, and conventions and notation, are different enough
  that we indicate the proof.  Write $G$ for the specified subgroup of
  $P\Gamma_{\!A}$.
  
  We begin by taking $r$ to be the Leech root $(0;1,-\w)$ and
  setting $r'=r-\rho=(0;1,\wbar)$.
  One can check that on the span of these vectors,
  $R_r' R_r$ acts by $-\w T_{0,-2\theta}$ (Here, as in lemma
  \ref{lem-homotopy-into-horoball}, using a  
  prime means replacing a Leech root by a translate of
  itself and replacing the complex reflection in the corresponding way). 
  Obviously $R_r' R_r$ acts
  trivially on $\Lambda$.  We re-express this by saying that $R_r'R_r$
  acts on $L$ by $-\w Q T_{0,-2\theta}$, where $Q$ scales
  $\Lambda\sset L$ by $-\wbar$ and fixes $r$ and $r'$.  One can check
  that $Q T_{\lambda,z} Q^{-1}=T_{-\wbar\lambda,z}$ for every
  translation $T_{\lambda,z}$.

  Now suppose $s\in S_0$.  As a Leech root, we have $s=T_{\sigma,y}(r)$ for
  some translation $T_{\sigma,y}$, and  we also
  have $s'=T_{\sigma,y}(r')$ because $s'=s-\rho$ and $r'=r-\rho$.  Therefore
  $$ R_s' R_s = T_{\sigma,y} \bigl(R_r' R_r\bigr) T_{\sigma,y}^{-1} = T_{\sigma,y}
  \cdot-\wbar Q T_{0,-2\theta}\cdot T_{\sigma,y}^{-1}
  $$
  Choosing another element $t$ of $S_0$ and writing it as
  $T_{\tau,z}(r)$, we get a similar description of $R_t' R_t$.  So 
  \begin{align*}
    \bigl(R_s' R_s\bigr)\bigl(R_t' R_t\bigr)^{-1}
    &{}=
    T_{\sigma,y}\cdot\bigl(-\w Q T_{0,-2\theta}\bigr)\cdot
    T_{\sigma,y}^{-1}
    \cdot
    T_{\tau,z}\cdot\bigl(-\wbar T_{0,2\theta} Q^{-1} \bigr)\cdot
    T_{\tau,z}^{-1}
    \\
    &{}=
    T_{\sigma,y} Q \bigl(T_{-\sigma,-y} T_{\tau,z}\bigr)
    Q^{-1}T_{\tau,z}^{-1}
    \\
    &{}=
    T_{\sigma,y}T_{\wbar\sigma,-y} T_{-\wbar\tau,z}T_{\tau,z}^{-1}
    \\
    &{}=
    T_{\sigma+\wbar\sigma-\wbar\tau-\tau,?}
    \\
    &{}=
    T_{-\w(\sigma-\tau),?}
  \end{align*}
  where the value of ``?'' could be worked out but is unimportant for us.

  We have shown that for each pair $s,t\in S_0$, $G$ contains a
  translation $T_{-\w(\sigma-\tau),?}$.  So it also contains a
  translation $ (R_s' R_s)T_{-\w(\sigma-\tau),?}(R_s'R_s)^{-1} =
  T_{\sigma-\tau,?}.  $ We have shown that for all $s,t\in S_0$, $G$
  contains a translation $T_{\sigma-\tau,?}$ and a translation
  $T_{-\w(\sigma-\tau),?}$.  It follows from \eqref{eq-product-of-translations} that for every
  $\lambda$ in the $\E$-span of the differences $\sigma-\tau$, $G$
  contains a translation $T_{\lambda,?}$.

  By hypothesis, the differences $s-t$, with $s,t\in S_0$, projected
  into $\rho^\perp/\spanof{\rho}$, span $\rho^\perp/\spanof{\rho}$
  over~$\E$.
  Another way to say this is the $\E$-span of the corresponding
  differences $\sigma-\tau$ is all of $\Leech$.  So we have shown that
  for every $\lambda\in\Leech$, $G$ contains a translation
  $T_{\lambda,?}$.  Using these in the commutator formula
  \eqref{eq-commutator-of-translations}, and using
  $\Lambda=\theta\Lambda^*$, shows that $G$ also contains $T_{0,z}$ for
  every $z\in\Im\E$.  It follows that $G$ contains all translations.
\end{proof}

\begin{lemma}
  \label{lem-Heisenberg-trick-application}
  Suppose $\rho$, $A$, $a$ and $S_0$ are as in lemma~\ref{lem-Heisenberg-trick-foundation}, and
  define
  $$S=\set{s,s-\rho}{s\in S_0}.$$ Then $G_a$ is generated by the
  Leech meridians $M_{a,A,s}$, where $s$ varies over $S$.
\end{lemma}

\begin{proof}
  Write $G$ for the subgroup of $G_a$ generated by these Leech meridians.
  Also, fix some $s\in S$.
  For any other Leech root $t$, there is a translation $g\in
  P\Gamma_{\!A}$ that sends $s$ to $t$.  We choose a path $\gamma$
  in $A$ from $a$ to $g(a)$, and identify $g$ with the element
  $(\gamma,g)$ of the orbifold fundamental group $G_a$, via \eqref{eq-P-Gamma-A-as-a-subgroup-of-G-a}.
  By the previous lemma we know that $(\gamma,g)$ lies in $G$.
  By drawing a picture, it is easy to see that
  $$
  (\gamma,g)\cdot(\mu_{a,A,s},R_s)\cdot(\gamma,g)^{-1}
  =
  (\mu_{a,A,t},R_t)
  $$
  in $G_a$, so $G$ also contains the Leech
  meridian associated to $t$.
  Since $t$ was an arbitrary Leech root,
  $G$ contains all the Leech meridians.  These generate $G_a$ by
  theorem~1.4 of \cite{AB-braidlike}, completing the proof that $G=G_a$.
\end{proof}

\begin{proof}[Proof of theorem~\ref{thm-130-meridians-generate}]
This follows from the previous lemma with $S_0$ equal to the set of
Leech roots $p_i$ and $\wbar p_i-l_j$, where $i,j=1,\dots,13$ and $p_i$
and $l_j$ are incident.  
To apply the lemma, we must show that the image of $K$ in
$\rho^\perp/\spanof{\rho}$ is equal to $\rho^\perp/\spanof{\rho}$, 
where $K$ is the $\E$-span of $\lbrace s - t \colon s, t \in S_0 \rbrace$.
We show that in fact $K = \rho^{\perp}$.

Note that $\rho$ has the same inner product with all elements of $S_0$.
So $K \subseteq \rho^{\perp}$. For the converse, 
let $K_1$ be the $\E$-span of $\lbrace (p_i - p_1), (l_i - l_1) \colon i = 1, \dotsb , 13 \rbrace$.
Observe that $K$ contains
$(p_1 - p_i)$.
Given a line root $l_i$, let $p_j$ be the point root
incident on $l_i$ and $l_1$. Writing
$(l_i - l_1) = (\bar{\omega} p_j - l_1) - (\bar{\omega} p_j - l_i) $ we find that
$K$ contains $(l_i - l_1)$ too.
We have now shown that $K_1 \subseteq K$.

Now suppose $v \in \rho^{\perp}$; we must show that it lies in $K$.
By adding a multiple of $(\bar{\omega} p_1 - l_1) - p_1 \in K$, which has inner product 
$\theta$ with $p_{\infty}$, we may suppose without loss of generality that $v \perp p_{\infty}$.
Being orthogonal to $\rho$ and $p_{\infty}$, the vector $v$ is also orthogonal to $l_{\infty}$.
Since $L$ is spanned by the point and line roots, we may write
$v = y + a p_1 + b l_1$ for some $y \in K_1$ and $a, b \in \E$.
Taking inner product with $p_{\infty}$ and $l_{\infty}$ and using the fact that $p_{\infty}, l_{\infty} \in K_1^{\bot}$
shows 
$a = b = 0$. So $v = y \in K_1 \subseteq K$.
\end{proof}

\section{Twenty-six generators based ``at'' a cusp}
\label{sec-26-generators-based-at-a-cusp}

\noindent
The goal of this section is to prove the theorem
\ref{thm-26-meridians-at-cusp-generate} below.
We recall
that in the $P^2\F_3$ model of $L$ from section~\ref{t-P2-F3-model-of-L}, the point-roots
are the $13$ vectors of the form $(0;\theta,0^{12})$ and the line-roots
are the $13$ vectors of the form $(1;1^4,0^9)$.  We also recall
from section~\ref{subsec-Leech-cusps-and-Leech-roots} that the null vector $\rho=(3\w-1;-1,\dots,-1)$
represents a Leech cusp, that $A$ indicates a closed horoball centered
at $\rho$, small enough to miss $\H$, and that $a\in A$ indicates the
basepoint.  Finally, recall that if $s$ is a point- or line-root, then we will call the
meridian $M_{a,A,s}\in G_a$ the corresponding point- or
line-meridian (based at~$a$).

\begin{theorem}[26 generators based ``at'' a cusp]
  \label{thm-26-meridians-at-cusp-generate}
  The $26$ point- and line-meridians $M_{a,A,s}$, based
  at $a$, generate the orbifold fundamental group~$G_a$.
\end{theorem}

To prove it we will show that the subgroup $G$ of $G_a$ generated by
the point- and line-meridians contains the~$130$ Leech meridians from
theorem~\ref{thm-130-meridians-generate}.  That theorem says that the $130$ meridians generate
$G_a$, so theorem~\ref{thm-26-meridians-at-cusp-generate} follows.  We
remark that the point-roots are Leech roots with respect to $\rho$, so
their corresponding meridians are Leech meridians.  The line-roots are
2nd-shell roots rather than Leech roots.

Our recipe for getting additional Leech meridians from the point- and
line-meridians is the following.  Suppose $x$ and $y$ are Leech roots
with inner product $\pm\theta$, or a point-root and an incident
line-root, in which case they also have inner product $\pm\theta$.
The span of two roots with inner product $\pm\theta$ is a well-known
Eisenstein lattice, called $D_4^\E$ because its underlying
$\Z$-lattice is a copy of the usual $D_4$ root lattice, scaled to have
minimal norm~$3$.  Among the $24$ roots in $\gend{x,y}$, it turns out
that there is exactly one Leech root that is not a scalar multiple of
$x$ or $y$.  Call it $z$.  We will show that the subgroup of $G_a$
generated by $M_{a,A,x}$ and $M_{a,A,y}$ also contains $M_{a,A,z}$.
``New'' Leech meridians obtained this way can be combined with old
ones using the same method, to obtain even more Leech meridians.
Repeating the process many times enables us to prove that $G$ contains
the $130$ Leech meridians of theorem~\ref{thm-130-meridians-generate}.

This idea suggested itself because the braid group associated to the
complex reflection group $\Aut D_4^\E$ is known to be isomorphic to
the usual $3$-strand braid group $B_3$, with the conjugacy class of
our meridians corresponding to the conjugacy class of the standard
generators of~$B_3$.  It was natural to hope that the meridians
$M_{a,A,x}$, $M_{a,A,y}$ and $M_{a,A,z}$ all lie in a single copy of
$B_3$, with the first two of them serving as generators.  So our first
result is that the braid group of $D_4^\E$ is generated by two
specific meridians, for either of two specific basepoints. It seems
silly to have to prove this, since the braid groups of finite complex
reflections groups are well understood
\cite{Bannai}\cite{Bessis}\cite{Michel}.  But we do not know of any
quotable result stating generation by particular loops with
particular basepoints.

\begin{lemma}[Generators for the braid group of $D_4^\E$]
  \label{lem-generators-for-braid-group-of-D-4}
  Consider the lattice $D_4^\E$ in $\C^2$, namely the $\E$-span of
  $\alpha=(\theta,0)$ and $\beta=(1,\sqrt2)$ under the
  usual inner product.
  Let $\H$ be the
  union of the mirrors orthogonal to the $24$ roots of $D_4^\E$, and
  let $c$ be any scalar multiple of
  $(-1,-\wbar\sqrt2)$
  or 
  $\frac{\w}{2}\bigl(\theta+1,(\theta-2)\sqrt2\bigr)$.
  Then
  $J_c=\piorb\bigl((\C^2-\H)/\gend{R_\alpha,R_\beta},c\bigr)$
  is generated by $M_{c,\alpha}$ and
  $M_{c,\beta}$.
\end{lemma}

\begin{remark}
The $\sqrt2$'s can be avoided by using three coordinates with the last
two being equal: $\alpha=(\theta,0,0)$ and $\beta=(1,1,1)$.
\end{remark}

\begin{proof}
  We begin by remarking that $\alpha^2=\beta^2=3$ and
  $\ip{\alpha}{\beta}=\theta$.  We write $\gamma=(1,\wbar\sqrt2)$ and
  $\delta=(1,\w\sqrt2)$ for the remaining roots of $D_4^\E$, up to
  scale.  We will begin the proof by establishing the lemma for a
  different basepoint $c_0=(-2,\sqrt2+\sqrt6)$.  This basepoint is
  probably the most natural: when one plots the situation in $\C P^1$, the four
  mirrors form the vertices of a regular tetrahedron and $c_0$
  corresponds to the midpoint of the edge between $\alpha^\perp$ and
  $\beta^\perp$.  We will suppress the subscript $0$ until we need to
  refer to the other $c$'s.

  One can check that $\ip{c}{\alpha}$ and $\ip{c}{\beta}$ are equal in
  absolute value, and smaller in absolute value than $\ip{c}{\gamma}$
  and $\ip{c}{\delta}$.  Therefore $\alpha^\perp$ and $\beta^\perp$
  are the mirrors nearest~$c$.  Write $J$ for the subgroup of $J_c$
  generated by $M_{c,\alpha}=(\mu_{c,\alpha},R_\alpha)$ and
  $M_{c,\beta}=(\mu_{c,\beta},R_\beta)$.  We are claiming that $J$ is all of
  $J_c$.  Theorem~1.2 of \cite{AB-braidlike} gives sufficient conditions for
  this.  First,  $J$ should surject to $\gend{R_\alpha,R_\beta}$,
  which is obvious.  
  Second, some
  triflection in $\alpha$ or $\beta$ should move $c$ closer to
  $c$'s projection to $\gamma^\perp$, and some triflection
  should do the same with $\delta$ in place of $\gamma$.  
  Direct calculation reveals
  that $R_\alpha$ does this for $\gamma^\perp$ and
  $R_\alpha^{-1}$ does it for $\delta^\perp$.
  This finishes the proof of the lemma for
  the basepoint $c_0$.

  A different method is required for the other basepoints.  First we
  suppose the basepoint $c$ is a scalar multiple of
  $(-1,-\wbar\sqrt2)$. It does not matter which multiple, because
  scaling permutes the possibilities (and
  transforms their meridians accordingly).
  We choose $c$ to be
  the particular multiple $c_1=(\theta\wbar-\w\sqrt3)\bigl(-1,-\wbar\sqrt2\bigr)$.
  The advantage of this scaling is that $\ip{c_0}{c_1}$ is
  positive, so the angle between $c_0$ and $c_1$ is as small as
  possible.  We regard this as saying that $c_1$ is ``as close as
  possible'' to $c_0$, which is desirable because we will  study the identification of $J_{c_0}$ with $J_{c_1}$ got by
  moving $c_0$ to $c_1$ along $\geodesic{c_0 c_1}$.  Namely, will show
  that $M_{c_0,\alpha}\in J_{c_0}$ corresponds to
  $M_{c_1,\alpha}\in J_{c_1}$, and similarly with
  $\beta$ in place of $\alpha$.  From this correspondence and the
  previous paragraph it follows
  immediately that $M_{c_1,\alpha}$ and
  $M_{c_1,\beta}$ generate $J_{c_1}$.

  We will write $p_i$ for the projection of $c_i$ to $\alpha^\perp$.
  One can check that the convex hull $T$ of $c_0$, $p_0$ and $c_1$
  meets $\H$ only at $p_0$.  
  (Let $\ip{T}{\alpha}$ be the triangle in $\C$ with vertices 
  $\ip{c_0}{\alpha}$, $\ip{p_0}{\alpha}$, $\ip{p_1}{\alpha}$.
  One checks that the triangle $\ip{T}{\alpha}$ touches the origin 
  only at its vertex $\ip{p_0}{\alpha}$. Similarly one checks
  that $\beta^{\perp}$, $\gamma^{\perp}$, $\delta^{\perp}$
  do not intersect $T$ by checking that the triangles
  $\ip{T}{\beta}$, $\ip{T}{\gamma}$, $\ip{T}{\delta}$
  do not contain the origin.)
  The obvious variant of
  Lemma~\ref{lem-homotopy-for-moving-the-basepoint} 
  for Euclidean space implies that
  moving the basepoint from $c_0$ to $c_1$ along $\geodesic{c_0c_1}$
  identifies $M_{c_0,\alpha}\in J_{c_0}$ with $M_{c_1,p_0,\alpha}\in
  J_{c_1}$.  Similarly, one checks that the convex hull of $p_0$,
  $c_1$ and $p_1$ misses $\H$ except that
  $\geodesic{p_0p_1}\sset\alpha^\perp$.  From this,
  lemma~\ref{lem-homotopy-for-moving-the-point-on-the-hyperplane-ordinary-basepoint-version}
  concludes that $M_{c_1,p_0,\alpha}$ is the same element of $J_{c_1}$
  as $M_{c_1,\alpha}$.   We have proven the correspondence between
  $M_{c_0,\alpha}$ and $M_{c_1,\alpha}$.  The same method applies with
  $\beta$ in place of $\alpha$.  This completes the proof for the
  basepoint~$c_1$.

  Exactly  the same analysis applies to the second possibility for
  $c$.  The scalar multiple of
  $\frac{\w}{2}\bigl(\theta+1,(\theta-2)\sqrt2\bigr)$ that we used was
  $$c_2=-(3+2\sqrt3+\theta\sqrt3)\Bigl(\theta+1,(\theta-2)\sqrt2\Bigr),$$
  for the same reason as before.
\end{proof}

As explained above, our method starts with two Leech meridians,
corresponding to Leech roots with inner product $\pm\theta$, and shows
that the group they generate contains a third Leech meridian.  So
naturally we will need to understand such pairs of Leech roots.  At
the same time we will classify the pairs of Leech roots with inner
product $-\frac32\pm\frac\theta2$.  We do not care about these pairs
themselves.  But they turn out to be key for the other half of our
method: using a point-meridian and a line-meridian to generate another
Leech meridian.  The precise statements of our constructions of
``new'' Leech meridians are lemmas
\ref{lem-configuration-3-Leech-roots-plus-1-second-shell}\eqref{item-two-Leech-meridians-give-a-third}
and~\ref{lem-point-and-line-meridians-give-another-Leech-meridian}.

\begin{lemma}
  \label{lem-ordered-pairs-of-Leech-roots-with-absolute-inner-product-sqrt3}
  Under the $\Gamma$-stabilizer $\Gamma_{\!\rho}$ of $\rho$, there are
  four orbits of ordered pairs of Leech roots with inner product of
  absolute value $\sqrt3$.  In the Leech model of $L$, orbit
  representatives are the pairs $(s,s')$ where
  \begin{center}
    \begin{tabular}{lll}
      $s=(0;1,-\w)$ and $s'=(\lambda_6;1,\w),$&with $\ip{s}{s'}=\theta$
      \\
      $s=(0;1,-\w)$ and $s'=(\lambda_9;1,\theta),$&with $\ip{s}{s'}=-\frac{3}{2}+\frac{\theta}{2}=\wbar\thetabar$
    \end{tabular}
  \end{center}
  and the pairs got from these by exchanging $s,s'$.  Here $\lambda_6$
  and $\lambda_9$ are any fixed vectors of norms $6$ and~$9$ in
  $\Leech$. 
\end{lemma}

\begin{proof}
  Suppose $(s,s')$ is such an ordered pair.  In the Leech model, the
  Leech roots are parameterized by \eqref{eq-Leech-roots-in-Leech-model}.  So $s$ has the form
  $\bigl(\sigma;1,\theta(\frac{\sigma^2-3}{6}+\nu)\bigr)$ where
  $\sigma$ lies in the Leech lattice and $\nu$ is purely imaginary and
  chosen so that the last coordinate lies in $\E$.  And similarly
  for~$s'$.  By the transitivity of $\Gamma_{\!\rho}$ on Leech roots,
  we may take $\sigma=0$ and $\nu=\frac{1}{2\theta}$, yielding
  $s=\bigl(0;1,\theta(-\frac12+\frac1{2\theta})\bigr)=(0;1,-\w)$.  The
  inner product $\ip{s}{s'}$ is best understood using \eqref{eq-general-inner-product-formula}, which
  in this case reduces to
  $$
  \ip{s}{s'}=\frac{6-\sigma'^2}{2}+3\bigl(\nu'-1/2\theta\bigr).
  $$
  The first term is real and the second is imaginary.
  Now we consider the elements of $\E$ of absolute value~$\sqrt3$.
  First, $\frac32\pm\frac\theta2$ cannot occur, because $\sigma'$
  would have to have norm~$3$ in order for $\ip{s}{s'}$ to have real
  part~$\frac32$.  This is impossible since the Leech lattice has
  minimal norm~$6$.   By the same reasoning, $\pm\theta$ can occur only
  when $\sigma'$ has norm~$6$, and $-\frac32\pm\frac\theta2$ can occur
  only when $\sigma'$ has norm~$9$.  Since $\Aut\Leech=6\cdot\Suz$ acts
  transitively on the norm~$6$ and norm~$9$ elements of $\Leech$, we
  may suppose without loss of generality that $\sigma'=\lambda_6$ or
  $\lambda_9$.  Then $\nu'$ is determined by the imaginary
  part of $\ip{s}{s'}$.  When $\ip{s}{s'}=\theta$ we get
  $3(\nu'-\frac{1}{2\theta})=\theta$, so $\nu'=-\frac{1}{2\theta}$ and $s'$ is
  as displayed.   Similarly, when
  $\ip{s}{s'}=-\frac{3}{2}+\frac{\theta}{2}$ one gets $\nu'=0$.  
\end{proof}

The proofs of the next two lemmas depend on how certain totally real
triangles in $\BB^{13}$ meet $\H$.  Because the calculations are long,
and similar computations will be needed later, we moved the verifications
to appendix~\ref{app-how-real-triangles-meet-mirrors}.
Lemma~\ref{lem-configuration-3-Leech-roots-plus-1-second-shell} relies
on
lemma~\ref{lem-triangle-analysis-for-2-Leech-meridians-give-a-third}
and
lemma~\ref{lem-configuration-2-Leech-roots-plus-1-second-shell-plus-1-third-shell}
relies on
lemma~\ref{lem-triangle-analysis-for-point-and-line-meridians-give-Leech-meridian}.

\begin{lemma}
  \label{lem-configuration-3-Leech-roots-plus-1-second-shell}
  Suppose $x,y$ are Leech roots with $\ip{x}{y}=\theta$, $p$ is the
  projection of $\rho$ to $x^\perp$, and $q$ is the projection of
  $\rho$ to the $x^\perp\cap y^\perp$.  Then
  \begin{enumerate}
  \item
    \label{item-three-mirrors-are-Leech-mirrors}
    Of the four mirrors of the reflection group $\gend{R_x,R_y}$,
    three are Leech mirrors (corresponding to the
    Leech roots $x$, $y$ and $z=-\wbar x-\w y$) and one is a second-shell
    mirror.
  \item
    \label{item-mirrors-through-q-are-just-3-Leech-and-1-second-shell}
    These four mirrors are the only ones containing $q$.
  \item
    \label{item-two-Leech-meridians-give-a-third}
    The subgroup of $G_a$ generated by the Leech
    meridians $M_{a,A,x}$ and $M_{a,A,y}$
        also contains the Leech meridian $M_{a,A,z}$.
  \end{enumerate}
\end{lemma}

\begin{remark}
An example of the situation described
in this lemma is given by
$x = -R_{p_2}(l)$ and $y = p$, where $l$ is a line-root
and $p, p_2$ are the two point-roots incident on $l$.
  We formulated our notion of ``meridian'' to include cases
  where the basepoint is not in general position.  The situation in
  this lemma is such a case, because $q$ is the projection of $\rho$
  to the second-shell mirror in \eqref{item-three-mirrors-are-Leech-mirrors}, and is only one of $4$
  mirrors passing through~$q$.
  In fact, one can prove that this is the only sort of non-genericity
  that occurs for our lattice $L$.
\end{remark}

\begin{proof}
  By
  lemma~\ref{lem-ordered-pairs-of-Leech-roots-with-absolute-inner-product-sqrt3}
  we may suppose without loss of generality that $x=(0;1,-\w)$ and
  $y=(\lambda_6;1,\w)$.
  And we suppose without
  loss of generality 
  that our basepoint $a$ lies in
  $\geodesic{\rho q}\cap A$ (see remark~\ref{rem-change-a-without-loss}).
  
  \eqref{item-three-mirrors-are-Leech-mirrors} Using $x^2=y^2=3$
  and $\ip{x}{y}=\theta$, one checks that $z$ and $\w x-\wbar y$
  are roots.  Then, using $\ip{\rho}{x}=\ip{\rho}{y}=\theta$, one
  checks that $\ip{\rho}{z}=\theta$, so that $z$ is also a Leech
  root, and that $\ip{\rho}{\w x-\wbar y}=3$, so that $(\w
  x-\wbar y)^\perp$ is a second-shell mirror.

  \eqref{item-mirrors-through-q-are-just-3-Leech-and-1-second-shell}
    Lemma~\ref{lem-triangle-analysis-for-2-Leech-meridians-give-a-third}
  identifies all the mirrors that meet the totally real triangle
  $\triangle\rho p q$.  The only ones that do are the
  four from \eqref{item-three-mirrors-are-Leech-mirrors}.  So they are the only ones that can contain~$q$.
    
  Before proving \eqref{item-two-Leech-meridians-give-a-third} we
  establish two preparatory results.  First, there is an isometry of
  $L$ that fixes~$\rho$ and permutes $x$, $y$ and~$z$ cyclically.
  To see this, one checks that $\ip{y}{z}=\theta$ and applies the
  transitivity in
  lemma~\ref{lem-ordered-pairs-of-Leech-roots-with-absolute-inner-product-sqrt3}
  to conclude that some $g\in\Gamma_\rho$ sends $(x,y)$ to
  $(y,z)$.  Now, $g$ must permute the three Leech roots in
  $\spanof{x,y}=\spanof{y,z}$.  Since $g$ sends $x$ to $y$ and
  $y$ to $z$, it must also send $z$ to $x$.
    
  Second, we claim that $\mu_{a,A,x}$ is homotopic to
  $\mu_{a,A,q,x}$ in $\BB^{13}-\H$, rel endpoints, and similarly
  with $y$ or $z$ in place of~$x$.  By the symmetry just established,
  the $y$ and $z$ cases follow from the $x$ case.
  Lemma~\ref{lem-triangle-analysis-for-2-Leech-meridians-give-a-third}
  describes how the mirrors of $\H$ meet  $\triangle\rho q p$.  Namely,
  $x^\perp$ contains $\geodesic{pq}$, the other three mirrors from
  \eqref{item-three-mirrors-are-Leech-mirrors} meet the triangle at
  $q$ only, and all other mirrors miss it completely.  This verifies
  the hypothesis of 
  lemma~\ref{lem-homotopy-for-moving-the-point-on-the-hyperplane-cusp-version}, whose conclusion is that $\mu_{a,A,x}\homotopic\mu_{a,A,q,x}$, as claimed.
    
  \eqref{item-two-Leech-meridians-give-a-third} Let $U$ be an open
  ball around $q$, small enough that the only mirrors meeting it are
  the ones through $q$.   We also take it
  small enough to miss~$A$.  Let $c$ be any point of $\geodesic{\rho
    q}\cap U$ other than $q$.  Examining the definition of the
  meridian shows that we may suppose without loss of generality that
  $c$ is the ``turning point'' of $\mu_{a,A,q,x}$.  Formally:
  $\mu_{a,A,q,x}$ equals $\geodesic{a c}$ followed by $\mu_{c,x}$
  followed by $R_x(\geodesic{c a})$.  (Note that $\geodesic{a
    c}\sset\geodesic{\rho q}$ by our choice of $a$ at the beginning of the
  proof.  And we can use $\geodesic{ac}$ in place of $\dodge{ac}$
  because it misses $\H$ by lemma~\ref{lem-triangle-analysis-for-2-Leech-meridians-give-a-third}.)  Another way to state the  homotopy
  $\mu_{a,A,x}\homotopic\mu_{a,A,q,x}$  of the previous paragraph is that
  $M_{a,A,x}=(\mu_{a,A,x},R_x)\in G_a$ corresponds to
  $M_{c,x}=(\mu_{c,x},R_x)\in G_c$  under the
  identification of $G_a$ with $G_c$ induced by $\geodesic{a c}$.
  And similarly with $y$ or $z$
  in place of~$x$.  Therefore proving \eqref{item-two-Leech-meridians-give-a-third} is equivalent to
  proving that $M_{c,z}$ lies in the subgroup of
  $G_c$ generated by $M_{c,x}$ and
    $M_{c,y}$.
  
  We focus attention on $U$ by defining a ``local'' analogue of $G_c$,
  namely $J_c=\piorb\bigl((U-\H)/\gend{R_x,R_y},c\bigr)$.  By
  considering the natural map $J_c\to G_c$, it suffices to prove that
  $M_{c,z}$ lies in the subgroup of $J_c$ generated by $M_{c,x}$ and
  $M_{c,y}$.  We convert this into a problem in the Euclidean space
  $T_q\BB^{13}$ as follows.  We write $g_0$ for the Riemannian metric
  on $\BB^{13}$ got by identifying $\BB^{13}$ with $T_q\BB^{13}$ under
  the exponential map.  This metric is Euclidean, so we write
  $\mu_{c,x}^\Euc$ for the path defined like $\mu_{c,x}$, but using
  this Euclidean metric in place of the hyperbolic metric.  We claim:
  $\mu_{c,x}^\Euc$ and $\mu_{c,x}$ are homotopic rel endpoints in
  $U-\H$.  In particular, $M_{c,x}=(\mu_{c,x},R_x)$ and
  $M_{c,x}^\Euc=(\mu_{c,x}^\Euc,R_x)$ are the same element of $J_c$,
  and similarly with $y$ or $z$ in place of $x$.  After proving this
  claim, it will suffice to prove that $M_{c,x}^\Euc$ and
  $M_{c,y}^\Euc$ generate~$J_c$.
  
  To prove the claim, we define for each $t\in(0,1]$ a diffeomorphism
    $S_t$ of $\BB^{13}$, namely the one corresponding to the
    scaling-by-$t$ map on $T_q\BB^{13}$.  Let $g_t$ be the Riemannian
    metric got by pulling back the hyperbolic metric under $S_t$ and
    multiplying it by $1/t$.  It is easy to see that the $g_t$
    converge to $g_0$ as $t\to0$.  For any $t\in[0,1]$ we
    define $\mu_{c,x}^t$ just as we did for $\mu_{c,x}$,
    except that we use $g_t$ in place of the hyperbolic metric.  All
    the ingredients in this definition (nearest points, geodesics,
    etc.)  vary continuously with $t$, so the $\mu_{c,x}^t$ provide
    a homotopy between $\mu_{c,x}^1=\mu_{c,x}$ and
    $\lim_{t\to0}\mu_{c,x}^t=\mu_{c,x}^\Euc$, as desired.  (Note that
    $\BB^{13}$ is nonpositively curved for any $t$; for $t>0$ this is
    because $(\BB^{13},g_t)$ is isometric to $\BB^{13}$ with its
    metric scaled, and scaling doesn't affect the sign of sectional
    curvature.  And $g_0$ is Euclidean by
    definition.  $U$ is convex under any $g_t$, because it
    is always an open ball centered at~$q$.)

    It remains to show that $M_{c,z}^\Euc$ lies in the group generated
    by $M_{c,x}^\Euc$ and
    $M_{c,y}^\Euc$.
    Obviously we may restrict attention
    to the $\BB^2$ orthogonal to $\spanof{x,y}^\perp$ at~$q$.
    We will identify $T_q\BB^2$ with the $\C^2$ from
    lemma~\ref{lem-generators-for-braid-group-of-D-4}, in a way which lets us quote that result.  
    We first note that $\BB^2=\BB(V)$ where $V$
    is the complex span of $\rho$, $x$ and~$y$.  We write $W$ for the
    subspace of $V$ spanned by $x$ and $y$, and $Q$ for the linear projection
    of $\rho$ to $W^\perp$.  This is a vector
    representing~$q\in\BB^2$.  We identify $\BB^2$ with a neighborhood
    of $0$ in $W$ by
    $$
    (w\in W)\leftrightarrow\bigl(\hbox{the image in $P V$ of $Q+w$}\bigr).
    $$
    Up to a scaling factor, this correspondence is essentially the exponential
    map of $\BB^{13}$ at $q$.
    This yields an identification of $T_q\BB^2$ with $T_0W=W$,
    in which the mirrors through $q$ correspond to the mirrors of
    the $D_4^\E$ spanned by $x$ and~$y$, and this identification
    is equivariant under the action of the finite group generated
    by $R_x$ and $R_y$.  (We saw in
    \eqref{item-mirrors-through-q-are-just-3-Leech-and-1-second-shell} that  no more mirrors pass through~$q$.)
    Furthermore, $c$ lies in $P(\C Q+\C\rho)$, whose corresponding
    complex line in $W$ is spanned by the projection $\pi_W(\rho)$ of
    $\rho$ to~$W$.  This projection can be worked out from
    $\ip{\rho}{x}=\ip{\rho}{y}=\theta$, namely $\pi_W(\rho)=\w x-\wbar
    y$.

    Now we identify $x$ (resp.\ $y$) with $\alpha$ (resp.\ $\beta$) in
    lemma~\ref{lem-generators-for-braid-group-of-D-4}.  Under this identification, $T_q\BB^2-T_q\H$ is
    identified with the mirror complement of $D_4^\E$, with $c$
    corresponding to some scalar multiple of
    $\w\alpha-\wbar\beta=(-1,-\wbar\sqrt2)$.  That lemma tells us
    that $M_{c,\alpha}$ and
    $M_{c,\beta}$ generate
    $\piorb\bigl((\C^2-\H)/\gend{R_\alpha,R_\beta},c\bigr)$.
    Transferring this back to $U-\H$ finishes the proof.
\end{proof}

\begin{lemma}
  \label{lem-configuration-2-Leech-roots-plus-1-second-shell-plus-1-third-shell}
  Suppose $x,z$ are Leech roots with
  $\ip{x}{z}=-\frac32+\frac{\theta}{2}$, and let $q$ be the
  projection of $\rho$ to the intersection of their mirrors.  Then
  \begin{enumerate}
  \item
    \label{item-2-mirrors-are-Leech-one-is-2nd-shell-and-one-is-3rd}
    Of the four mirrors of the reflection group $\gend{R_x,R_z}$,
    two are Leech mirrors ($x^\perp$ and ${z}^\perp$), one is a
    second-shell mirror ($y^\perp$ for $y=\wbar x-z$), and one is a
    third-shell mirror.
  \item
    \label{item-mirrors-through-q-are-just-2-Leech-1-second-shell-1-third-shell}
    These mirrors are the only ones containing $q$.
  \item
    \label{item-Leech-meridian-and-2nd-shell-meridian-give-Leech-meridian}
    The subgroup of $G_a$ generated by the Leech meridian
    $M_{a,A,x}$ and the second-shell meridian
    $M_{a,A,y}$ also contains the Leech meridian $M_{a,A,z}$.
  \end{enumerate}
\end{lemma}

\begin{proof}
  This is very similar to the proof of
  lemma~\ref{lem-configuration-3-Leech-roots-plus-1-second-shell}.
  By lemma~\ref{lem-ordered-pairs-of-Leech-roots-with-absolute-inner-product-sqrt3} we may suppose without loss of generality that
  $x=(0;1,-\w)$ and $z=(\lambda_9;1,\theta)$, and as before we take
  $a\in\geodesic{\rho q}$. 

  \eqref{item-2-mirrors-are-Leech-one-is-2nd-shell-and-one-is-3rd}
  Using $x^2=z^2=3$ and $\ip{x}{z}=-\frac32+\frac\theta2$, one
  checks that $y$ and $x+z$ are roots.  So the $24$ roots in
  $\spanof{x,z}$ are their unit multiples together with those of $x$
  and~$z$.  Using $\ip{\rho}{x}=\ip{\rho}{z}=\theta$, one computes
  $\ip{\rho}{y}=3\wbar$ and $\ip{\rho}{x+z}=2\theta$.  So $y$ is a
  second shell root and $x+z$ is a third shell root.

  \eqref{item-mirrors-through-q-are-just-2-Leech-1-second-shell-1-third-shell}
  This is just like the corresponding part of the previous lemma,
  except that we refer to lemma~\ref{lem-triangle-analysis-for-point-and-line-meridians-give-Leech-meridian}
  in place of lemma~\ref{lem-triangle-analysis-for-2-Leech-meridians-give-a-third}.
  
  In preparation for \eqref{item-Leech-meridian-and-2nd-shell-meridian-give-Leech-meridian}, we claim that $\mu_{a,A,x}$ is
  homotopic to $\mu_{a,A,q,x}$ in $\BB^{13}-\H$, rel endpoints,
  and similarly with $y$ or $z$ in place of $x$.  This is just like
  the corresponding part of the previous proof, except that there is
  no cyclic symmetry.  So one has to analyze three triangles rather
  than just one.  This is done in
  lemma~\ref{lem-triangle-analysis-for-point-and-line-meridians-give-Leech-meridian}.

  \eqref{item-Leech-meridian-and-2nd-shell-meridian-give-Leech-meridian}
  This is just like the corresponding part of the previous lemma.  The
  only difference is in the very last step: now $\ip{\rho}{x}=\theta$
  and $\ip{\rho}{y}=3\wbar$, so $\pi_W(\rho)=-\frac{\theta}{2}\w
  x+(\frac{\theta}{2}-1)\w y$.  We still identify $T_q\BB^2$ with
  $\C^2$ by taking $x$ and $y$ to correspond to lemma~\ref{lem-generators-for-braid-group-of-D-4}'s $\alpha$ and $\beta$.
  This makes sense since one can check $\ip{x}{y}=\theta$.  Under this
  identification, $\pi_W(\rho)$ corresponds to
  $\frac{\w}{2}\bigl(\theta+1,(\theta-2)\sqrt2\bigr)\in\C^2$, and we can
  apply lemma~\ref{lem-generators-for-braid-group-of-D-4} just as before.
\end{proof}

\begin{lemma}
  \label{lem-point-and-line-meridians-give-another-Leech-meridian}
  Suppose $x$ is a point-root and $y$ is an incident line-root.
  Then $z=\wbar x-y$ is a Leech root, and the subgroup of $G_a$
  generated by $M_{a,A,x}$ and $M_{a,A,y}$
  contains the Leech meridian $M_{a,A,z}$.
\end{lemma}

\begin{proof}
  Using $x^2=y^2=3$ and $\ip{x}{y}=\theta$, one checks $z^2=3$, so $z$
  is a root.  Using $\ip{\rho}{x}=\theta$ and $\ip{\rho}{y}=3\wbar$,
  one checks $\ip{\rho}{z}=\theta$, so $z$ is a Leech root.
  Similarly, one checks that
  $\ip{x}{z}=-\frac{3}{2}+\frac{\theta}{2}$.  So we may apply
  lemma~\ref{lem-configuration-2-Leech-roots-plus-1-second-shell-plus-1-third-shell}
  to $x$ and~$z$.
  One checks that the root called $y$ there is the same as the one we
  have called $y$.  To finish the proof we appeal to
  lemma~\ref{lem-configuration-2-Leech-roots-plus-1-second-shell-plus-1-third-shell}\eqref{item-Leech-meridian-and-2nd-shell-meridian-give-Leech-meridian}.
\end{proof}

\begin{proof}[Proof of theorem~\ref{thm-26-meridians-at-cusp-generate}]
Let $G$ denote the subgroup of $G_a$ generated by the $26$ point and
line meridians $M_{a,A,s}$ based at $a$.
We must prove $G=G_a$.
By
theorem~\ref{thm-130-meridians-generate} it suffices to show that $G$ contains the Leech
meridians associated to the Leech roots $p_i$, $\wbar p_i-l_j$,
$p_i-\rho$ and $\wbar p_i-l_j-\rho$, where $i,j=1,\dots,13$ and $p_i$
and $l_j$ are incident.  $G$ contains the Leech meridians associated
to the $p_i$ by definition, and those associated to the $\wbar p_i-l_j$ by
lemma~\ref{lem-point-and-line-meridians-give-another-Leech-meridian}.  By $L_3(3)$ symmetry, it now suffices to show that $G$
contains the meridians corresponding to the Leech roots $p_1-\rho$ and $\wbar
p_1-l_1-\rho$.

For $p_1-\rho$, consider the following sequence of Leech roots:
\begin{align*}
  s_1&{}=\wbar p_1 -l_1
  &s_7&{}=-\wbar s_1-\w p_4
  &s_{13}&{}=-\wbar s_{12}-\w s_{11} 
  \\
  s_2&{}=\wbar p_1  -l_{11}
  &s_8&{}=-\wbar s_4-\w p_3    
  &s_{14}&{}=-\wbar s_{13}-\w p_1 
  \\
  s_3&{}=\wbar p_1  -l_{13}
  &s_9&{}=-\wbar s_4-\w p_5    
  &s_{15}&{}=-\wbar s_3-\w s_{14} 
  \\
  s_4&{}=\wbar p_2  -l_2
  &s_{10}&{}=-\wbar s_5-\w s_8  
  &s_{16}&{}=-\wbar s_{15}-\w s_7  
  \\
  s_5&{}=\wbar p_5  -l_5
  &s_{11}&{}=-\wbar s_6-\w s_9  
  &s_{17}&{}=-\wbar s_{16}-\w p_{10}  
  \\
  s_6&{}=\wbar p_{11}-l_{11}
  &s_{12}&{}=-\wbar s_2-\w s_{10} 
\end{align*}
In each equation in the left column, the point and line root on the
right hand side are incident.  By
lemma~\ref{lem-point-and-line-meridians-give-another-Leech-meridian},
$G$ contains the Leech meridian associated to the Leech root defined
by that equation.  In each equation in the other two columns, the
roots appearing on the right side are Leech roots and $\ip{\hbox{the
    first one}}{\hbox{the second}}=\theta$.  By
repeated use of lemma~\ref{lem-configuration-3-Leech-roots-plus-1-second-shell}\eqref{item-two-Leech-meridians-give-a-third},
$G$ contains the  meridians associated to the Leech roots defined
by these equations.  And one checks that $s_{17}=p_1-\rho$.

For $\wbar p_1-l_1-\rho$ the argument is the same, defining
\begin{align*}
  s_1  &{}=\wbar p_3-l_2
  &s_6  &{}=-\wbar s_1-\w p_5
  &s_{11}&{}=-\wbar s_8-\w s_{10}
  \\
  s_2  &{}=\wbar p_6-l_5
  &s_7  &{}=-\wbar s_5-\w p_{12}
  &s_{12}&{}=-\wbar s_2-\w s_{11}
  \\
  s_3  &{}=\wbar p_8-l_5
  &s_8  &{}=-\wbar s_4-\w s_6
  &s_{13}&{}=-\wbar s_{12}-\w s_1
  \\
  s_4  &{}=\wbar p_2-l_6
  &s_9  &{}=-\wbar s_3-\w s_7
  &s_{14}&{}=-\wbar s_{13}-\w p_{11}
  \\
  s_5  &{}=\wbar p_2-l_{12}
  &s_{10}&{}=-\wbar s_9-\w p_{13}
\end{align*}
and checking that $s_{14}=\wbar p_1-l_1-\rho$.  We found these
sequences of Leech roots by a rather intensive computer search.  But
their validity can be verified easily.
\end{proof}

\section{Change of basepoint}
\label{sec-change-of-basepoint}

\noindent
In this section we prove the main theorem of the paper,
theorem~\ref{t-26-meridians-based-at-tau-generate}: the orbifold
fundamental group
$G_\tau=\piorb\bigl((\BB^{13}-\H)/P\Gamma,\tau\bigr)$ is generated by
the $26$ point- and line-meridians $M_{\tau,H}$.  Here
$\tau=(4+\sqrt3;1,\dots,1)$ is the $26$-point specified in
section~\ref{subsec-point-roots-etc} and $H$ varies over the point- and
line-mirrors.  The starting point of the proof is
theorem~\ref{thm-26-meridians-at-cusp-generate}: for any basepoint $a$
in the horoball $A$ centered at the Leech cusp $\rho$, the $26$ point-
and line-meridians $M_{a,A,H}$ generate $G_a$.  Here
$\rho=(3\w-1;-1,\dots,-1)$ is the Leech cusp defined in
section~\ref{subsec-Leech-cusps-and-Leech-roots} and used throughout
section~\ref{sec-26-generators-based-at-a-cusp}, and $H$ varies over
the same $26$ mirrors.  By remark~\ref{rem-change-a-without-loss} we may choose $a$ to be
a point of $\geodesic{\tau\rho}$ very close to $\rho$; exactly how
close will be specified later.  So $\geodesic{\tau a}$ is a subsegment
of $\geodesic{\tau\rho}$.  In light of theorem~\ref{thm-26-meridians-at-cusp-generate},
theorem~\ref{t-26-meridians-based-at-tau-generate} follows immediately
from the next lemma.

\begin{lemma}
  \label{lem-change-of-basepoint}
  Suppose $H$ is a point- or line-mirror.  Then
  the meridian
$M_{\tau,H}\in G_\tau$
    corresponds to the
  meridian
  $M_{a,A,H}\in G_a$
  under the isomorphism
  $G_\tau\iso G_a$ induced by the path $\geodesic{\tau a}$.  
\end{lemma}

For the lemma to make sense, one must verify that
$\geodesic{\tau a}$ misses $\H$.  Because
$\geodesic{\tau a}\sset\geodesic{\tau\rho}$, this follows from the
stronger result (lemma~\ref{lem-complex-triangles-miss-mirrors-except-as-known}) that the complex triangle
$\triangle \rho\tau l_\infty$ misses $\H$ except at $l_\infty$.

\begin{proof}
  We will give the proof when $H$ is a line-mirror, and then remark on
  the changes needed for the point-mirror case.  Recall from
  section~\ref{subsec-point-roots-etc} that $l_\infty$ is where all
  $13$ line-mirrors intersect.  We will use a $4$-step homotopy.  The
  first step corresponds to the bottom region (shaded darkly) in
  figure~\ref{fig-change-of-basepoint}, the second step to the region
  above it (lighter), the third to the next (dark again), and the
  fourth to the rightmost region (light again).  The key fact is that
  these regions miss $\H$ except at known points.  These verifications
  are lemmas
  \ref{lem-disjointness-needed-for-step-1-of-basepoint-change},
  \ref{lem-step-4-line-case} and~\ref{lem-complex-triangles-miss-mirrors-except-as-known} in the appendices.

  \begin{figure}
    \begin{center}
      \begin{tikzpicture}[x=1.45cm,y=1.45cm,every to/.style={hyperbolic plane}]
        \def\LLx{-.15}
        \def\LLy{-.2}
        \def\URx{5.3}
        \def\URy{5}
        \clip(\LLx,\LLy)rectangle(\URx,\URy);
        \def\Aheight{3}
        \def\aheight{\Aheight+1}
        \def\rhoheight{\aheight+.8}
        \def\rightmost{\pprimex+.5}
        \def\radius{2.5}
        \def\pprimex{4}
        \def\linftyangle{140}
        \coordinate (center) at (\pprimex,0);
        \coordinate (pprime) at (\pprimex,\radius);
        \coordinate (bprime) at (\pprimex,\Aheight);
        \coordinate (rhoABOVEpprime) at (\pprimex,\rhoheight);
        \path (center)++(\linftyangle:\radius) coordinate (linfty);
        \path (linfty);
        \pgfgetlastxy{\XCoord}{\YCoord};
        \coordinate (b) at (\XCoord,\Aheight);
        \coordinate (rhoABOVElinfty) at (\XCoord,\rhoheight);
        \def\smallradius{1}
        \pgfmathparse{veclen(\radius,\smallradius)}
        \let\centergap\pgfmathresult
        \def\tauangle{135}
        \coordinate(smallcenter) at(\pprimex-\centergap,0);
        \path(smallcenter)++(\tauangle:\smallradius) coordinate (tau);
        \path(tau);
        \pgfgetlastxy{\XCoord}{\YCoord};
        \coordinate (a) at (\XCoord,\Aheight+1);
        \coordinate (rhoABOVEa) at (\XCoord,\rhoheight);
        \pgfmathparse{atan(\radius/\smallradius)}
        \path(smallcenter)++(\pgfmathresult:\smallradius)coordinate(p);
        \def\lightpath{(a)to(bprime)--(pprime)to(linfty)to(tau)--cycle}
        \fill[gray!30]\lightpath;
        \def\darkpath{(b)to(a)to(linfty)to(tau)to(p)to(linfty)--cycle}
        \fill[gray!70]\darkpath;
        \draw\lightpath;
        \draw\darkpath;
        \def\dotradius{.03}
        \foreach\x in{tau,a,linfty,pprime,bprime,b,p} \fill (\x) circle (\dotradius);
        \draw[thick] (center)++(80:\radius) arc (80:180:\radius);
        \draw[thick](\LLx,0)--(\rightmost,0);
        \draw[thick](\LLx,\Aheight)--(\rightmost,\Aheight);
        \draw [dashed,->,>=latex] (a)--(rhoABOVEa);
        \draw [dashed,->,>=latex] (b)--(rhoABOVElinfty);
        \draw [dashed,->,>=latex] (bprime)--(rhoABOVEpprime);
        \tikzstyle{lab}=[inner sep=0pt, outer sep=2pt]
        \draw(rhoABOVEa)node[lab,anchor=south]{$\rho$};
        \draw(rhoABOVElinfty)node[lab,anchor=south]{$\rho$};
        \draw(rhoABOVEpprime)node[lab,anchor=south]{$\rho$};
        \draw(a)node[lab,anchor=north east]{$a$};
        \draw(b)node[lab,anchor=north west]{$b$};
        \draw(tau)node[lab,anchor=north,outer sep=3pt]{$\tau$};
        \draw(linfty)node[lab,anchor=north west]{$l_\infty$};
        \draw(pprime)node[lab,anchor=north,outer sep=3pt]{$p'$};
        \draw(\rightmost,0)node[lab,anchor=west]{$\partial\BB^{13}$};
        \draw(\rightmost,\Aheight)node[lab,anchor=west]{$\partial A$};
        \draw(center)++(76:\radius)node[lab,outer sep=0pt]{$H$};
        \draw(p)node[lab,anchor=north west]{$p$};
      \end{tikzpicture}
    \end{center}
    \caption{The regions correspond to the four steps in the proof of
      lemma~\ref{lem-change-of-basepoint}, starting from the bottom.
      We use the upper half space model with the Leech cusp $\rho$ at
      vertical infinity and $A$ being the horoball centered there,
      whose boundary is the horizontal line.  Each of the regions
      misses $\H$ except for obvious intersection points.}
    \label{fig-change-of-basepoint}
  \end{figure}
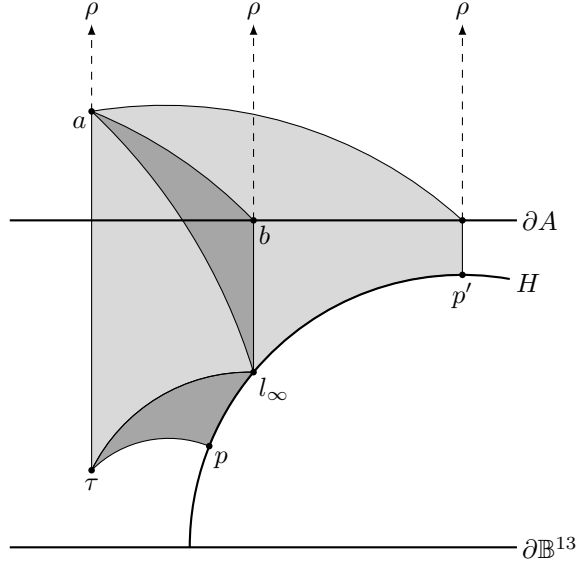

  The first step is that $M_{\tau,H}$ is equal to
  $M_{\tau,l_\infty,H}$ in $G_\tau$.  This follows from lemma~\ref{lem-homotopy-for-moving-the-point-on-the-hyperplane-ordinary-basepoint-version}
  once one checks that $\triangle\tau p l_\infty$ misses $\H$, except
  that $\geodesic{p l_\infty}$ lies in $H$ and $l_\infty$ may lie in
  additional hyperplanes.  Here $p$ means the point of $H$ closest to
  $\tau$.  This check is lemma~\ref{lem-disjointness-needed-for-step-1-of-basepoint-change}.

  The second step is that the isomorphism $G_\tau\iso G_a$ identifies
  $M_{\tau,l_\infty,H}$ with $M_{a,l_\infty,H}$.  This follows from
  lemma~\ref{lem-homotopy-for-moving-the-basepoint}, once one checks that the (complex) triangle $\triangle
  a\tau l_\infty$ misses $\H$ except at $l_\infty$.  This check is
  lemma~\ref{lem-complex-triangles-miss-mirrors-except-as-known}:  it proves the corresponding result for
  $\triangle\rho\tau l_\infty$, which contains $\triangle a\tau
  l_\infty$.

  The third step is that $M_{a,l_\infty,H}$ is equal to
  $M_{a,A,l_\infty,H}$ in $G_a$.  By taking $a$ very far up in
  figure~\ref{fig-change-of-basepoint} this becomes almost obvious.  Namely, the uniform
  distance between $\geodesic{a l_\infty}$ and the concatenation of
  $\geodesic{a b}$ and $\geodesic{b l_\infty}$ tends to $0$ as
  $a$ approaches $\rho$.  Here $b$ is the point of $A$ closest to $l_\infty$.
  Therefore, by taking $a$ close enough to $\rho$, we may take the
  uniform distance between $M_{a,l_\infty,H}$ and $M_{a,A,l_\infty,H}$
  to be arbitrarily small.  We take it small enough that the
  straight-line homotopy misses $\H$ (This uses the fact verified in lemma
  \ref{lem-step-4-line-case} that the geodesic segment joining $b$ and
  $l_{\infty}$ misses $\H$ except at $l_{\infty}$.)

  The final step is like the first: $M_{a,A,l_\infty,H}$ is equal in
  $G_a$ to $M_{a,A,H}$.  This follows from lemma~\ref{lem-homotopy-for-moving-the-point-on-the-hyperplane-cusp-version}, once one
  checks that $\triangle\rho l_\infty p'$ misses $\H$ except that
  $\geodesic{p'l_\infty}\sset H$ and that $l_\infty$ may lie in
  additional mirrors.  Here $p'$ is the point of $\H$ nearest $\rho$.
  This check is lemma~\ref{lem-step-4-line-case}.
  
  Putting these four steps together gives
  \begin{equation*}
    M_{\tau,H}
    =
    M_{\tau,l_\infty,H}
    \leftrightarrow
    M_{a,l_\infty,H}
    =
    M_{a,A,l_\infty,H}
    =
    M_{a,A,H}
  \end{equation*}
  Here the first two terms are equal in $G_\tau$ by the first 
  step of the homotopy.  The second and third terms correspond under
  the isomorphism $G_\tau\iso G_a$ induced by $\geodesic{\tau a}$, by
  the second step.  And the  last three terms are equal in $G_a$ by the
  third and fourth steps.   This finishes the proof for line-meridians.

  The point-meridian case is exactly the same.  Now $H$ is
  a point-mirror and $p$ and $p'$ are
  the projections of $\tau$ and $\rho$ to it.  And  
  we replace $l_\infty$ by $p_\infty$.  In step~1, $\triangle\tau p
  l_\infty$ gets replaced by 
  $\triangle\tau p p_\infty$, whose intersection with $\H$ 
  is also given by lemma~\ref{lem-disjointness-needed-for-step-1-of-basepoint-change}.  In step~2, $\triangle\tau\rho
  l_\infty$ gets replaced by $\triangle\tau\rho p_\infty$, whose
  intersection with $\H$ is also given by lemma~\ref{lem-complex-triangles-miss-mirrors-except-as-known}.  In step~4,
  $\triangle\rho p'l_\infty$ gets replaced by $\triangle\rho p'
  p_\infty$, whose intersection with $\H$ is given by lemma~\ref{lem-step-4-point-case}
  rather than lemma~\ref{lem-step-4-line-case}.
\end{proof}

\appendix
\section{How eight totally real triangles meet the mirrors}
\label{app-how-real-triangles-meet-mirrors}

\noindent
In this appendix we examine how certain totally real triangles in
$\overline{\BB^{13}}$ meet the mirror arrangement $\H$.  This was needed in
the proofs of lemmas
\ref{lem-configuration-3-Leech-roots-plus-1-second-shell},
\ref{lem-configuration-2-Leech-roots-plus-1-second-shell-plus-1-third-shell}
and \ref{lem-change-of-basepoint}.  Recall from
section~\ref{subsec-geodesics-and-geodesic-triangles} that if
$x,y,z\in L\tensor\C-\{0\}$ have norms and inner products in
$(-\infty,0]$, then the convex hull of the corresponding points in
  $\overline{\BB^{13}}$ is the projectivization of the convex hull of
  $x$, $y$ and~$z$ in $L\tensor\C$. In
  section~\ref{subsec-geodesics-and-geodesic-triangles} we called such
  a triangle in $\overline{\BB^{13}}$ totally real.  For totally real
  triangles we don't usually distinguish between the triangle in
  $L\tensor\C$ and the triangle in $\overline{\BB^{13}}$.  The
  following lemma is trivial but crucial. 
  
  \begin{lemma}[How to intersect a totally real triangle and a hyperplane]
    \label{lem-how-real-triangle-meets-hyperplane}
    Suppose $x,y,z\in L\tensor\C$ have norms and inner products in
    $(-\infty,0]$, $T$ is the totally real triangle $\triangle x y z$,
      and $s\in L\tensor\C$.  Then $T\cap s^\perp$ is the preimage of
      the origin under the map $T\to\C$ given by $v\mapsto\ip{v}{s}$.

      In particular, if the convex hull in $\C$ of $\ip{x}{s}$,
      $\ip{y}{s}$ and $\ip{z}{s}$ does not contain the origin, then
      $T\cap s^\perp=\emptyset$.
      \qed
  \end{lemma}

In the situation of the lemma we will write $\ip{T}{s}$ for the convex
hull in $\C$ of $\ip{x}{s}$, $\ip{y}{s}$ and $\ip{z}{s}$.  
Note that if the vertices of the triangle $\ip{T}{s}$ have co-ordinates in 
$\Q[\omega]$ or $\Q[\omega, \sqrt{3}]$, then the condition
$0 \in \ip{T}{s}$ can be checked by exact arithmetic in that field.
Most arguments in this appendix amount to showing that $0\notin\ip{T}{s}$
for various triangles~$T$ and roots~$s$.  Some of the results require
computer calculation, and others depend on properties of the Leech
lattice. For example, lemmas
\ref{lem-triangle-analysis-for-2-Leech-meridians-give-a-third}
and~\ref{lem-step-4-line-case} use the following result:

\begin{lemma}[Leech lattice points near $\lambda_6/\theta$ and
    $\lambda_9/\theta$]
  \label{lem-lattice-points-near-vectors-over-theta}
  Suppose $\lambda_6$ and $\lambda_9$ are vectors in
  $\Lambda$ with norms $6$ and~$9$.  Then the nearest points of
  $\Lambda$ to $\lambda_6/\theta$ are three in number, at
  distance~$\sqrt2$.  And the nearest points of $\Lambda$ to
  $\lambda_9/\theta$ are $36$ in number, at distance~$\sqrt3$.
\end{lemma}

\begin{proof}
The essential point is that the shortest elements of $\Lambda$ which
are congruent mod~$\theta$ to a lattice vector of norm~$6$ (resp.~$9$)
are $3$ (resp.~$36$) in number, all of norm~$6$ (resp.~$9$).  This is
the lemma on p.~153 of \cite{Wilson}.  

We abbreviate $\lambda_6/\theta$ to $C$ (for ``centroid'').  $\Lambda$
contains the vector $C-\lambda_6/\theta$, which lies at
distance~$\sqrt2$ from~$C$.  (It lies in $\Lambda$ because it is the
zero vector.)  Now consider any lattice point at
distance${}\leq\sqrt2$ from~$C$, and write it as $C+x/\theta$ with
$x\in\Lambda\tensor\C$ of norm${}\leq6$.  Since $\Lambda$ contains the
difference between this point and the one just
mentioned, it contains $(x+\lambda_6)/\theta$.  That is, $x$ lies
in $\Leech$ and is congruent to $-\lambda_6$ mod~$\theta$.  By
Wilson's result, the possibilities for $x$ are the three minimal
representatives of $-\lambda_6$'s congruence class mod~$\theta$.  This
finishes the $\lambda_6$ case.
The $\lambda_9$ case is similar.
\end{proof}

When we have a Leech cusp $\rho$ in mind, we will call a horosphere
centered there a critical horosphere if it is tangent to some mirror.
We use the same language for the (open or closed) horoball it bounds.
Recall the definition of a horosphere of height $h$ given in the discussion
following equation \eqref{eq-definition-of-height}.
If $s$ is a root then it is easy to see that the horosphere tangent to
its mirror is the one of height $\frac13|\ip{\rho}{s}|^2$.  (Compute the
inner product of $\rho$ with its projection onto $s^\perp$.) So the
heights of the first four critical horospheres are $1$, $3$, $4$
and~$7$.  The first-, second- and third-shell mirrors (defined in
section~\ref{subsec-Leech-cusps-and-Leech-roots}) are the mirrors
tangent to the first, second and third critical horoballs.

The following lemma is the first of several in this appendix having the general
form: prove that some particular triangle misses the mirrors except for some
obvious intersection points. The general strategy is to start by showing that
the triangle is covered by the union of a ball and a horoball. We enumerate
the finitely many mirrors that meet the ball, and check by direct computation
that they miss the triangle. Then we enumerate the mirrors that meet the 
horoball; there are infinitely many such mirrors, but they correspond to roots
of small height and therefore they may be parametrized.
Checking that these mirrors miss the triangle requires some
intricate analysis rather than just direct computation.

\begin{lemma}[The triangle needed in lemma~\ref{lem-configuration-3-Leech-roots-plus-1-second-shell}]
  \label{lem-triangle-analysis-for-2-Leech-meridians-give-a-third}
  In the Leech model, consider the Leech roots $x=(0;1,-\w)$ and
  $y=(\lambda_6;1,\w)$, where $\lambda_6$ is a norm~$6$ vector in
  $\Lambda$.  Define $p$, resp.\ $q$, as the projection of $\rho$ to
  $x^\perp$, resp.\ $\spanof{x,y}^\perp$.  Then the totally real
  triangle $T=\triangle\rho p q$ meets $\H$ as follows: $x^\perp$
  meets $T$ in $\geodesic{pq}$, the other three mirrors of
  $\gend{R_x,R_y}$ meet $T$ only at $q$, and all other mirrors
  miss~$T$.
\end{lemma}

\begin{proof}
  First we introduce various important points.
  Recall that
  $\rho=(0;0,1)$.  By definition,
  $$
  p=
  \rho-\frac{\ip{\rho}{x}}{x^2}x
  =
  \bigl(0;{1}/{\theta},-{\wbar}/{\theta}\bigr),
  $$
  and one can compute $p^2=\ip{\rho}{p}=-1$.  For later
  calculations, we also write $p$ in the form~\eqref{eq-definition-of-vector-s}, writing
  $\sigma_p$, $m_p$, $N_p$ and $\nu_p$ in place of their unsubscripted forms
  that appear there.  Obviously we have $\sigma_p=0$ and
  $m_p=1/\theta$. And  $N_p$ is $p^2$, which we just computed to be~$-1$.  One can
  then solve for $\nu_p$, namely $\nu_p=-\frac{\theta}{18}$.  One can
  verify the following formula for
  $q$ by checking that it is orthogonal to $x$ and $y$:
  $$
  q
  =
  \rho-\bigl(\w x-\wbar y\bigr)
  =
  \bigl(\wbar\lambda_6;\thetabar,\theta\wbar\bigr).
  $$ One can check 
  $q^2=\ip{\rho}{q}=\ip{p}{q}=-3$.  Since the vectors we have chosen
  to represent the three vertices of $\triangle\rho p q$ have inner products
  in $(-\infty,0]$, we will be able to apply
    lemma~\ref{lem-how-real-triangle-meets-hyperplane}.  Also, writing
    $q$ in the form~\eqref{eq-definition-of-vector-s}, just as we did
    for $p$, gives $\sigma_q=\wbar\lambda_6$, $m_q=\thetabar$,
    $N_q=-3$ and $\nu_q=-\frac{\theta}{2}$.
    
    The mirrors of $\gend{R_x,R_y}$ obviously meet $T$ as claimed, so
    it suffices to show that no other mirrors meet~$T$.  By
    construction, $p$ lies on the boundary of the first critical
    horoball.  And $\height(q)=3$, so $q$ lies on the boundary of the
    second.  Therefore only one mirror not in the first shell could
    meet the triangle, and then only at $q$.  This mirror corresponds
    to the second-shell root $\w x-\wbar y$ in $\spanof{x,y}$ that we
    left unnamed in
    lemma~\ref{lem-configuration-3-Leech-roots-plus-1-second-shell}\eqref{item-three-mirrors-are-Leech-mirrors}.
    So it suffices to show that no Leech mirrors meet the triangle
    except the three coming from $\spanof{x,y}$.

    Our strategy is to write a general Leech root $s$ in the form
    \eqref{eq-Leech-roots-in-Leech-model}, namely
    $$
    s
=
\biggl(
\sigma
;
1
,
\theta\Bigl(\frac{\sigma^2-3}{6}+\nu\Bigr)
\biggr)
    $$
and compute its inner products with
    $\rho$, $p$ and~$q$, and then apply
    lemma~\ref{lem-how-real-triangle-meets-hyperplane} to gain control over
    $\sigma$ and~$\nu$.
    To find $\ip{p}{s}$ and $\ip{q}{s}$ we appeal to formula
    \eqref{eq-general-inner-product-formula}; being able to apply this
    is the reason we computed $\sigma_p,\dots,\nu_q$.  By
    \eqref{eq-general-inner-product-formula}, we have
    \begin{align}
      \label{eq-ip-of-p-and-s}
      \ip{p}{s}
      &{}=
      \textstyle
      \frac{\theta}{6}\sigma^2
      +\Bigl(\frac{1}{2}-\theta\nu\Bigr)
      \\
      \label{eq-ip-of-q-and-s}
      \ip{q}{s}
      &{}=
      \textstyle
      \frac{\theta}{2}\Bigl(
        \bigl(\sigma+\wbar\frac{\lambda_6}{\theta}\bigr)^2-2\Bigr)
      -\Bigl(\theta\Im\bigip{\sigma}{\wbar\frac{\lambda_6}{\theta}}
        -\frac{3}{2}+3\theta\nu\Bigr)
    \end{align}
    In each of these, the first term is imaginary and the second is real.
    Our goal is to show that only three possible pairs
    $\sigma$, $\nu$ allow $0$ to lie in the triangle
    $\ip{T}{s}\sset\C$ whose vertices are \eqref{eq-ip-of-p-and-s}, \eqref{eq-ip-of-q-and-s} and
    $\ip{\rho}{s}=\theta$.

    First we consider the case $\sigma=0$.  Then \eqref{eq-ip-of-p-and-s} and \eqref{eq-ip-of-q-and-s}
    simplify dramatically, and $\ip{T}{s}$ is the triangle in $\C$
    with vertices  
    $\frac{1}{2}-\theta\nu$,
    $3(\frac{1}{2}-\theta\nu)$ and $\theta$.  Since the last vertex is above the
    real axis, while
    the first two are real and differ by a factor of~$3$, $\ip{T}{s}$
    can only contain~$0$ if both of its first two vertices are~$0$.
    This forces $\nu=1/2\theta$, which leads to $s=x$, whose mirror we
    already know meets~$T$.

    Now we suppose $\sigma\neq0$.  So \eqref{eq-ip-of-p-and-s} lies above the real
    axis, just as $\ip{\rho}{s}$ does.  
    If $\sigma$ lies at distance${}>\sqrt2$ from
    $-\wbar\lambda_6/\theta$ then \eqref{eq-ip-of-q-and-s}  also lies above
    the real axis.  In this case it is obvious that
    $0\notin\ip{T}{s}$.  By
    lemma~\ref{lem-lattice-points-near-vectors-over-theta}, the only
    other possibility is that $\sigma$ is one of the three elements of
    $\Lambda$ that lie at distance~$\sqrt2$ from
    $-\wbar\lambda_6/\theta$, in which case $\ip{q}{s}$ is real.
    One of these nearest neighbors is $0\in\Lambda$, which we treated
    in the previous paragraph.  In each of the other two cases, 
    $\ip{p}{s}$ and $\ip{\rho}{s}$ are still above the $x$-axis and 
    $\ip{q}{s}$ is on the $x$-axis.
    So the only way the origin can lie in $\ip{T}{s}$ is for it to be the
    vertex $\ip{q}{s}$.  Then  $\nu$ is determined in terms of $\sigma$
using equation \eqref{eq-ip-of-q-and-s}, since $\ip{q}{s} = 0$.
    We have shown that there are at most three Leech mirrors that meet
    the triangle.  Since we know three Leech mirrors that do meet it,
    coming from roots in $\spanof{x,y}$, the proof is complete.    
\end{proof}

\begin{lemma}[Three triangles needed in lemma~\ref{lem-configuration-2-Leech-roots-plus-1-second-shell-plus-1-third-shell}]
  \label{lem-triangle-analysis-for-point-and-line-meridians-give-Leech-meridian}
  In the Leech model, consider the Leech roots $x=(0;1,-\w)$ and
  $z=(\lambda_9;1,\theta)$ and define the second shell root $y$ as
  $\wbar x-z$.  Write $X$, $Y$ and $Z$ for the projections of $\rho$
  to $x^\perp$, $y^\perp$ and $z^\perp$, and $q$ for the projection of
  $\rho$ to $\spanof{x,y}^\perp$.  Then the only mirrors of $\H$ which
  meet any of the totally real triangles $\triangle\rho q X$, $\triangle\rho q Y$ and
  $\triangle\rho q Z$ are the four mirrors of $\gend{R_x,R_y}$.
\end{lemma}

\begin{proof}
  We begin by finding basic
  data about various important points.  In
  addition to $x$ and $z$ given in the statement, we have
\begin{equation*}
  \rho=(0;0,1)
  \qquad\hbox{and}\qquad
  y=(-\lambda_9;\theta\w,2\wbar)
\end{equation*}  
where $\lambda_9$ is some norm~$9$ vector in the Leech lattice.
Recall that $x$ and $z$ are Leech roots, so
$\ip{\rho}{x}=\ip{\rho}{z}=\theta$.  Also,
$\ip{x}{z}=-\frac32+\frac\theta2$ and $y$ is a second-shell root with
$\ip{\rho}{y}=3\wbar$.  
The projections $X$, $Y$, $Z$ are
\begin{align*}
  X
  &=\rho-\frac{\ip{\rho}{x}}{x^2}x
  =\rho+x/\theta
  =(0;1/\theta,-\wbar/\theta)
  \\
  Y&{}=\rho-\frac{\ip{\rho}{y}}{y^2}y
  =\rho-\wbar y
  =(\wbar\lambda_9;\thetabar,1-2\w)
  \\
  Z&{}=\rho-\frac{\ip{\rho}{z}}{z^2}z
  =\rho+z/\theta
  =(\lambda_9/\theta;1/\theta,2)
\end{align*}
Using these one can check
\begin{align*}
  X^2=Z^2&{}=-1
  &
  \ip{\rho}{X}=\ip{\rho}{Z}&{}=-1
  &
  \ip{X}{y}&{}=1+3\wbar
  \\
  Y^2&{}=-3
  &
  \ip{\rho}{Y}&{}=-3
  &
  \ip{Y}{x}&{}=-\w\theta.
\end{align*}
Because
$X\perp x$ and $Y\perp y$, the following gives a formula for the
projection $q$ of $\rho$ to $\spanof{x,y}^\perp$:
\begin{equation*}
  q=(2\w-1)\Bigl(\rho
  -\frac{\ip{\rho}{x}}{\ip{Y}{x}}Y
  -\frac{\ip{\rho}{y}}{\ip{X}{y}}X\Bigr)
  =
  \Bigl((2\wbar-\w)\lambda_9;3\thetabar,3-3\w\Bigr)
\end{equation*}
(The initial factor $2\w-1$ makes $\ip{\rho}{q}$ negative,
and also makes $q$ a primitive lattice vector.)  Using this one can
check
\begin{equation*}
  q^2=-18
  \quad\hbox{and}\quad
  \ip{q}{\rho}
  =\ip{q}{X}
  =\ip{q}{Y}
  =\ip{q}{Z}
  =-9
\end{equation*}
We will be able to apply
lemma~\ref{lem-how-real-triangle-meets-hyperplane} to $\triangle\rho q
X$ because
$\rho$, $q$ and $X$ have negative inner products.  And similarly with
$Y$ or $Z$ in place of~$X$.

We first claim that $\triangle\rho q X\cup\triangle\rho q
Y\cup\triangle\rho q Z$ lies in the interior of the fourth critical
horoball.  It suffices to show that $X$, $Y$, $Z$ and $q$ do.  
By definition, $Y$ lies on the
boundary of the 2nd critical horoball and $X$ and $Z$ lie on the
boundary of the 1st.  For $q$, we use
$q^2=-18$ and $\ip{q}{\rho}=-9$ to get $\height(q)=9/2<7$, as desired.

Let $Q_X$ be  $\triangle\rho q X$ minus the first (open) critical
horoball around~$\rho$, and similarly for $Q_Y$ and $Q_Z$.  By working
in the hyperbolic plane containing $\triangle\rho q X$, it is obvious that
the point of $Q_X$ furthest from $q$ is~$X$.  The distance formula
\eqref{eq-complex-hyperbolic-metric} gives $d(q,X)=\cosh^{-1}\sqrt{9/2}$.  Similarly, $Z$ is the
point of $Q_Z$ furthest from $q$, and $d(q,Z)$ is also
$\cosh^{-1}\sqrt{9/2}$.  Finally, the furthest point of $Q_Y$ from $q$
is where $\geodesic{Y\!\rho}$ intersects the
 first critical horosphere.  We can find
this point by parameterizing $\geodesic{Y\!\rho}-\{\rho\}$ by
$Y+t\rho$ for $t\in[0,\infty)$.  The intersection point is defined by the condition
  $\height_\rho(Y+t\rho)=1$.  Writing this out explicitly using
  \eqref{eq-definition-of-height} and solving for $t$ gives $t=1$.  So the point is $Y+\rho$.
  The distance formula gives $d(q,Y+\rho)=\cosh^{-1}\sqrt2$.  This is
  smaller than $\cosh^{-1}\sqrt{9/2}$, so we conclude that $Q_X\cup Q_Y\cup
  Q_Z$ lies in the closed ball around $q$ of radius
  $\cosh^{-1}\sqrt{9/2}$.

  Now suppose $s$ is a root whose mirror meets $\triangle\rho q
  X\cup\triangle\rho q Y\cup\triangle\rho q Z$.  Since these triangles
  lie in the fourth (open) critical horoball, $s$ must be a 1st, 2nd or 3rd
  shell root.  By scaling $s$ we may
  suppose $\ip{\rho}{s}\in\{\theta,3,2\theta\}$.  We will work with
  many projections of $s$, so we name them in a uniform way:
  \begin{center}
    \setlength{\tabcolsep}{0pt}
    \begin{tabular}{rll}
      {$s_q$}&${}=\frac{\ip{s}{q}}{q^2}q$
      &\ is its projection to the span of $q$,
      \\
      {$s_{q\perp}$}&${}=s-s_q$
      &\ is its projection to the orthogonal complement of this,
      \\
      $s_{xy}$&
      &\ is its projection to the span of $x$ and $y$,
      \\
        {$s_{qxy}$}&${}=s_q+s_{xy}$
        &\ is its projection to the span of $q$, $x$ and~$y$,
        \\
          {$s_{qxy\perp}$}&${}=s-s_{qxy}$
          &\ is its projection to the orthogonal complement of this,
          \\
            {$s_{\rho x}$}&
            &\ is its projection to the span of $\rho$ and $x$,
  \\
 $s_9$&&\ is its
  projection to the span of $\lambda_9\in\Lambda$, and
            \\
            $s_\Lambda$&${}=s_9+s_{qxy\perp}$
            &\ is its projection to the summand $\Lambda=\spanof{\rho,x}^\perp$ of $L$.
    \end{tabular}
  \end{center}

  \noindent
  Because all mirrors miss the first (open) critical horoball, $s^\perp$
  must meet $Q_X\cup Q_Y\cup Q_Z$.  Because this set lies in the
  closed $\bigl(\cosh^{-1}\sqrt{9/2}\bigr)$-ball around $q$, we get
   $d(q,s^\perp)\leq\cosh^{-1}\sqrt{9/2}$.  Using the
  second distance formula \eqref{eq-distance-point-to-hyperplane} gives
  $$
  \sinh^{-1}\sqrt{-\frac{\bigl|\ip{q}{s}\bigr|^2}{q^2 s^2}}
  \leq\cosh^{-1}\sqrt{9/2}.
  $$
  Now $s^2=3$, $q^2=-18$ and $\sinh(\cosh^{-1}(?))=\sqrt{(?)^2-1}$
  give us the inequality $\bigl|\ip{q}{s}\bigr|^2\leq 189$.  
  Using $q^2=-18$ a second time shows that the most negative  $s_q^2$ can be is
  $-21/2$.  Next, $3=s^2=s_q^2+(s_{q\perp})^2$, so $(s_{q\perp})^2\leq
  27/2$.  Since $s_{xy}$ is a projection of $s_{q\perp}$ inside the
  positive definite space $q^\perp$, it follows that
  $s_{xy}^2\leq27/2$.

  We recall that $\spanof{x,y}$ is a copy of the Eisenstein lattice
  $D_4^\E$, for which one can introduce explicit coordinates, for
  example as in lemma~\ref{lem-generators-for-braid-group-of-D-4}.  Since all inner products in $L$ are
  divisible by $\theta$, we have $s_{xy}\in\theta(D_4^\E)^*$.  One can
  check that $\theta(D_4^\E)^*$ is a copy of $D_4^\E$ scaled to have
  minimal norm $3/2$.  An $\E$-basis consists of $(\theta x+y)/2$ and
  $(\theta y-x)/2$.
  Let $S_{xy}$ be the set of vectors in this
  lattice with norm${}\leq27/2$.  We have shown that the projection
  $s_{xy}$ of
  $s$ must lie in $S_{xy}$.  This set can be enumerated on a computer
  and turns out to have size~$937$.  (PARI has a built-in function to do this. It is
  most natural to rescale by multiplying all inner products by $4/3$, to work
  with a copy of the standard $D_4$ lattice over~$\Z$.  Then one enumerates all
  lattice vectors of norm~$\leq\frac{4}{3}\cdot\frac{27}{2}=18$. The
  number~$937$ matches what one expects from \cite[Table~4.8]{Conway-Sloane}.)

  Now we consider the possibilities for $s_{qxy}$.  We claim it
  lies in
  $$
  S_{qxy}=\Bigset{s_{xy}+tq}{s_{xy}\in S_{xy}\hbox{ and }
    t=\frac{\bigl(\hbox{one of $\thetabar$, $3$, and $2\thetabar$}\bigr)-\ip{s_{xy}}{\rho}}{-9}}
  $$
  This is easy: $s_{qxy}$ equals $s_{xy}$ plus some multiple of
  $q$, and the  multiple of $q$ is determined by the condition
  $\ip{\rho}{s_{qxy}}=\ip{\rho}{s}\in\{\theta,3,2\theta\}$.
  So $S_{qxy}$ has order $937\cdot3=2811$.

  We are working under the assumption that $s$ is orthogonal to some
  point of $Q_X\cup Q_Y\cup Q_Z$.  Now, $s_{qxy}$ has the same inner
  products as $s$ with all elements of the span of $q,x,y$, which
  contains $Q_X\cup Q_Y\cup Q_Z$.  It follows that $s_{qxy}^\perp$
  meets one of the triangles $\triangle\rho qX$, $\triangle\rho qY$
  and $\triangle\rho qZ$.  Of the 2811 possibilities for $s_{qxy}$,
  only~$460$ satisfy this condition.  And the following considerations
  further restrict the possibilities.

  Because the span of $q$, $x$ and $y$ is the same as the span of
  $\rho$, $x$ and $\lambda_9$ and $\lambda_9$ is orthogonal
  to the span of $\rho$ and $x$, we have
  $s_{q x y} = s_9 + s_{x \rho}$ and hence 
  $s = s_{q x y \perp}  + s_9 + s_{x \rho}$.
  It follows that $s_{\Lambda} = s_{q x y \perp} + s_9$.
 Note that the projection $s_9$ of $s$ to
  $\C\lambda_9$ coincides with the corresponding projection of
  $s_{qxy}$.  Furthermore, the equalities
  \begin{gather*}
    s_\Lambda^2=(s_{qxy\perp})^2+s_9^2=(3-s_{qxy}^2)+s_9^2
    \\
    \ip{s_\Lambda}{\lambda_9}=\ip{s_{qxy}}{\lambda_9}
  \end{gather*}
  show that $s_{qxy}$ determines $s_\Lambda^2$ and
  $\ip{s_\Lambda}{\lambda_9}$.  Therefore it determines all inner
  products in the span of $s_\Lambda$ and $\lambda_9$.  Also, note
  that $s_\Lambda$ lies in $\Lambda$, not just $\Lambda\tensor\C$,
  because $\Lambda$ is an orthogonal summand of~$L$.

  Of the 460 possibilities for $s_{qxy}$, 449 lead to
  $\bigl|\ip{s_\Lambda}{\lambda_9}\bigr|^2>9s_\Lambda^2$, which
  violates the Cauchy-Schwarz inequality.  Of the remaining 11
  possibilities, 4 lead to  $s_\Lambda^2=3$, which contradicts the
  fact that $\Lambda$ has minimal norm~$6$.  Of the remaining 7
  possibilities, 3 lead to the similar contradiction that $(\zeta
  s_\Lambda-\lambda_9)^2=3$ for some unit~$\zeta$ of~$\E$.  The
  remaining 4 possibilities for $s_{qxy}$ turn out to be $x$,
  $\wbar y$, $z$ and $x+z$, all of which are roots of
  $\spanof{x,y}$.  In particular, they all have norm~$3$.  It follows
  that
  $$
  s_{qxy\perp}^2=s^2-s_{qxy}^2=3-3=0.
  $$
  Since $s_{qxy\perp}$ lies in the positive-definite space $q^\perp$,
  we conclude $s_{qxy\perp}=0$.  Therefore $s=s_{qxy}$, which we have
  already observed is a root of $\spanof{x,y}$.
\end{proof}

For the rest of the results in this section we will need to understand
the mirrors near the intersection point $p_\infty$ of all $13$
point-mirrors.  Just as we spoke of critical horoballs around $\rho$,
we will speak of critical balls around $p_\infty$.  They are defined
the same way: the balls centered at $p_\infty$ and tangent to
some mirror.  The radii of the critical balls are called critical
radii. Even though some mirrors pass through $p_\infty$, we
don't regard $0$ as a critical radius.

For this analysis we use the $P^2\F_3$ model, in which
$p_\infty=(\thetabar;0,\dots,0)$.  For a given root $s$ of $L$, the
distance formula \eqref{eq-distance-point-to-hyperplane} gives
$d(p_\infty,s^\perp)$ in terms of $p_\infty^2=-3$, $s^2=3$ and
$|\ip{p_\infty}{s}|^2$.  Since the inner product lies in $\theta\E$ and
hence has norm in $3\cdot\{0,1,3,4,7,9,12,\dots\}$, one can work out
the critical radii
as
\begin{equation*}
r_1,r_2,r_3,r_4,r_5,r_6,\ldots=
\sinh^{-1}\sqrt{\hbox{($1,3,4,7,9,12,\dots$)}/3}
\end{equation*}
This assumes that there are indeed roots in $L$ having the appropriate
inner products with $p_\infty$, which is easy to check in the range we
care about (see lemma~\ref{lem-mirrors-near-a-13-point} below).  Numerically, the first six
critical radii are approximately $.549$, $.881$, $.987$, $1.210$,
$1.317$ and $1.444$.  We call the mirrors tangent to the $n$th
critical ball around $p_\infty$ the batch~$n$ mirrors.  For
completeness we also refer to the point-mirrors as the batch~$0$
mirrors.  Lemma~\ref{lem-mirrors-near-a-13-point} below describes
batches $0,\dots,5$ explicitly, but until
lemma~\ref{lem-step-4-line-case} we only need batches $0$, $1$
and~$2$.  These are easy to enumerate and appear in
table~\ref{tab-roots-near-w-P}.  As an example application we reprove
the following result from \cite[Prop.~1.2]{Basak-bimonster-1}.  Although this
result helps motivate the main theorem of the paper (theorem~\ref{t-26-meridians-based-at-tau-generate}),
it is not logically necessary for us.

\begin{table}
\begin{tabular}{lllrl}
\bf distance\rlap{($s^\perp$,$p_\infty$)}
&&\bf representative
&\bf $\bigl|3^{13}{:}$\rlap{$L_3(3)$-orbit$\bigr|$}
\\
$0$
&&$(0;\theta,0^{12})$
&$13$
\\
\noalign{\smallskip}
$\sinh^{-1}\sqrt{1/3}$
&$\!\!\!\approx.549$
&$(1;1^4,0^9)$
&$1053$
&$\!\!\!=13\cdot3^4$
\\
\noalign{\smallskip}
$\sinh^{-1}\sqrt{3/3}$
&$\!\!\!\approx.881$
&$(\theta;1^3,-1^3,0^7)$
&$113724$
&$\!\!\!=156\cdot3^6$
\\
&&$(\theta;\theta,\thetabar,0^{11})$
&$1404$
&$\!\!\!=13\cdot12\cdot3^2$
\\
&&$(\theta;\theta,\theta,0^{11})$
&$702$
&$\!\!\!=\binom{13}{2}\cdot3^2$
\\
&&$(\theta;\thetabar,\thetabar,0^{11})$
&$702$
\end{tabular}
\smallskip
\caption{The $3^{13}{:}L_3(3)$-orbits of batch $0$, $1$ and $2$ roots,
  up to units.
  Their mirrors pass through or near the $13$-point $p_\infty$.  The
  last~$13$ coordinates, read modulo~$\theta$, must give
  an element of the line code.%
}
\label{tab-roots-near-w-P}
\end{table}

\begin{lemma}[Mirrors near a $26$-point {\cite[prop.~1.2]{Basak-bimonster-1}}]
\label{lem-mirrors-near-the-26-point}
The mirrors of $L$ closest to the $26$-point $\tau$ are the point- and line-mirrors, at distance
$\sinh^{-1}\bigl(6+8\sqrt3\bigr)^{-1/2}\approx.223$. 
\end{lemma}

\begin{proof}
First we use \eqref{eq-complex-hyperbolic-metric} to compute
$$
d(\tau,p_\infty)
=\cosh^{-1}\sqrt{\frac{\bigl|\ip{\tau}{p_\infty}\bigr|^2}{
    p_\infty^2\, \tau^2}}
=\cosh^{-1}\sqrt{\frac{19+8\sqrt3}{6+8\sqrt3}}
\approx.740.
$$
That the  point- and line-mirrors lie at distance
$\sinh^{-1}\bigl(6+8\sqrt3\bigr)^{-1/2}\approx.223$ from $\tau$ is a similar
calculation, using \eqref{eq-distance-point-to-hyperplane} in place of \eqref{eq-complex-hyperbolic-metric}. 
Any  mirror which passes as near or nearer to $\tau$ as these
do must lie at
distance at most
$$
d(\tau,p_\infty)+\sinh^{-1}\bigl(6+8\sqrt3\bigr)^{-1/2}\approx.740+.223
=.963
$$
from~$p_\infty$.   This is less than
$r_3=\sinh^{-1}\sqrt{4/3}\approx.987$, so all such mirrors occur
in batches $0$, $1$ and~$2$.   Our computer iterated over 
these batches and found that the mirrors closest to $\tau$ are just
the point and line mirrors.
\end{proof}

\begin{lemma}[Two triangles needed for step~$1$ of lemma~\ref{lem-change-of-basepoint}]
\label{lem-disjointness-needed-for-step-1-of-basepoint-change}
Let $s$ be a point-root and $x$ be the projection of the $26$-point $\tau$ to its
mirror.  
Then the only mirrors of $L$ that meet 
the totally real triangle $\triangle\tau x p_\infty$  are the point-mirrors.
The intersection of $\triangle\tau x p_\infty$ with each of them is $p_\infty$, except
that the intersection with $s^\perp$ is $\overline{x p_\infty}$.  The
same results hold if $s$ is a line-mirror, provided we replace $p_\infty$ by~$l_\infty$.
\end{lemma}

\begin{proof}
  We prove the point-mirror case; the line-mirror case follows by
  applying an element of $L_3(3):2$ that swaps the point- and
  line-mirrors.  We gave $d(p_\infty,\tau)\approx.740$ in the previous
  lemma, and $x$ is closer to $p_\infty$ than $\tau$ is.  So the
  triangle lies within $d(p_\infty,\tau)$ of $p_\infty$.  This is less
  than $r_2=\sinh^{-1}1\approx.881$, so only the batch $0$ and~$1$
  mirrors might meet the triangle.  The batch~$0$ mirrors are the
  point-mirrors, which obviously meet the triangle as stated.  For the
  batch~$1$ roots we used our computer to check (using
  lemma~\ref{lem-how-real-triangle-meets-hyperplane}) that none of
  their mirrors meet the triangle.  When doing this we
  replaced the vector $(\thetabar;0^{13})$ representing $p_\infty$ by
  $\theta p_\infty=(3;0^{13})$.  Then the vectors representing the
  $\tau$, $x$ and $p_\infty$ have negative inner products, so
  lemma~\ref{lem-how-real-triangle-meets-hyperplane} applies.
\end{proof}

So far we have not needed a concrete description of the Leech lattice.
But we need one to check that $\H$ is disjoint from the homotopies in
step~4 of the proof of lemma~\ref{lem-change-of-basepoint}.  We will
use Wilson's $L_3(3)$-invariant model from \cite[p.~188]{Wilson}.  Its
description involves $\psi=1-3\wbar\in\E$ from \cite[p.~154]{Wilson},
and uses $13$ coordinates in $\E$, indexed by the points of $P^2\F_3$.
Namely, $(x_1,\dots,x_{13})\in\E^{13}$ lies in $\Lambda$ just if
$x_1+\cdots+x_{13}=0$, all coordinates are congruent modulo~$\psibar$,
and the element of $\F_3^{13}$ got by reducing the components
mod~$\theta$ is an element of the line difference code described in
section~\ref{t-P2-F3-model-of-L}.  The inner product is the usual one
divided by~$13$.  For completeness we record:

\begin{lemma}
  \label{lem-it-is-the-Leech-lattice}
  The lattice just described is isometric to the complex Leech lattice.
\end{lemma}

\begin{proof}
  We will use the vectors $P_i$ and $L_j$ from the next lemma, and
  their inner product information from the remark after it.  But we
  won't use that lemma itself.   Write $\Lambda$ for the $\E$-lattice
  just defined. We  will show that it is isometric to the complex
  Leech lattice, which we have denoted $\Lambda$ elsewhere in the
  paper.  First, define $\delta_{ij}=L_i-L_j$ and
  $\varepsilon_{ij}=(P_i-P_j)/\psi$.  The $\varepsilon_{ij}$ have the
  form $(0,\dots,0,\psibar\theta,0,\dots,0,-\psibar\theta,0,\dots,0)$,
  so they lie in $\Lambda$.  We claim that $\Lambda$ is spanned by the
  $\delta_{ij}$, $\varepsilon_{ij}$ and any one $P_i$, say $P_1$.  To
  see this, given $x\in\Lambda$, the reduction of its components
  modulo~$\theta$ lies in the line difference code.  By adding suitable
  multiples of $\delta_{ij}=L_i-L_j$, we may suppose without loss of
  generality that this codeword is the zero codeword.  That is, all
  coordinates are divisible by~$\theta$.  By adding a multiple of
  $P_1$, we may suppose that the last component is~$0$.  It follows
  that all components are zero mod~$\psibar$.  Since all coordinates
  are divisible by $\psibar\theta$, and the coordinate sum is zero,
  $x$ may be expressed as a linear combination of the
  $\varepsilon_{ij}$.

  Second, all inner products in $\Lambda$ lie in $\theta\E$---in
  particular, $\Lambda$ is integral.  This is just a computation using
  the data in remark~\ref{rem-data-about-P-and-L}.  Third, the determinant of $\Lambda$ is
  $3^6$.  
 To see this, note that  
  the $\E$-span of
  the $\varepsilon_{ij}$ is a scaled copy of $\E\tensor A_{12}$, with
  determinant~$3^{12}\cdot13$.  (Recall the factor $\frac{1}{13}$ in
  $\Lambda$'s inner product.)  Adjoining the
  $\delta_{ij}$ gives a larger lattice, whose quotient by this copy of
  $\E\tensor A_{12}$ is isomorphic to the line difference code $\F_3^6$, and whose determinant
  is $3^6\cdot13$.  Finally, adjoining $P_1$ reduces the determinant
  to~$3^6$.  Since $\Lambda$ has determinant~$3^6$, and all inner
  products are divisible by~$\theta$, it is an Eisenstein Niemeier
  lattice (cf.\ \cite[Sec.~2]{Allcock-Y555}).  And the only Eisenstein Niemeier lattice
  whose isometry group contains $L_3(3)$ is the Leech lattice \cite[Thm.~4]{Allcock-Y555}.
\end{proof}

\begin{correction}
  This proof is  similar to that of Lemma~6 in
	\cite{Allcock-Y555}, which asserts that there is a unique $L_3(3)$-invariant
	integral lattice properly containing $\theta\E\tensor A_{12}$.  This is wrong
	because it neglects the $13$-part of the discriminant group of $\theta\E\tensor A_{12}$.  A counterexample
	is $\Lambda$ in the form above.  That lemma was used only once in \cite{Allcock-Y555},
	in the proof of Lemma~$9$, where it is used to recognize a certain lattice as
	the complex Leech lattice.  That result can be saved by addressing the $13$-part 
	in a manner similar to the proof above.
\end{correction}

We will need to understand the lattice points near a particular point
$C$ of $\Lambda\tensor\C$ that is not in $\Lambda$ itself.  This will allow
us to write down Leech model versions of the point-roots, line-roots
and $13$-points, and thereby identify the $P^2\F_3$ and Leech models
of~$L$.  The name ``centroid'' in the next lemma is explained after
lemma~\ref{lem-point-and-line-roots-in-Leech-model}.

\begin{lemma}[The ``centroid'' $C$]
  \label{lem-the-center}
  For $i=1,\dots,13$ define $P_i$ as the Leech lattice vector
  $(12\theta,\thetabar^{12})$ with
  the $12\theta$ in the $i$th position, and define $C=-P_1/\psi$.  Then
  the only points of $\Lambda$ at
  distance${}<\sqrt{42/13}$ from~$C$ are the
\begin{equation*}
  C+\frac{1}{\psi}P_i
  =
  \frac{P_i-P_1}{\psi}
  =
  \begin{cases}
    (0,\dots,0)
    &
    \hbox{if $i=1$}
    \\
    (-\psibar\theta,0,\dots,0,\psibar\theta,0,\dots,0)
    &
    \hbox{if $i=2,\dots,13$}
  \end{cases}
\end{equation*}
all of which lie at distance~$\sqrt{36/13}$.

Similarly, for $j=1,\dots,13$ define $L_j$ as the Leech lattice vector
$\bigl((-9)^4,4^9\bigr)$ with the
$(-9)$'s in the positions corresponding to the points of $P^2\F_3$
that lie on the $j$th line.   Then the 
only points of
$\frac{1}{\theta}\Lambda$ at distance${}<\sqrt{14/13}$ from $C$ are 
the
\begin{equation*}
  C-\frac{\wbar}{\theta\psi}L_j
  =
  \begin{cases}
    \frac{1}{\theta}\bigl(
    \phantom{4-8\w}\llap{$-9\w$},
    -3\hbox{\rm\ (three times)},
    -\wbar\hbox{\rm\ (\rlap{nine}\phantom{eight} times)}\bigr)
    &
    \hbox{if $P_1\in L_j$}
    \\
    \frac{1}{\theta}\bigl(
    4-8\w,
    -3\hbox{\rm\ (\rlap{four}\phantom{three} times)},
    -\wbar\hbox{\rm\ (eight times)}\bigr)
    &
    \hbox{otherwise}
  \end{cases}
\end{equation*}
all of which
lie at distance~$\sqrt{12/13}$.
\end{lemma}

\begin{numberedremark}
  \label{rem-data-about-P-and-L}
  Even though the $P_i$ and $L_j$ are vectors, we write ``$P_i\in
  L_j$'' as shorthand for ``the $i$th point and $j$th line of
  $P^2\F_3$ are incident''.  The following data is useful in later
  calculations.  The $P_i$ and $L_j$ have norm~$36$, and for $i\neq j$
  we have $\ip{P_i}{P_j}=\ip{L_i}{L_j}=-3$.  Also,
  $\ip{P_i}{L_j}=-9\theta$ or~$4\theta$ and
  $\bigl(C+\frac{1}{\psi}P_i\bigr)-\bigl(C-\frac{\wbar}{\theta\psi}L_j\bigr)$
  has norm $3$ or~$4$, both according to whether or not $P_i\in L_j$.
  In particular, since $P_1\in L_1$ and $C+P_1/\psi=0$, we have
  $(C-\frac{\wbar}{\theta\psi}L_1)^2=3$. 
\end{numberedremark}

\begin{proof}
  To show that these vectors lie in $\Lambda$
  (resp.\ $\frac{1}{\theta}\Lambda$), one just checks the conditions
  in Wilson's definition.  Their distances from~$C$ are as stated
  because $P_i^2=L_j^2=36$ for all $i$ and~$j$.  It is easy to check
  that all the $P_i$ (resp.\ all the $L_j$) are congruent mod~$\psi$.
  The rest of the proof is like the proof of lemma~\ref{lem-lattice-points-near-vectors-over-theta}.  That is,
  the lemma follows from the claim: the $P_i$ (resp.\ $L_j$) are the
  only norm${}<42$ representatives of their mod~$\psi$ congruence
  class.  We treat the $P_i$ case;  the $L_j$ case is the
  same.
  
  Let $v$ be a Leech vector such that $(v - C)^2 < 42/13$.
  Let $x = \psi( v - C)$. 
  Then $x\in\Lambda$ has norm${}<42$ and is distinct from all the
  $P_i$, but congruent to them mod $\psi$.  As a nonzero member of
  $\psi\Lambda$, $x-P_i$ has norm at least~$6\cdot13=78$.  Since
  $P_i^2=36$ and $x^2<42$, the angle between $x$ and $P_i$ is obtuse.
  That is, in the real Euclidean space underlying $\Lambda\tensor\C$,
  $x$ and $P_i$ have negative inner product.  This holds for all $i$.
  So $x$ has negative inner product (in this real Euclidean space)
  with $P_1+\cdots+P_{13}=0$, which is absurd.
\end{proof}

From \ref{subsec-Leech-model-of-L} recall that the
Leech model for our lattice is
$L \simeq \Lambda \oplus H$
where $H$ has gram matrix $\bigl(\begin{smallmatrix}0&\thetabar\\\theta&0\end{smallmatrix}\bigr)$.
Wilson's definition of $\Lambda$ gives an $L_3(3)$ action on $L$.
The $P^2 \F_3$ model of $L$ gives another $L_3(3)$ action on $L$.
We should note that these two $L_3(3)$'s are not conjugate in 
$\op{Aut}(L)$. 
The lattice pointwise fixed by the two $L_3(3)$'s are
$H$ and $(\E p_{\infty}  +  \E l_{\infty})$ respectively
and these two lattices are not isometric.

Now we can write down the point- and line-roots in the Leech model.
We define $p_i$ to be the Leech root obtained by using
$\sigma=C+\frac{1}{\psi}P_i$ and $\nu=1/2\theta$ in \eqref{eq-Leech-roots-in-Leech-model}.
Explicitly, 
\begin{align*}
  \label{eq-point-roots-in-Leech-model}
  p_1
  &{}=
  \bigl(0,\dots,0;1,-\w\bigr)
  \\
  p_i
  &{}=
  \bigl(-\psibar\theta,0,\dots,0,\psibar\theta,0,\dots,0;1,-\wbar\bigr)      
  \hbox{\quad if $i=2,\dots,13$}
\end{align*}
where the $\psibar\theta$ appears in the $i$th
position.  Similarly, we define $l_j$ as the root obtained from
\eqref{eq-definition-of-vector-s} by taking
$\sigma=\bigl(C-\frac{\wbar}{\theta\psi}L_j\bigr)\theta\w$,
$m=\theta\w$, $N=3$, and $\nu=\thetabar$
or~$\thetabar/2$ according to whether $P_1\in L_j$ or not.  Explicitly,
$l_j$ has the form
\begin{equation*}
  l_j=
    \begin{cases}
      \bigl(
      \phantom{4\w-8\wbar}\llap{$-9\wbar$},
      -3\w\hbox{ (three times)},
      -1\hbox{ (\rlap{nine}\phantom{eight} times)};
      \theta\w,2\wbar\bigr)
    &
    \hbox{if $P_1\in L_j$}
    \\
      \bigl(
      4\w-8\wbar,
      -3\w\hbox{ (\rlap{four}\phantom{three} times)},
      -1\hbox{ (eight times)};
      \theta\w,\thetabar\,\bigr)
    &
    \hbox{otherwise.}
    \end{cases}
  \end{equation*}
Here the $-9\wbar$ or $4\w-8\wbar$ appears in the first position, and
the $(-3\w)$'s resp.\ $(-1)$'s appear in positions indexed by points
of $P^2\F_3$ lying resp.\ not lying in $L_j$.  One verifies the next
lemma by direct computation.  

\begin{lemma}[The point- and line-roots in the Leech model]
  \label{lem-point-and-line-roots-in-Leech-model}
  These vectors $p_i$, $l_j$  and $\rho$
  in the Leech model have the same inner products with each other
  as do the vectors with the same names in the $P^2\F_3$ model.
  This defines an isometry of the two models, that identifies their
  point-roots, line-roots and Leech cusp~$\rho$.  Under this
  identification, in the Leech model we have
  \begin{align*}
    p_\infty&{}=(\psi C;\psi,-\wbar\psi)=
    (12\thetabar,\theta,\dots,\theta;\psi,-\wbar\psi)
    \\
    l_\infty&{}=(-\wbar\theta\psi C;-\wbar\theta\psi,6\wbar-\w)
    =
    (-36\wbar,3\wbar,\dots,3\wbar;-\wbar\theta\psi,6\wbar-\w)
  \end{align*}
  Writing these vectors in the form \eqref{eq-definition-of-vector-s}, their values of $\sigma$
  and $m$ are their entries before and immediately after the semicolon, their
  values of $N$ are both~$-3$, and their values of $\nu$ are
  $-\frac{13}{6}\theta$ and $-\frac{17}{2}\theta$, respectively.  \qed
\end{lemma}

\begin{remarks}
  Here is how we found this identification.  For each $i$ we
  have $\ip{\rho}{p_i}=\theta$ in the $P^2\F_3$ model, so their
  analogues in the Leech model should have the form $(\sigma_i;1,\ldots)$
  for suitable $\sigma_i\in\Lambda$.  The inner product formula
  \eqref{eq-general-inner-product-formula} shows that the differences
  between the $\sigma_i$ must be minimal vectors of $\Lambda$.  This
  suggested looking for the centroid of the $\sigma_i$ in
  $\Lambda\tensor\C$, expecting it to have stabilizer $L_3(3)$ in the
  affine isometry group $\Lambda:6\Suz$ of the Leech lattice.   We
  were already familiar with the principle used in
  lemma~\ref{lem-lattice-points-near-vectors-over-theta}:    the
  lattice vectors near an element of the rational span of a
  lattice are related to the short lattice vectors in a suitable congruence
  class.  This suggested looking for a congruence class in $\Lambda$
  with stabilizer $L_3(3)$ and exactly~$13$ minimal representatives.
  In Wilson's model of $\Lambda$ the $P_i$ form an $L_3(3)$-orbit of
  size~$13$.  So we investigated and found that they are congruent
  mod~$\psi$.  This suggested that the centroid should be $C$ from
  lemma~\ref{lem-the-center} and that the $\sigma_i$'s should be its
  nearest neighbors.  This determined the $p_i$ up to their values of
  $\nu$. For $p_1$ we chose $\nu=\frac{1}{2\theta}$ arbitrarily, and
  then used the orthogonality of the $p_i$ and
  \eqref{eq-general-inner-product-formula} to compute the $\nu$ values
  of the other $p_i$.  Since $\rho$ and the point-roots span $L$ up to
  finite index, the expressions for all the other vectors follow.

  As we mentioned after \eqref{eq-general-inner-product-formula},
  inner products with a vector $(\sigma;m,...)$ in the Leech model can
  be expressed in terms of $\sigma/m\in\Lambda\tensor\C$.  For
  $p_\infty$ and $l_\infty$ this point is $C$.  And
  lemma~\ref{lem-the-center} shows that for the point-roots
  (resp.\ line-roots), the corresponding points of $\Lambda\tensor\C$
  are $C$'s nearest neighbors in $\Lambda$
  (resp.\ $\frac{1}{\theta}\Lambda$).  This will be crucial in the
  proofs of lemma \ref{lem-step-4-point-case}
  and~\ref{lem-step-4-line-case}.  Finally, with additional work it is
  possible to introduce a version of the Leech model ``centered at
  $C$'', to avoid hiding some of the $L_3(3)$ symmetry.
\end{remarks}

\begin{lemma}[The triangle needed for step~$4$ in lemma~\ref{lem-change-of-basepoint}, point-mirror case]
\label{lem-step-4-point-case}
Let $x$ be the projection of $\rho$ onto a point-mirror, and $T$ be
the totally real triangle $\triangle\rho x p_\infty$.
Then $T$ meets that point-mirror in $\geodesic{x p_\infty}$, meets the
other point-mirrors in $p_\infty$ only, and misses all other mirrors.
\end{lemma}

\begin{proof}
  By $L_3(3)$ symmetry it suffices to treat the point-root $p_1$.
  Using the Leech model, and the formula for $p_1$ given just before
  lemma~\ref{lem-point-and-line-roots-in-Leech-model}, we find
  $x=(0;1/\theta,\discretionary{}{}{}-\wbar/\theta)
  = ( \theta^{-1} p_1 + \rho)$, which has $x^2=\ip{\rho}{x}=-1$.
  The other two vertices of $T$ are represented by $p_\infty$ from
  lemma~\ref{lem-point-and-line-roots-in-Leech-model} and
  $\rho=(0;0,1)$.  In calculations we will use $\thetabar\psibar
  p_\infty$ rather than~$p_\infty$, because it has negative inner
  product (namely $-39$) with both $\rho$ and $x$.  So we will be able
  to apply lemma~\ref{lem-how-real-triangle-meets-hyperplane}.

  First we claim that $T$ lies in the union of the second critical
  horoball around $\rho$ and the second critical ball around
  $p_\infty$.  To see this, note that $\height_\rho(x)=1$, so
  $x$ lies on the boundary of the
  first critical horoball.  (In fact this holds by construction.)  Also, one checks
  $\height_\rho(p_\infty)=13>3$, so $p_\infty$ lies outside the second
  critical horoball.  Therefore it suffices to find where the boundary
  of this horoball intersects $\geodesic{p_\infty\rho}$ and
  $\geodesic{p_\infty x}$, and check that both points lie in the
  second critical ball around~$p_\infty$.  Since
  $\geodesic{p_\infty\rho}$ travels directly toward the horoball, its
  intersection point is closer to $p_\infty$ than is the intersection
  point of $\geodesic{p_\infty x}$.  So it suffices to show that this
  second intersection point lies in the second critical horoball.  We
  found this intersection point by the method from the proof of
  lemma~\ref{lem-triangle-analysis-for-point-and-line-meridians-give-Leech-meridian},
  obtaining $\thetabar\psibar p_\infty+\bigl(9\sqrt{26}-39\bigr)x$.
  This can be verified by checking that it has height~$3$.  It has
  norm $-702$ and its inner product with $\thetabar\psibar p_\infty$
  is $1404-351\sqrt{26}$.  So the point of $\BB^{13}$ it represents
  lies at distance $\cosh^{-1}\sqrt{63-12\sqrt{26}}\approx.810$ from
  $p_\infty$.  This is less than $r_2=\sinh^{-1}(1)\approx.881$, so this
  point lies in the second critical ball, as desired.

  Therefore the only mirrors that might meet $T$ are the Leech mirrors
  (which include the batch~$0$ roots, namely the point-roots) and the
  batch~$1$ mirrors around $p_\infty$.  We used computer calculations
  in the $P^2\F_3$ model, and
  lemma~\ref{lem-how-real-triangle-meets-hyperplane}, to check that the
  batch~$1$ mirrors miss~$T$.  So it remains to examine how the mirror
  of a Leech root $s$ can meet $T$.  This part of the proof is similar
  to the proof of
  lemma~\ref{lem-triangle-analysis-for-2-Leech-meridians-give-a-third}:
  we will compute $\ip{T}{s}\sset\C$ and examine whether it contains the
  origin.  For use in the calculation we tabulate the parameters of
  the important vectors when they are written as in
  \eqref{eq-definition-of-vector-s}:
  $$
    \begin{matrix}
      &&\sigma&m&N&\nu
      \\
      x
      &
      \hbox{(a vertex of $T$)}
      &0&1/\theta&-1&-\theta/18
      \\
      \thetabar\psibar p_\infty
      &
      \hbox{(a vertex of $T$)}
      &-13\theta C&-13\theta
      &-117&-169\theta/2
      \\
      s
      &
      \hbox{(a Leech root)}
      &\sigma&1&3&\nu
    \end{matrix}
    $$
    The inner product formula \eqref{eq-general-inner-product-formula}
    and some simplification gives
    \begin{align*}
      \ip{x}{s}
      &{}=
      \textstyle
      \frac{\theta}{6}\sigma^2-      
      \frac{\theta}{3}\bigl(3\nu+\theta/2\bigr)
      \\
      \bigip{\thetabar\psibar p_\infty}{s}
      &{}=
      \textstyle
      \frac{13\theta}{2}\Bigl((C-\sigma)^2-\frac{36}{13}\Bigr)
      -13\theta\Bigl(\Im\ip{C}{\sigma}+3\bigl(\nu+\theta/6\bigr)\Bigr)
    \end{align*}
    The first terms are their imaginary parts.
    So $\ip{x}{s}$ lies on the real axis (if $\sigma=0$)
    or above it (otherwise).  Also, 
		lemma~\ref{lem-the-center} gives exactly the inequality needed to
		make
		$\ip{\thetabar\psibar
      p_\infty}{s}$ lie on or above the real axis, namely that $C-\sigma$
		has norm${}\geq\frac{36}{13}$. It is something of  a miracle that the geometry of the
    Leech lattice gives us exactly the bound that we need 
    for our argument to work.
		Finally,
    $\ip{\rho}{s}=\theta$ since $s$ is a Leech root.  So
    $\ip{T}{s}$ lies in the closed upper half plane.  If
    $\sigma\neq0$ then two of its vertices lie strictly above the real
    axis, so $\ip{T}{s}$ misses the origin unless
    the third vertex coincides with the origin, i.e., unless
    $\ip{\thetabar\psibar p_\infty}{s}=0$.  But then $s$ is a
    point-root, since these are the only roots orthogonal to
    $p_\infty$.  So suppose $\sigma=0$.  Then the above formulas
    simplify to the real numbers $\ip{x}{s}=-\theta(\nu+\theta/6)$ and
    $\ip{\thetabar\psibar
      p_\infty}{s}=-39\theta(\nu+\theta/6)$.
    (The latter uses $C^2=36/13$.)
    Since these differ by a
    positive factor, the only way $\ip{T}{s}$ can contain the origin
    is for both of them to vanish.  In particular
    $\ip{p_\infty}{s}=0$, so again $s$ is a point-root, indeed~$p_1$.
\end{proof}

The analogue of lemma~\ref{lem-step-4-point-case} for a line-mirror requires a much larger
tabulation of the mirrors near $p_\infty$ than we have needed so far.
The following lemma describes batches $0,\dots,5$ in a manner suitable
for machine computation.  Batches $0$, $1$ and~$2$ appear in
table~\ref{tab-roots-near-w-P}.  In batches $3$, $4$ and $5$ there are
$743\,418$, $107\,953\,560$ and $480\,961\,338$ mirrors.  To actually
construct the batches one should refer to the tabulation of the line
code in tables 2 and~3 of \cite{Allcock-Y555}.

\begin{lemma}[Mirrors near a $13$-point]
\label{lem-mirrors-near-a-13-point}
%
Write vectors of $L$ in the $P^2 \F_3$ model.
Then the mirrors
in batches $b=0,\dots,5$, i.e.,
those at distance${}<\sinh^{-1}2\approx1.444$ from $p_\infty$ 
, are 
the orthogonal complements of the roots $s=(s_0;s_1,\dots,s_{13})$ 
described in the next paragraph.  For $b=0$ this gives
all roots, while for the other cases it gives one from each
scalar class.

Define the ``desired norm'' $N=3$, $4$, $6$, $7$, $10$ or $12$ and the
``required coordinate sum'' $S=0$, $1$, $0$, $1$, $1$, or $0\in\F_3$
according to the value of $b=0,\dots,5$.  Choose any codeword
$w=(w_1,\dots,w_{13})$ in the line code whose weight is at most $N$
and whose coordinate sum is $S$.  Choose $e_1,\dots,e_{13}$ with
$e_i\in\{0,3,6\}$ if $w_i\neq0$ and $e_i\in\{0,3,9,12\}$ if $w_i=0$,
such that the sum of the $e_i$ equals $N-\hbox{\rm weight}(w)$.
Choose any Eisenstein integers $s_1,\dots,s_{13}$ such that $s_i$ mod
$\theta$ is $w_i$, and $|s_i|^2$ is either $e_i$ or $e_i+1$ according to
whether $w_i=0$ or $w_i\neq0$.  Prefix the coordinates
$(s_1,\dots,s_{13})$ by a coordinate $s_0=0$, $1$, $\theta$, $-2$,
$2-(\w\hbox{ \rm or }\wbar)$, or $3$ according to value of
$b=0,\dots,5$.
\end{lemma}

\begin{proof}
Suppose $s=(s_0;s_1,\dots,s_{13})$ is a root in one of the batches
$0,\dots,5$.  Then $\bigl|\ip{s}{p_\infty}\bigr|^2<3\cdot12$, which
boils down to the condition $|s_0|^2<12$, i.e.,
$|s_0|^2\in\{0,1,3,4,7,9\}$. If $|s_0|^2=0$ then of course $s_0=0$.
In the other cases there is a unique way to scale $s$ by a unit, such
that $s_0$ is as described at the end of the lemma.  In all cases,
$s^2=3$ says that the vector $(s_1,\dots,s_{13})$ has norm~$N$, and the
definition of $L$ requires that its reduction $w$ modulo~$\theta$ is
in the line code and has coordinate sum~$S$.  For $i=1,\dots,13$ we take $e_i$ to be
defined as $|s_i|^2$ if $w_i=0$, or $|s_i|^2-1$ if $w_i\neq0$.  If
$w_i=0$ then $s_i$ must be divisible by $\theta$, and 
$|s_i|^2\leq12$, so $e_i\in\{0,3,9,12\}$.  If $w_i\neq0$ then $N\leq12$
gives $|s_i|^2\leq9$ since $w_i$ has weight at least~$4$ (so there are
at least~$3$ other nonzero coordinates).  An element of $\E$ not
divisible by $\theta$, and having norm${}\leq9$, must have norm~$1$,
$4$ or $7$.  That is, when $w_i\neq0$ we have proven $e_i\in\{0,3,6\}$.
We have just
established the lemma's constraints on the $e_i$. The
constraints on the $s_i$ in terms of the $e_i$ are satisfied by the
construction of the $e_i$.
Conversely, if one follows the
instructions in choosing $w,e_1,\dots,e_{13},s_1,\dots,s_{13},s_0$
then one obtains a norm~$3$ vector of $L$ in the specified batch.
\end{proof}

\begin{lemma}[The triangle needed for step~$4$ in lemma~\ref{lem-change-of-basepoint}, line-mirror case]
\label{lem-step-4-line-case}
Let $x$ be the projection of $\rho$ onto a line-mirror, and $T$ be the
totally real triangle $\triangle\rho x l_\infty$.  Then
$T$ meets that line-mirror in $\geodesic{x l_\infty}$,
meets the other line-mirrors in $l_\infty$ only, and misses all other
mirrors.
\end{lemma}

\begin{proof}
  This is similar to
  lemma~\ref{lem-step-4-point-case} but there are some
  new issues.  By symmetry we may take the line root to be~$l_1$.
  Computation gives $x=\rho-\wbar l_1$, with $x^2=\ip{\rho}{x}=-3$.
  In computations we use $\w\psibar l_\infty$ in place of $l_\infty$
  because it has negative inner product with $\rho$ and $x$,
  namely~$-39$.  This is also its norm.

  We claim that $T$ lies in the union of the sixth critical ball
  around $l_\infty$ and the third critical horoball around~$\rho$.
  Following the proof of lemma~\ref{lem-step-4-point-case}, it is enough to check
  that the point where $\geodesic{l_\infty x}$ pierces the boundary of
  the $3$rd critical horoball lies in the $6$th critical ball.   One can check
  that this point is represented by the vector $\w\psibar
  l_\infty+\bigl(4\sqrt{39}-13\bigr)x$, by computing that its height
  is~$4$.  Its norm is $-1404$ and its inner product with $\w\psibar
  l_\infty$ is $468-156\sqrt{39}$.  This lets us compute the distance
  from this point to $l_\infty$, namely
  $\cosh^{-1}\sqrt{\frac{64}{3}-\frac{8}{3}\sqrt{39}}\approx1.407$.
  This is less than $r_6=\sinh^{-1}(2)\approx1.444$, as desired.

  So the only mirrors that could meet $T$ are the first and second
  shell mirrors around $\rho$ (this includes the line mirrors, which
  are the batch~$0$ roots around~$l_\infty$), and
  the mirrors in batches~$1,\dots,5$ around $l_\infty$.  The latter
  checks are done in the $P^2\F_3$ model.  We applied an isometry $F$
  of $L$ that exchanges the point- and line-roots up to signs, namely
  $p_i\mapsto l_{14-i}$, $l_j\mapsto-p_{14-j}$.  (It follows that
  $F(p_\infty)=l_\infty$ and $F(l_\infty)=-p_\infty$, which is useful
  for writing down a matrix for $F$.)  So it suffices to check that
  the triangle with vertices $F(\rho)$, $F(x)$ and $F(\w\psibar
  l_\infty)$ misses all batch~$1,\dots,5$ mirrors around~$p_\infty$.
  For this we enumerated these mirrors in the $P^2\F_3$ model by using
  lemma~\ref{lem-mirrors-near-a-13-point}, and checked that they all
  miss this triangle by using
  lemma~\ref{lem-how-real-triangle-meets-hyperplane}.  This was the
  only computer calculation in the paper that took more than a
  moment---sixteen hours on a laptop.

  To examine how the Leech mirrors and second-shell mirrors meet $T$,
  we return to the Leech model.
  We begin by writing down the parameters $\sigma$, $m$, $N$ and
  $\nu$ when we write the following important vectors in the form
  \eqref{eq-definition-of-vector-s}:
  $$
    \begin{matrix}
      &&\sigma&m&N&\nu
      \\
      x&
      \hbox{(a vertex of $T$)}&
      -\bigl(C-\frac{\wbar}{\theta\psi}L_1)\theta&-\theta&-3&-\theta
      \\
      \w\psibar l_\infty&
      \hbox{(a vertex of $T$)}&
      -13\theta C&-13\theta
      &-39&-\frac{221}{2}\theta
      \\
      s&
      \hbox{(a Leech root)}&
      \sigma&1&3&\nu
      \\
      \llap{or }
      s&
      \hbox{(a second-shell root)}&\sigma&\theta&3&\nu
    \end{matrix}
    $$ Again we will use lemma~\ref{lem-how-real-triangle-meets-hyperplane} to determine whether $s^\perp$
    meets $T$.  We begin with the case of a second-shell root because
    it is simpler.   The real part of the inner product
    formula \eqref{eq-general-inner-product-formula} is
    \begin{align}
      \label{eq-Re-ip-x-with-s}
      \Re\ip{x}{s}
      &{}=
      \textstyle
      \frac{3}{2}\Bigl(
      \frac{\sigma}{\theta}
      -\bigl(C-\frac{\wbar}{\theta \psi}L_1)
      \Bigr)^2
      \\
      \label{eq-Re-ip-w-psibar-l-infty-with-s}
      \Re\bigip{\w\psibar l_\infty}{s}
      &{}=
      \textstyle
      \frac{39}{2}\Bigl(
      \bigl(\frac{\sigma}{\theta}-C\bigr)^2
      -\frac{12}{13}
      \Bigr)
    \end{align}
    We see that \eqref{eq-Re-ip-x-with-s} is at least~$0$, with
    equality if and only if
    $\frac{\sigma}{\theta}=C-\frac{\wbar}{\theta\psi}L_1$.  And
    \eqref{eq-Re-ip-w-psibar-l-infty-with-s} is at least~$0$ by
    lemma~\ref{lem-the-center}.  Together with $\Re\ip{\rho}{s}=3$, we
    see that $\ip{T}{s}$ lies in the closed
    right half plane.  If $\frac{\sigma}{\theta}\neq
    C-\frac{\wbar}{\theta\psi}L_1$ then 
    \eqref{eq-Re-ip-x-with-s} is strictly positive, so the only way
    $\ip{T}{s}$ could contain the origin is for its vertex
    $\ip{\w\psibar l_\infty}{s}$ to be the origin.  That is, $s\perp
    l_\infty$, so $s$ is a line root.
    
      So now suppose
    $\frac{\sigma}{\theta}=C-\frac{\wbar}{\theta\psi}L_1$.
    This implies that 
   $(s - \bar{\omega} l_1)$ is a multiple of $\rho$ 
   (just write $s$ and $l_1$ in the Leech coordinate system).
   Since $s$ and $l_1$ both have norm $3$, it quickly follows that
   $s = \bar{\omega} l_1 + n \theta \rho$  for some integer n.
   Using the known inner products between points, lines,
   $l_{\infty}$ and $\rho$, we find that
    \begin{equation*}
      \bigip{\w\psibar l_\infty}{s}=39 n \theta
      \quad\hbox{and}\quad
      \ip{x}{s}= 3 n \theta. 
    \end{equation*}
    Since these differ by a
    positive factor, the only way $0$ could lie in $\ip{T}{s}$ is for
    both of them  to be~$0$.  Then $s$ is a line root,
    indeed $l_1$.

    Now suppose $s$ is a Leech root.  The imaginary part of
    \eqref{eq-general-inner-product-formula} is
    \begin{align}
      \label{eq-Im-ip-x-with-s}
      \Im\ip{x}{s}
      &{}=\textstyle
      \frac{\theta}{2}\Bigl[
        \Bigl(\sigma-
        \bigl( C-\frac{\wbar}{\theta\psi}L_1\bigr
        )\Bigr)^2-2\Bigr]
      \\
      \label{eq-Im-ip-w-psibar-l-infty-with-s}
      \Im\bigip{\w\psibar l_\infty}{s}
      &{}=\textstyle
        \frac{13\theta}{2}\Bigl(\bigl(C-\sigma\bigr)^2-\frac{38}{13}\Bigr)
    \end{align}
    We have $\ip{\rho}{s}=\theta$, which lies above the real axis.  We
    claim that \eqref{eq-Im-ip-x-with-s} does also.
    To see this, recall from remark~\ref{rem-data-about-P-and-L} that
    $C-\frac{\wbar}{\theta\psi}L_1\in\frac{1}{\theta}\Lambda$ has
    norm~$3$.  So it equals a norm~$9$ vector of $\Lambda$, divided by
    $\theta$.  By lemma~\ref{lem-lattice-points-near-vectors-over-theta}, the distance from $C-\frac{\wbar}{\theta\psi}L_1$ to
    $\Lambda$ is $\sqrt3$.   So the first term in the brackets in
    \eqref{eq-Im-ip-x-with-s} is at least~$3$.  So
    \eqref{eq-Im-ip-x-with-s} lies above the real axis.  Next,
    lemma~\ref{lem-the-center} says that either
    $(C-\sigma)^2\geq\frac{42}{13}$ or else
    $\sigma=C+\frac{1}{\psi}P_i$ for some $i=1,\dots,13$.  In the
    first case we see that \eqref{eq-Im-ip-w-psibar-l-infty-with-s}
    lies above the real axis, so $T$ does too, so it
    cannot contain the origin.

    So suppose $\sigma=C+\frac{1}{\psi}P_i$.  
    This implies that 
    $(s - p_i)$ is a multiple of $\rho$.  
   Since $s$ and $p_1$ both have norm $3$, it follows
    that  $s = p_i + n \rho$ for some integer $n$.
  Using the known inner products between points, lines,
   $l_{\infty}$ and $\rho$, we find that
    \begin{align}
      \label{eq-ip-x-with-s}
      \ip{x}{s}
      &{}=
      \begin{cases}
        \theta - 3 n + \bar{\omega} \theta
        &\hbox{if $P_i\in L_1$}
        \\
        \theta-3n
        &\hbox{otherwise}.
      \end{cases}
      \\
      \label{eq-ip-w-psibar-l-infty-with-s}
      \bigip{\w\psibar l_\infty}{s}
      &{}=
      \thetabar+6-39n
    \end{align}
    (Note that \eqref{eq-ip-x-with-s} is above the real axis and
    \eqref{eq-ip-w-psibar-l-infty-with-s} is below it.)  
    We want to determine whether the origin lies in the
    triangle $\ip{T}{s}$ that has \eqref{eq-ip-w-psibar-l-infty-with-s}, \eqref{eq-ip-x-with-s} and $\ip{\rho}{s}=\theta$ for
    vertices.  We can find the intersection of $\ip{T}{s}$ with $\R$ by writing
    $A$, resp.\ $B$, for the convex combination of $\bigip{\w\psibar
      l_\infty}{s}$ and $\ip{\rho}{s}$, resp.\ $\ip{x}{s}$, that has
    no imaginary part.  Then the origin lies in $\ip{T}{s}$ just if it
    lies in the interval with endpoints $A$ and~$B$.  One works out
    $A$ and $B$, with the result
    $$
    \textstyle
    A=3-\frac{39}{2}n
    \qquad\qquad
    B=
    \begin{cases}
      3-15n&\hbox{if $P_i\in L_1$}\\
      3-21n&\hbox{otherwise.}
    \end{cases}
    $$
    If $n\leq0$ then $A$ and $B$ are  both positive. 
    And if $n\geq1$ then $A$ and $B$ are both negative.  So the origin
    does not lie in $\ip{T}{s}$.
\end{proof}

\section{How two complex triangles meet the mirrors}
\label{app-how-complex-triangles-meet-mirrors}

\noindent
In this appendix we prove lemma~\ref{lem-complex-triangles-miss-mirrors-except-as-known}: the complex triangles $\triangle\rho\tau
p_\infty$ and $\triangle\rho\tau l_\infty$ miss the mirror arrangement except
at $p_\infty$ and $l_\infty$ respectively.  This is the key fact in step~2 in
the proof of lemma~\ref{lem-change-of-basepoint}.  Because these
triangles are complex instead of totally
real, we cannot use the machinery in appendix~\ref{app-how-real-triangles-meet-mirrors}.  Instead we
exploit the fact that they lie in
$\BB(F)\iso\BB^1$, where $F$ is
the $L_3(3)$-invariant
sublattice of $L$.  Our first step is
to understand how this $\BB^1$ meets the mirrors:

\begin{lemma}
  \label{lem-how-the-1-ball-meets-the-mirrors}
  Suppose $x\in\BB^1\cap\H$.  Then either $x$ is represented by a
  norm~$-3$ vector of $F$, or is orthogonal to a norm~$3$ vector of~$F$.
\end{lemma}

\begin{proof}

Write $M$ for the sublattice of $L$ spanned by the roots orthogonal
  to $x$. The lattice $M$ is non-empty because $x \in \H$. So $1\leq\dim M\leq13$.
  As an $\E$-lattice with all  inner products divisible by~$\theta$,
  and spanned by roots, $M$ is a direct sum of copies of the
  Eisenstein root lattices $A_2^\E$, $D_4^\E$, $E_6^\E$ and $E_8^\E$. 
  (See \cite[Thm.~3]{Allcock-Y555}.) 
 There are at most $13$ direct summands of $M$ and $L_3(3)$ acts
 on the set of these.
  Since $L_3(3)$ is simple and contains an element of order~$13$, the smallest
  nontrivial permutation representations of $L_3(3)$ are on~$13$ objects.
Therefore either
  (a) $L_3(3)$ preserves each direct summand of $M$
  or
  (b) $M$ is the sum of $13$ copies of
  $A_2^\E$.


First we treat case (a). 
The action of $L_3(3)$ on $F^{\bot}$ is
known and irreducible: it is the deleted permutation
representation coming from $L_3(3)$'s action on the points of
$P^2\F_3$.
So as an $L_3(3)$ representation, $x^\perp$ decomposes into two irreducible factors:
$F^\perp$ and a one dimensional representation.
   We are assuming that the underlying vector space of each direct summand of $M$
 is preserved by $L_3(3)$ and so must contain one of these irreducible factors.
 Since each summand of $M$ have dimension at most $4$, it follows that $M$ must
 isomorphic to $A_2^\E$ and lie in the fixed space of $L_3(3)$.
  So $F$ contains a root orthogonal to $x$.
  
  Now we treat case (b).  As the orthogonal complement of $13$
  mutually orthogonal roots, the line in $L\tensor\C$ corresponding to
  $x$ is represented by a lattice vector; we choose a primitive one
  and use the same name $x$ for it.  The product of the
  $\wbar$-reflections in the $13$ roots, times the scalar $\w$, acts
  on $L$ by scaling $x$ by $\w$ and fixing $x^\perp$ pointwise.  So it
  is given by the same formula \eqref{eq-formula-for-omega-reflection} as a reflection, namely $v\mapsto
  v-(1-\w)\ip{v}{x}x/x^2$.  (This differs from a reflection in that
  its fixed set in $\BB^{13}$ is a point not a hyperplane.)  Since
  this preserves $L$, we have $(1-\w)\ip{v}{x}/x^2\in\E$ for
  every $v\in L$.  Since $L=\theta\cdot L^*$, there exists $v\in L$
  with $\ip{v}{x}=1-\wbar$.  So $3/x^2\in\E$, which forces $x^2=-3$.
\end{proof}

It is easy to see that $F$ is spanned by $\rho=(-\psibar;-1,\dots,-1)$ and
$p_\infty=(\thetabar;0,\dots,0)$.  These have norms $0$ and~$-3$, and
$\ip{\rho}{p_\infty}=\theta\psibar$.  We will also use a second null
vector $\rho'$, which is defined as $-\w$ times the image of $\rho$
under the $\w$-``reflection'' in $p_\infty$.  (As in the proof above,
this fixes a single point of $\BB^{13}$ rather than a hyperplane.)
One can check $\rho'=\psibar
p_\infty-\w\rho=(\wbar\psibar;\w,\dots,\w)$ and
$\ip{\rho}{\rho'}=13\theta$.

Everything becomes easier if we work in a certain superlattice $E$ of
$F$, namely the one spanned by $\rho/\psibar$ and $\rho'/\psibar$.  We
use the notation $[u,v]$ to mean $(u\rho'+v\rho)/\psibar$.
Obviously $E$ contains $\rho$, and it contains $p_\infty$ by our
formula for $\rho'$ in terms of $p_\infty$ and $\rho$.  So $E$ does
indeed contain~$F$.
The main advantage of working in $E$ is that the inner product has the
simple form
$$
\Bigip{[u,v]}{[u',v']}
=
\begin{pmatrix}
  u&v
\end{pmatrix}
\begin{pmatrix}
  0&\thetabar\\
  \theta&0
\end{pmatrix}
\begin{pmatrix}
  \ubar'\\
  \vbar'
\end{pmatrix}
$$
which is the same as in \eqref{eq-inner-product-in-Leech-model}.  One can
check that $p_\infty=[1,\w]$ and
$l_\infty=[\thetabar\wbar,\w-2]$.
Now we can prove the main result of this appendix:

\begin{lemma}[Two triangles needed in step~2 of lemma~\ref{lem-change-of-basepoint}]
  \label{lem-complex-triangles-miss-mirrors-except-as-known}
  The only point of $\H$ in  $\triangle\rho \tau p_\infty$
  (resp.\ $\triangle\rho\tau l_\infty$) is 
  $p_\infty$ (resp.\ $l_\infty$). 
\end{lemma}

\begin{proof}
  Because $\tau$ is the midpoint of $\geodesic{p_\infty l_\infty}$,
  the union of these two triangles is the larger triangle $T=\triangle\rho p_\infty
  l_\infty$.  So it suffices to show that this triangle misses $\H$ except at
  $p_\infty$ and $l_\infty$.  We identify the projective space of
  $F\tensor\C$ with $\C\cup\{\infty\}$ by plotting a vector $[u,v]$ as
  $v/u$.  Then $\BB^1$ corresponds to the upper half plane, and
  $\rho$, $p_\infty$ and $l_\infty$ correspond to $\infty$ , $\w$ and
  $\frac{3}{2}+\frac{\theta}{6}$ respectively.  So  $T$ is the
  hyperbolic triangle they span.
  The edges of $T$ are
the vertical lines through $\omega$, $(\tfrac{3}{2} + \tfrac{\theta}{6})$ and an arc $C$ of the circle 
$\lbrace z \colon \abs{z - \tfrac{1}{3}}^2 < \tfrac{13}{9} \rbrace$.
The image of $\tau$ in the upper half plane is $(1 + i)$, which is the midpoint of the circular arc $C$.

  Our strategy is to find all the norm~$-3$ vectors of $F$
  representing points of~$T$, and all norm~$3$ vectors orthogonal to
  points of~$T$, and then apply lemma~\ref{lem-how-the-1-ball-meets-the-mirrors}.  It is convenient to
  find all such vectors in $E$ first, and then discard the ones that
  lie outside~$F$.  We begin by writing an arbitrary norm~$-3$ vector
  of $E$ as $x=\bigl[m,\frac{\theta}{\bar
      m}\bigl(\frac{0-(-3)}{6}+\nu\bigr)\bigr]$, where $m\in\E-\{0\}$,
  and $\nu\in\Im\C$ is chosen so that the second component lies
  in~$\E$.  (This is just like the analysis leading to \eqref{eq-definition-of-vector-s}.)  The
  corresponding point of $\BB^1$ has imaginary part $\theta/2|m|^2$.
  Now suppose $x\in T$.  The only point of $T$ with imaginary part
  $\theta/6$ is $l_\infty$, and all its other points have larger
  imaginary part.  Therefore either $x$ is a multiple of $l_\infty$,
  or else $|m|=1$.  In the latter case we scale $x$ so that $m=1$, so
  $x = [1,\theta(\frac12+\nu)]$.  Since the second component lies
  in~$\E$, we have $\nu=\frac{1}{2\theta}+\frac{n}{\theta}$ for some
  $n\in\Z$.  Then the point of $\BB^1$ represented by $x$ is
  $n-\wbar$.  Since every point of $T$ has real part at least
  $-\frac12$ and at most $\frac32$, the only possibilities for $n$ are
  $-1$, $0$ and~$1$.  The case $n=-1$ yields $x=p_\infty$.  The case
  $n=0$ does not actually arise, because $-\wbar$ lies below the
  geodesic $\geodesic{p_\infty l_\infty}$.  The case $n=1$ yields a
  point of $T$.  We have shown that the only points of $T$ represented
  by norm~$-3$ vectors of $E$ are $p_\infty$, $l_\infty$ and
  $[1,1-\wbar]$.

  Now consider a point of $T$ orthogonal to a norm~$3$ vector of $E$, say
  $[u,v]$.  The orthogonal complement of $[u,v]$ is spanned by
  $[\ubar,\vbar]$, which is a norm~$-3$ vector of $E$.
  So we have
  shown that the only points of $T$ that are represented by a
  norm~$-3$ vectors of $E$, or are orthogonal to a norm~$3$ vector
  of~$E$, are $p_\infty$, $l_\infty$ and $[1,1-\wbar]$.
  
  Finally, consider a point of $T\cap\H$ other than $p_\infty$ and
  $l_\infty$.  By lemma~\ref{lem-how-the-1-ball-meets-the-mirrors},
  either it is represented by a norm~$-3$ vector of $F$, or is
  orthogonal to a norm~$3$ vector of $F$.  This norm~$\pm3$ vector
  lies in $E$, so the previous two paragraphs show that the point is
  $[1,1-\wbar]$.  So it suffices to show that $[1,1-\wbar]\notin\H$.  For
  this we observe that $F$ contains neither the norm~$-3$ vector
  $[1,1-\wbar]$, nor the norm~$3$ vector $[1,1-\w]$ orthogonal to it.
\end{proof}


\begin{thebibliography}{8}

\bibitem[A1]{Allcock-Inventiones} D. Allcock, The Leech lattice and
  complex hyperbolic reflections, {\it Invent. Math.} {\bf 140} (2000)
  283--301.

\bibitem[A2]{Allcock-monstrous} D. Allcock, A monstrous proposal, in
  {\it Groups and Symmetries, From neolithic Scots to John McKay},
  ed. J. Harnad.  AMS and CRM, (2009). arXiv:math/0606043.

\bibitem[A3]{Allcock-Y555}
D. Allcock,
{\it On the Y555 complex reflection group,} 
J. Alg. {\bf 322}, no. 5 (2009) 1454-1465.

\bibitem[AB]{AB-braidlike}
D. Allcock and T. Basak,
Geometric generators for braid like groups,
{\it Geom. Topol.} {\bf 20} (2016), no. 2, 747--778.
arXiv:1403:2401 

\bibitem[ACT1]{ACT-surfaces} D. Allcock, J. Carlson and D. Toledo, The
  complex hyperbolic geometry of the moduli space of cubic surfaces,
  {\it J. Algebraic Geom.} {\bf11} (2002) 659--724.

\bibitem[ACT2]{ACT-threefolds} D. Allcock, J. Carlson and D. Toledo, The
  moduli space of cubic threefolds as a ball quotient, {\it
    Mem. Amer. Math. Soc.} {\bf209} (2011). ISBN 978-0-8218-4751-0.

\bibitem[Ban]{Bannai}
E. Bannai, Fundamental groups of the spaces of regular orbits of the
finite unitary reflection groups of dimension 2, {\it
  J. Math. Soc. Japan} {\bf 28} (1976), Number 3, 447--454.

\bibitem[Ba1]{Basak-bimonster-1}
T. Basak,
The complex Lorentzian Leech lattice and the bimonster,
{\it J. Alg.} {\bf 309} (2007) 32--56. 

\bibitem[Ba2]{Basak-Coxeter-diagrams} T.\ Basak, On Coxeter
  diagrams of complex reflection groups, {\it Trans. Amer. Math. Soc.}
  {\bf 364}
  (2012), no. 9, 4909--4936.

\bibitem[Ba3]{Basak-bimonster-2}
T. Basak,
The complex Lorentzian Leech lattice and the bimonster II, 
{\it Trans. Amer. Math. Soc.} {\bf 368} (2016), no. 6, 4171--4195.
arXiv:0811.0062.

\bibitem[Be]{Bessis} D. Bessis,
Finite complex reflection arrangements are $K(\pi,1)$.
{\it Ann. of Math. (2)} {\bf 181} (2015), no. 3, 809--904.

\bibitem[BH]{Bridson-Haefliger} M. Bridson and A. Haefliger, {\it
  Metric spaces of non-positive curvature}, Grundlehren der
  Mathematischen Wissenschaften, 319. Springer-Verlag, Berlin,
  1999. 

\bibitem[Br1]{Brieskorn} E. Brieskorn, Die Fundamentalgruppe des
Raumes der regul\"aren Orbits einer endlichen komplexen
Spiegelungsgruppe, {\it Invent. Math.} {\bf12} (1971) 57--61.

\bibitem[Br2]{Brieskorn-question} E. Brieskorn, Vue d’ensemble sur les
  probl\`emes de monodromie, in: Singularit\'és \`̀a Carg\`ese,
  Rencontre sur les Singularit\'és en G\'éom\'étrie Analytique,
  Inst. \'Études sci. de Carg\`èse, 1972, Asterisque, Nos. 7 et 8,
  Soc. Math. France, Paris, 1973, pp. 393–413.
  
\bibitem[ATLAS]{ATLAS} J. H. Conway et. al., {\it Atlas of finite
  groups}. Oxford University Press, Eynsham, 1985. 
  
\bibitem[CSi]{Conway-Simons}
J. H. Conway and C. S. Simons,
{\it 26 Implies the Bimonster,}
J. Alg. {\bf 235}, (2001) 805-814.

\bibitem[CSl]{Conway-Sloane}
J. H. Conway, and N. J. A. Sloane, 
{\it Sphere Packings, Lattices and Groups 3rd Ed.} 
Springer-Verlag (1998)

\bibitem[FN]{Fox-Neuwirth}
R. Fox and L. Neuwirth, The braid groups, {\it Math. Scand.} {\bf10} (1962) 119--126.

\bibitem[Go]{Goldman} W. Goldman, {\it Complex hyperbolic geometry}
  Oxford mathematical monographs (1999).


\bibitem[La]{Laza}
  R. Laza,  Deformations of singularities and variation of GIT
  quotients. {\it Trans. Amer. Math. Soc.} {\bf 361} (2009), no. 4, 2109--2161. 

\bibitem[Li]{Libgober} Libgober
A. Libgober, On the fundamental group of the space of cubic
surfaces. {\it Math. Z.} {\bf 162} (1978), no. 1, 63--67.
  
\bibitem[L\"o]{Lonne}
  L\"onne, Michael Fundamental group of discriminant complements of
  Brieskorn-Pham polynomials. {\it C. R. Math. Acad. Sci. Paris} {\bf 345} (2007),
  no. 2, 93--96.

\bibitem[L1]{Looijenga-triangle-singularities}
E. Looijenga, The
  smoothing components of a triangle singularity. II. {\it Math. Ann.}
  {\bf 269}
  (1984), no. 3, 357--387.

\bibitem[L2]{Looijenga-compactifications-II} E. Looijenga,
  Compactifications defined by arrangements. II. Locally symmetric
  varieties of type IV. {\it Duke Math. J.} {\bf 119} (2003), no. 3, 527--588.
  
\bibitem[L3]{Looijenga-Artin-groups}
E. Looijenga, Artin groups and the fundamental groups of some moduli
spaces. {\it J. Topol.} {\bf1} (2008) 187--216.

\bibitem[Mi]{Michel} M. Brou\'e, {\it Introduction to complex
  reflection groups and their braid groups}. Lecture Notes in
  Mathematics, 1988. Springer-Verlag, Berlin, 2010.

\bibitem[vdL]{van-der-Lek} H. van der Lek, {\it The homotopy type of
  complex hyperplane complements}, thesis, Katholieke Universiteit te
  Nijmegen, 1983.

\bibitem[W]{Wilson} 
R. Wilson,
{\it The complex Leech Lattice and maximal subgroups of the Suzuki group}.  
J. Algebra  84  (1983), 151-188. 
\end{thebibliography}
\end{document}